\documentclass[a4paper,reqno]{amsart}
\usepackage{longtable}
\usepackage{stmaryrd,amssymb,mathtools,booktabs,upref}
\usepackage[inline,shortlabels]{enumitem}
\usepackage[hidelinks]{hyperref}
\usepackage{color}
\usepackage{float}
\usepackage{todonotes}
\allowdisplaybreaks

\newtheorem{innercustomthm}{Theorem}
\newenvironment{customthm}[1]
  {\renewcommand\theinnercustomthm{#1}\innercustomthm}
  {\endinnercustomthm}

\newtheorem{theorem}{Theorem}[section]
\newtheorem{proposition}[theorem]{Proposition}
\newtheorem{lemma}[theorem]{Lemma}
\newtheorem{corollary}[theorem]{Corollary}

\newtheorem*{convention}{Convention}
\newtheorem*{question}{Question}

\theoremstyle{definition}
\newtheorem{definition}[theorem]{Definition}
\newtheorem{example}[theorem]{Example}
\theoremstyle{remark}
\newtheorem{remark}[theorem]{Remark}

\numberwithin{equation}{section}

\newcommand{\cD}{\mathcal{D}}
\newcommand{\cF}{\mathcal{F}}
\newcommand{\cA}{\mathcal{A}}
\newcommand{\cT}{\mathcal{T}}

\newcommand{\cK}{\mathcal{K}}
\newcommand{\cV}{\mathcal{V}}

\newcommand{\cU}{\mathcal{U}}

\newcommand{\Lie}[1]{\mathrm{#1}}
\newcommand{\G}{\Lie{G}}
\newcommand{\GL}{\Lie{GL}}

\newcommand{\bH}{\mathbb{H}}
\newcommand{\bC}{\mathbb{C}}
\newcommand{\bR}{\mathbb{R}}

\newcommand{\bN}{\mathbb{N}}

\newcommand{\lie}[1]{\mathfrak{#1}}
\newcommand{\mfa}{\lie{a}}

\newcommand{\mfg}{\lie{g}}
\newcommand{\mfh}{\lie{h}}
\newcommand{\mfk}{\lie{k}}
\newcommand{\mfn}{\lie{n}}

\newcommand{\mfs}{\lie{s}}

\newcommand{\mfu}{\lie{u}}

\newcommand{\spa}[1]{\mathrm{span}(#1)}

\DeclareMathOperator{\im}{Im}
\DeclareMathOperator{\Gr}{Gr}
\DeclareMathOperator{\End}{End}
\DeclareMathOperator{\Hom}{Hom}
\DeclareMathOperator{\ad}{ad}
\DeclareMathOperator{\diag}{diag}
\DeclareMathOperator{\id}{id}

\DeclareMathOperator{\tr}{tr}

\DeclarePairedDelimiterX{\inp}[2]{\langle}{\rangle}{#1,#2}

\setlist{nosep}

\begin{document}

\title{Torsion-free $H$-structures on almost Abelian solvmanifolds}

\author{Marco Freibert}

\address{Fachbereich Mathematik\\
Universit\"at Hamburg\\
Bundesstra\ss e 55 \\
D-20146 Hamburg \\
Germany}

\email{marco.freibert@uni-hamburg.de}

\thanks{}

\begin{abstract}
In this article, we provide a general set-up for arbitrary linear Lie groups $H\leq \GL(n,\bR)$ which allows to characterise the almost Abelian Lie algebras admitting a torsion-free $H$-structure. In more concrete terms, using that an $n$-dimensional almost Abelian Lie algebra $\mfg=\mfg_f$ is fully determined by an endomorphism $f$ of $\bR^{n-1}$, we give a description of the subspace $\cF_{\mfh}$ of all $f\in\End(\bR^{n-1})$ for which $\mfg_f$ admits a ``special'' torsion-free $H$-structure in terms of the image of a certain linear map. For large classes of linear Lie groups $H$, we are able to explicitly compute $\cF_{\mfh}$ and so give characterisations of the almost Abelian Lie algebras admitting a torsion-free $H$-structure. 

Our results reprove all the known characterisations of the almost Abelian Lie algebras admitting a torsion-free $H$-structure for different single linear Lie groups $H$ and extends them to big classes of linear Lie groups $H$. For example, we are able to provide characterisations in the case $n=2m$, $H\leq \GL(m,\bC)$ and $H$ either being a complex Lie group or being totally real, or in the case that $H$ preserves a pseudo-Riemannian metric. In many cases, we show that the space $\cF_{\mfh}$ coincides with what we call the \emph{characteristic subalgebra} $\tilde{\mfk}_{\mfh}$ associated to $\mfh$, and that then the torsion-free condition is equivalent to the left-invariant flatness condition. In particular, we prove this to be the case if $H$ is a complex linear Lie group or if $\mfh$ does not contain any elements of rank one or two and is either metric or totally real.
\end{abstract}
\maketitle
\section*{Introduction}
\emph{Almost Abelian} Lie groups, i.e. real Lie groups having a codimension one normal subgroup, and \emph{almost Abelian solvmanifolds} $\Gamma\backslash G$, i.e. quotients of a $1$-connected almost Abelian Lie group $G$ by a cocompact lattice $\Gamma$, form distinguished and very accessible classes of Lie groups or compact manifolds, respectively, and cover, in particular, in low dimensions, many well-known examples of Lie groups or compact manifolds, respectively. In three dimensions, all Lie groups $G$ but those having associated Lie algebra isomorphic to either $\mathfrak{so}(3)$ or $\mathfrak{so}(2,1)$, and so, in particular, all solvable ones, are almost Abelian. Thus, not surprisingly, five of the eight Thurston geometries can be modeled as left-invariant metrics on a three-dimensional almost Abelian Lie group. Coming to four dimensions, all but two solvable unimodular four-dimensional Lie algebras are almost Abelian. As unimodularity is a necessary condition for the existence of a lattice, some of the most important compact four-dimensional manifolds are almost Abelian solvmanifolds. This applies, e.g,  to compact complex surfaces where complex tori, hyperelliptic surfaces, primary Kodaira surfaces and Inoue surfaces of type $S^0$ are all almost Abelian solvmanifolds. Moreover, in higher dimensions, the compact complex manifolds introduced in \cite{EP} as analogues of Inoue surfaces are almost Abelian solvmanifolds. Besides, note that, in contrast to arbitrary $1$-connected solvable Lie groups, there is a relatively easy necessary and sufficient criterion for the existence of a cocompact lattice for $1$-connected almost Abelian Lie group $G$ \cite{Bo} and that the famous criterion of Mal'cev \cite{Ma} shows that lattices always exist if $G$ is even nilpotent.

Coming back to the above mentioned compact complex manifolds which are almost Abelian solvmanifolds $\Gamma\backslash G$, we observe that the complex structure on these manifolds is always an \emph{invariant} one, i.e. comes from a left-invariant one on $G$. Invariant geometric structures on almost Abelian solvmanifolds have played also an important role in providing famous counterexamples, e.g. the first example of a compact manifold admitting both a complex and a symplectic structure but no K\"ahler structure \cite{T} and the first example of a seven-dimensional compact manifold admitting a closed but no torsion-free $\G_2$-structure \cite{Fe} are almost Abelian solvmanifolds endowed with invariant geometric structures.

Now invariant geometric structures on solvmanifolds may solely be studied at the level of the associated solvable Lie algebra $\mfg$. Even more when $\mfg$ is almost Abelian, then $\mfg=\mfg_f$ is fully determined by one endomorphism $f\in \End(\mfu)\cong \End(\bR^{n-1})$ of the codimension one Abelian ideal $\mfu$. Thus, the existence question of a certain type of invariant geometric structure on a given almost Abelian solvmanifold $\Gamma\backslash G$ may be reformulated in terms of the endomorphism $f$ having specific kinds of properties, which makes the investigation which almost Abelian solvmanifolds admit a certain kind of invariant geometric structure highly approachable.

Hence, it is not too surprising that the last years have seen a rising interest in the study of these kinds of Lie algebras, Lie groups and solvmanifolds, and of invariant geometric structures on them: In a series of papers, Avetisyan and coauthors have investigated algebraic properties of almost Abelian Lie algebras over arbitrary fields \cite{A}, algebraic and topological properties of real and complex almost Abelian Lie group \cite{AABMRYZZ}, \cite{ABPR} and left-invariant Hermitian structures on complex almost Abelian Lie groups \cite{ABBMW}.  Moreover, the left-invariant positive Hermitian curvature flow on complex almost Abelian Lie groups has been studied in \cite{St}. Restricting for the rest of the paper to the case of \emph{real} almost Abelian Lie groups $G$, we note that various kinds of left-invariant geometric flows on this class of Lie groups have been investigated in the literature cf., e.g., \cite{Ar}, \cite{ArL}, \cite{BaFi},\cite{FrSchW}, \cite{L}, \cite{LRV}, \cite{LW}. Moreover, for different types of geometric structures, the almost Abelian Lie algebras admitting that type of geometric structure have been characterised or even classified, cf, e.g., \cite{AB1}, \cite{AB2}, \cite{AT}, \cite{ArBDGH},\cite{ArL}, \cite{BFrLT}, \cite{BeFi}, \cite{FiP1}, \cite{FiP2} \cite{Fr1}, \cite{Fr2}, \cite{LRV}, \cite{LW}, \cite{Mo}, \cite{P}. Here, charactisation usually means characterising those endomorphism $f$ of $\bR^{n-1}$ for which the associated almost Abelian Lie algebra $\mfg_f$ admits the geometric structure in question while classification usually refers to giving all possible Jordan normal forms of these $f$s.

Now many of the considered geometric structures may be described as torsion-free $H$-structures for a specific linear Lie group $H\leq \GL(n,\bR)$ and in these cases, the characterisation is usually given by determining the subspace of $\End(\bR^{n-1})$ for which $\mfg_f$ admits a torsion-free $H$-structure. In more detail, such characterisations have been obtained for $H=\mathrm{Sp}(2m,\bR)$ (``symplectic structures'') and for $H=\mathrm{U}(m)$ (``K\"ahler structures'') in \cite[Proposition 4.1]{LW} and following remarks, for $H=\mathrm{Gl}(m,\bC)$ (``complex structures'') in \cite[Lemma 6.1]{LRV}, and classifications for these linear subgroups $H$ have been obtained in the nilpotent case in \cite{ArBDGH}. Moreover, for $H=\mathrm{Sp}(2k,\bC)$ (``complex symplectic structures'') characterisations and classifications have been provided in \cite{BFrLT} and for $H\in \{\mathrm{G}_2,\mathrm{G}_2^*\}$ in \cite{Fr2}. Furthermore, characterisations for $H=\mathrm{Sp}(k)$ (``hyperk\"ahler structures'') have been given in \cite[Proposition 3.2]{BDFi} and for $H=\mathrm{GL}(k,\bH)$ in \cite[Theorem 3.2]{AB1} (``hypercomplex structures'') and classifications in the latter case for nilpotent $\mfg$ or $\mfg$ having dimension $12$ have been obtained in \cite{AB2}.

In this paper, we reprove all the just mentioned characterisation results and extend them vastly to big classes of linear subgroups. For this purpose, we show that the subspace $\cF_{\mfh}$ of all $f\in \End(\bR^{n-1})$ for which $\mfg_f$ admits a torsion-free \emph{special} $H$-structure $P$ is equal to the image $\cT(\cT^{-1}(\End(\bR^n)))$ of a certain linear map $\cT=\cT_{\mfh}:\cD_{\mfh}\rightarrow \Hom(\bR^{n-1},\bR^n)$. Here, \emph{special} means that $P$ admits a \emph{special} adapted basis $(X_1,\ldots,X_n)$, i.e. one such that $(X_1,\ldots,X_{n-1})$ is a basis of a codimension one Abelian ideal and $\End(\bR^{n-1})$ denotes the subspace of $\Hom(\bR^{n-1},\bR^n)$ consisting of those homomorphisms $f$ for which the image $\im(f)$ lies entirely in $\bR^{n-1}=\bR^{n-1}\times \{0\}\subseteq \bR^n$. Moreover, $\cD_{\mfh}$ is a certain subspace of $(\bR^n)^*\otimes (\bR^n)^*\otimes \bR^n$, which, given a special $H$-structure $P$ on an almost Abelian Lie algebra $\mfg$, can be thought of being the space of all $H$-connections on $\mfg$ which are torsion-free on the codimension one Abelian ideal $\mfu$ if one identifies $\bR^n$ with $\mfg$ via a special adapted basis. Moreover, we prove that $\cF_{\mfh}$ always contains the \emph{characteristic subalgebra} $\tilde{\mfk}_{\mfh}$ of $\mfh$ defined as
\begin{equation*}
\tilde{\mfk}_{\mfh}:=\left\{\left. F|_{\bR^{n-1}}\right| F\in \mfh,\ F(\bR^{n-1})\subseteq \bR^{n-1}\right\}\subseteq \End(\bR^{n-1}),
\end{equation*}
and that for any $f\in \tilde{\mfk}_{\mfh}$ the almost Abelian Lie algebra $\mfg_f$ admits a \emph{left-invariantly flat} special $H$-structure $P$. Here, \emph{left-invariantly flat} means that $G_f$ admits a left-invariant flat torsion-free connection compatible with $P$. We observe that if $\cF_{\mfh}=\tilde{\mfk}_{\mfh}$, then a special $H$-structure $P$ is left-invariantly flat if and only if $P$ is torsion-free.

To address also non-special $H$-structures $P$ on a given almost Abelian Lie algebra $\mfg$, we prove that $\mfu$ determines a unique $H$-orbit $[U]$ in $\mathrm{Grass}_{n-1}(\bR^n)$ and we then call $P$ of \emph{type} $[U]$. Moreover, we observe that for any $T\in \GL(n,\bR)$, the $H'=THT^{-1}$-structure $P'=PT^{-1}$ is torsion-free if and only if $P$ is torsion-free, and that we may always choose $T$ so that $P'$ is special. Hence, to determine all $f\in \End(\bR^{n-1})$ for which $\mfg_f$ admits any kind of $H$-structure, we first need to classify all $H$-orbits in $\mathrm{Gr}_{n-1}(\bR^n)$ and then compute for each of these orbits the associated subspace $\cF_{\mfh'}$ of $\End(\bR^{n-1})$. We remark that we will carry out both steps only in a few specific cases and mostly concentrate on determining the subspaces $\cF_{\mfh}$ for large classes of linear subalgebras $\mfh$ since on the one hand, the determination of $\cF_{\mfh}$ for a given linear subalgebra $\mfh$ is already of much interest, and secondly, for conjugated linear subalgebras $\mfh$ and $\mfh'$, the subspaces $\cF_{\mfh}$ and $\cF_{\mfh'}$ look often very similarly if $\mfh$ and $\mfh'$ share some common properties.

For the purpose of determining $\cF_{\mfh}$ for special classes of linear subalgebras $\mfh$, we start in Section \ref{sec:mfhcommutingwithendo} by considering subalgebras $\mfh$ of $\End(\bR^n)$ which commute with a given endomorphism $A$ of $\bR^n$. We first obtain full characterisations of those $f$ for which $\mfg_f$ admits a product or tangent structure of any possible type and show, in particular:
\begin{customthm}{1}
Any almost Abelian Lie algebra $\mfg$ admits product structures of any possible signature and also tangent structures.
\end{customthm}
In the mentioned section, we then prove:
\begin{customthm}{2}
Let $\mfh$ be a linear subalgebra of $\mathfrak{gl}(2m,\bR)$ which is also a Lie algebra over the complex numbers. Then $\cF_{\mfh}=\tilde{\mfk}_{\mfh}$.
\end{customthm}
We note that this theorem reproves, in particular, the characterisations of almost Abelian Lie algebras admitting torsion-free $H$-structures for the cases $H=\GL(m,\bC)$ and $H=\mathrm{Sp}(2k,\bC)$ from \cite{LRV} and \cite{BFrLT}, respectively.

In Section \ref{sec:mfhcommutingwithendo}, we finally consider \emph{totally real} subalgebras $\mfh$, i.e. $\mfh$ is a real subalgebra of $\mathfrak{gl}(J)\cong \mathfrak{gl}(m,\bC)$ for some complex structure $J$ on $\bR^{2m}$ with $\mfh\cap J\mfh=\{0\}$. For these subalgebras, Theorem \ref{th:totallyreal} gives (almost) a complete description of the associated subspace $\cF_{\mfh}$. As a special case of this theorem, we obtain
\begin{customthm}{3}
Let $\mfh$ be a \emph{super-elliptic}, i.e. $\mfh$ does not contain any endomorphisms of rank one or two, totally real linear subalgebra. Then $\cF_{\mfh}=\tilde{\mfk}_{\mfh}$. This applies, in particular, to $\mfh$ being a hypercomplex subalgebra or to $\mfh$ being a hyperparacomplex subalgebra which is induced by an elliptic subalgebra.
\end{customthm}
This theorem reproves the characterisations of almost Abelian Lie algebras admitting torsion-free $H$-structures for the cases $H=\mathrm{GL}(k,\bH)$ and $H=\mathrm{Sp}(k)$ from \cite{AB1} and \cite{BDFi}, respectively. 
We note that we also completely characterise in Theorem \ref{th:hyperparacomplexclass} the almost Abelian Lie algebras admitting a hyperparacomplex structure and in Corollary \ref{co:flathyperparacomplex} identify the flat ones among them. Finally, we compute in Theorem \ref{th:unitarysubalgebras} the subspace $\cF_{\mfh}$ for a unitary subalgebra, giving back the characterisation of the almost Abelian Lie algebras admitting a K\"ahler structure from \cite{LW}.

In the computation of $\cF_{\mfh}$ for totally real subalgebras, we will use decisively that the \emph{first prolongation} $\cK_{\mfh}^{(1)}$ of the \emph{associated} tableau $\cK_{\mfh}:=\left\{\left. F|_{\bR^{n-1}}\right|F\in \mfh \right\}\subseteq \Hom(\bR^{n-1},\bR^n)$ is of a very special form. Although an investigation of all linear subalgebras $\mfh$ for which $\cK_{\mfh}^{(1)}$ is of this special form seems to be too complicated, cf. the discussion at the beginning of Section \ref{sec:cK(1)specialtype}, we study linear subalgebras $\mfh$ with $\cK_{\mfh}^{(1)}$ of special types in Section \ref{sec:cK(1)specialtype}. First, we obtain in Theorem \ref{th:cK1=0} the following result, where we refer to the definition of the subspace $W_{\mfh}$ to the mentioned section:
\begin{customthm}{4}
Let $\mfh$ be a linear subalgebra with $\cK^{(1)}_{\mfh}=\{0\}$. Then 
\begin{equation*}
		\cF_{\mfh}=\tilde{\mfk}_{\mfh}+ (\bR^{n-1})^*\otimes W_{\mfh}.
	\end{equation*}
If $\mfh$ is super-elliptic, then $\cF_{\mfh}=\tilde{\mfk}_{\mfh}$, which applies to super-elliptic metric subalgebras.
\end{customthm}
We note that this theorem reproves the characterisation of the almost Abelian Lie algebras admitting a torsion-free $\G_2$-structure from \cite{Fr2}.

Afterwards, we consider in Subsections \ref{subsec:cK1=S2Uv} and \ref{subsec:cK1=S2Uw} the case that $\cK_{\mfh}^{(1)}=S^2\cU\otimes z$ for a subspace $\cU$ of $(\bR^{n-1})^*$ and $z\notin \bR^{n-1}$ or $0\neq z\in \bR^{n-1}$, respectively. This is motivated by the fact that for a metric subalgebra $\mfh$, the first prolongation $\cK_{\mfh}^{(1)}$ is of exactly of one of these two forms, where it depends on whether $\bR^{n-1}$ is non-degenerate or degenerate which of these forms it takes. We prove in both subsections that $\mfh$ always contains a subalgebra which is metric for a suitably chosen pseudo-metric on $\bR^n$ and obtain some structural results for the linear subalgebra $\mfh$. As for the computation of $\cF_{\mfh}$, we obtain in Subsection \ref{subsec:cK1=S2Uv} the following result, cf. Theorem \ref{th:cK1=S2cUv}:
\begin{customthm}{5}
Let $\mfh$ be a linear subalgebra with $\cK_{\mfh}^{(1)}=S^2 \cU\otimes v$ for some non-zero subspace $\cU$ of $(\bR^{n-1})^*$ and some $v\in \bR^{n}\setminus \bR^{n-1}$. Then:
\begin{enumerate}[(a)]
\item
If $\mfh$ is not elliptic, then $\cF_{\mfh}=\tilde{\mfk}_{\mfh}$.
\item
If $\mfh$ is elliptic, then
\begin{equation*}
\cF_{\mfh}=\tilde{\mfk}_{\mfh}+ \spa{\alpha\otimes \nu(\alpha)|\alpha\in \cU}.
\end{equation*}
This applies in particular to non-degenerate metric subalgebras, also those with $\cU\neq \{0\}$.
\end{enumerate}
\end{customthm}
For the definition of $\cU$ and $\nu$, we refer to Subsection \ref{subsec:cK1=S2Uv}. Moreover, we note that this reproves the charactersion of those almost Abelian Lie algebras admitting a torsion-free $\G_2^*$-structure with $\mfu$ being non-degenerate from \cite{Fr2}.

In Subsection \ref{subsec:cK1=S2Uw}, we obtain in Theorem \ref{th:cK1=S2cUwfirst} and Theorem \ref{th:cK1=S2cUwsecond} a full description of $\cF_{\mfh}$ for $\mfh$ being an elliptic linear subagebra with $\cK_{\mfh}^{(1)}=S^2 \cU\otimes w$ for $w\in \bR^{n-1}$. These theorems also apply to degenerate metric subalgebras and give back the characterisation of the almost Abelian Lie algebras admitting a torsion-free $\G_2^*$-structure with degenerate $\mfu$ from \cite{Fr2}.

Finally, let us note that in all considered cases, the super-ellipticity of $\mfh$ implied $\cF_{\mfh}=\tilde{\mfk}_{\mfh}$, leading to the question.
\begin{question}
Does any super-elliptic linear subalgebra $\mfh$ satisfy $\cF_{\mfh}=\tilde{\mfk}_{\mfh}$?
\end{question}
 \section{$H$-structures and almost Abelian Lie groups}

\subsection{$H$-structures}
We start with the basic definition of an $H$-structure on a manifold $M$:
\begin{definition}
Let $M$ be an $n$-dimensional manifold and $H$ be a Lie subgroup of $\GL(n,\bR)$. An \emph{$H$-structure} $P$ on $M$ is a reduction of the frame bundle $\cF(M)$ to $H$, i.e. an $H$-principal subbundle of $\cF(M)$. A (local) section of $P$ is also called an \emph{adapted (local) frame} of $P$ and an element $u\in P_x$ is called an \emph{adapted basis (of $P$ at $x\in M$)}. Here and in the following, we consider elements $u\in \cF(M)_x$ depending on the context either as a basis $u=(X_1,\ldots,X_m)$ of $T_x M$ or as a linear isomorphism $u:\bR^n\rightarrow T_x M$.
\end{definition}
Next, we recall the definition of a connection compatible with a given $H$-structure:
\begin{definition}
Let $M$ be an $n$-dimensional manifold, $H$ be a Lie subgroup of $\GL(n,\bR)$ and $P$ be a $H$-structure. Let $\nabla$ be a connection on $M$. We call $\nabla$ \emph{compatible with $P$} or simply an \emph{$H$-conection} if for every $p\in M$ there exists a neighbourhood $U$ and a local adapted frame $(X_1,\ldots,X_m)$ of $P$ on $U$ such that the associated connection one-form $\omega=(\omega_{ij})_{i,j=1,\ldots,n}\in \Omega^1(U,\mathfrak{gl}(n,\bR))$ implicitly defined by
\begin{equation*}
\nabla X_i=\sum_{j=1}^n \omega_{ji} X_j
\end{equation*}
for $i=1,\ldots,n$ has values in $\mathfrak{h}=\mathrm{Lie}(H)\subseteq \mathfrak{gl}(n,\bR)$.
\end{definition}
\begin{remark}
Note that if $\nabla$ is compatible with an $H$-structure $P$, then the associated matrix-valued one-form $\omega$ has values in $\mfh$ for \emph{any} adapted local frame of $P$.
\end{remark}
With these definitions at hand, we define now when an $H$-structure is called \emph{torsion-free} and \emph{flat}:
\begin{definition}
An $H$-structure $P$ on a manifold $M$ is called
\begin{itemize}[wide]
\item
\emph{flat} if for every point $p\in M$ there exists an open neighbourhood $U$ of $p$ and a chart $\varphi=(x_1,\ldots,x_n):U\rightarrow \varphi(U)\subseteq \bR^n$ such that $(\partial_{x_1},\ldots,\partial_{x_n})$ is a local section of $P$
\item
and \emph{torsion-free} if $M$ admits a torsion-free $H$-connection $\nabla$.
\end{itemize}
\end{definition}
We recall the following well-known equivalent definition of a flat $H$-structure and that flatness implies torsion-freeness.
\begin{lemma}\label{le:flatequivalent}
Let $P$ be an $H$-structure on $M$. Then $P$ is flat if and only if locally around every point there exists a flat torsion-free $H$-connection.

Consequently, any flat $H$-structure $P$ is torsion-free.
\end{lemma}
\begin{proof}
First, let $P$ be flat and $(\partial_1,\ldots,\partial_n)$ be a local section of $P$ by coordinate vector fields. Let $\nabla$ be the local connection uniquely defined by $\nabla_{\partial_i}\partial_j=0$ for all $i,j=1,\ldots,n$. Surely, $\nabla$ is an $H$-connection which is, moreover, torsion-free and flat since obviously $T^{\nabla}(\partial_i,\partial_j)=0$ and $R^{\nabla}(\partial_i,\partial_j)\partial_k=0$ for all $i,j,k=1,\ldots,n$.

Conversely, assume that $\nabla$ is a flat torsion-free $H$-connection on an open subset $U$ of $M$ and let $\omega\in \Omega^1(U,\mfh)$ be the connection one-form for some local section $e=(X_1,\ldots,X_n)$ of $P$ on $U$. The flatness of $\nabla$ implies $d\omega+\omega\wedge \omega=0$. We need to find another local section $e'=(X_1',\ldots,X_n')$ of $P$ on $U$ for which the connection form $\omega$ vanishes as then the torsion-freeness shows that $[X_i',X_j']=0$ for all $i,j=1,\ldots,n$ and so that $e'$ is locally a section by coordinate fields.
For this, note that giving another section $e'$ of $P$ on $U$ is equivalent to giving some $A\in C^{\infty}(U,H)$ such that $e'=e\cdot A$ and that then the associated connection one-form $\omega'$ satisfies $\omega'=A^{-1}\omega A+A^{-1} dA$. Considering $\rho:=A^{-1}\omega A+A^{-1} dA$ as a one-form on $U\times H$ with values in $\mathfrak{h}$, the condition
$d\omega+\omega\wedge \omega=0$ implies that the ideal in $\Omega(U\times H)$ generated by the components of $\rho$ is a differential ideal and so the kernel of $\rho$ defines an integrable distribution $D$ on $U\times H$. As any integral manifold $N$ of $D$ may locally be written as a graph over $U$, we obtain around every point in $U$ a smooth function $V\ni x\mapsto A(x)\in H$ such that for any $v\in T_x V$, we have
\begin{equation*}
0=\rho_{(x,A(x))}(v,dA_x (v))=A^{-1}(x)\,\omega_x(v)\, A(x)+A^{-1}(x) dA_x(v)=\omega'_x(v)	
\end{equation*}
i.e. the associated connection one-form $\omega'$ vanishes.

Finally, assume that $P$ is flat. Then, for every $p\in M$, we have a flat torsion-free $H$-connections $\nabla^p$ defined on a neighbourhood $U_p$ around $p\in M$. Using a partition of unity subordinate to the cover $(U_p)_{p\in M}$, we can patch the connections $\nabla^p$ together to obtain a connection $\nabla$ on the entire manifold $M$. Since $\mfh$ is a linear subspace of $\mathfrak{gl}(n,\bR)$, the connection $\nabla$ is an $H$-connection and this connection is again torsion-free as torsion-freeness is a linear condition.
\end{proof}
\begin{remark}
If $H$ is closed, then there is a one-to-one correspondence between $H$-structures $P$ and sections $\sigma:M\rightarrow \cF(M)/H$ of the fibre bundle $\cF(M)/H$. In this case, one can show, cf., e.g., \cite{Br1}, that an $H$-structure $P$ is torsion-free if and only if the associated section $\sigma_P:M\rightarrow \cF(M)/H$ is locally flat up to second order terms, i.e. there exists locally a section $\sigma$ of $\cF(M)/H$ corresponding to a (locally defined) flat $H$-structure such that the $1$-jets of $\sigma_P$ and $\sigma$ coincide. This this the reason why torsion-free $H$-structures are also called \emph{$1$-flat}.
\end{remark}
If $H'$ is conjugated to $H$ within $\GL(n,\bR)$, then there is a bijection between $H$- and $H'$-structures which preserves both the torsion-free and flatness condition: 
\begin{lemma}\label{le:conjugacyclassesofH}
	Let $H$ be a Lie subgroup of $\GL(n,\bR)$ and $T\in \GL(n,\bR)$ be given. Then the map $P\mapsto P\, T^{-1}$ provides a bijection between $H$-structures and $THT^{-1}$-structures which preserves both the torsion-free and the flatness condition.
\end{lemma}
\begin{proof}
The map $P\mapsto P\, T^{-1}$ is obviously a bijection between $H$-structures and $THT^{-1}$-structures. Next, let $\nabla$ be an $H$-connection. Then the connection one-form $\omega$ of $\nabla$ with respect to some local adapted frame $e$ of $P$ has values in $\mfh$ and so the connection one-form $\omega'$ of $\nabla$ with respect to $e\cdot T^{-1}$, which is a local adapted frame of $P T^{-1}$, is given by $\omega'=T\omega T^{-1}$. Thus, $\omega'$ has values in $T\mfh T^{-1}$ and so $\nabla$ is an $THT^{-1}$-connection. This directly implies that both the torsion-free and the flatness condition is preserved by the map $P\mapsto P\, T^{-1}$.
\end{proof}
In this article, we will consider $H$-structures $P$ on Lie groups, more exactly those which are \emph{left-invariant} in the following sense:
\begin{definition}
Let $G$ be a Lie group with Lie algebra $\mfg$ and $P$ be an $H$-structure on $G$. Then $P$ is called \emph{left-invariant} if $P$ admits a global adapted frame \linebreak $(X_1,\ldots,X_n)$ consisting of left-invariant vector fields $(X_1,\ldots,X_n)$. If $P$ is left-invariant, we obtain an associated $H$-orbit $(X_1,\ldots,X_n)\cdot H$ in the frame bundle $\cF(\mfg)$ of the Lie algebra $\mfg$ and, conversely, such an orbit uniquely determines a left-invariant $H$-structure. Hence, we identify left-invariant $H$-structures with $H$-orbits in $\cF(\mfg)$ and speak also of an \emph{$H$-structure} on the Lie algebra $\mfg$.
\end{definition}
By definition, an $H$-structure $P$ is torsion-free if and only if it admits a torsion-free $H$-connection $\nabla$. If $P$ is left-invariant, we will show below that we may then choose $\nabla$ to be \emph{left-invariant} in the following sense:
\begin{definition}
A connection $\nabla$ on a Lie group $G$ is said to be \emph{left-invariant} if $\nabla_X Y$ is left-invariant for all left-invariant vector fields $X,Y$. Note that then $\nabla$ may be identified with a bilinear map $\nabla:\mfg\times \mfg\rightarrow \mfg$ 
\end{definition} 
\begin{remark}
Let $G$ be a Lie group.
\begin{itemize}[wide]
\item
Note that a connection $\nabla$ on $G$ is left-invariant if and only if for all $g\in G$ and all $X,Y\in \mathcal{X}(G)$, we have
\begin{equation*}
	\nabla_{l_g^*X} l_g^*Y=l_g^*(\nabla_X Y)
\end{equation*}
\item
If $P$ is a left-invariant $H$-structure on $G$ and $\nabla$ is left-invariant as well, then $\nabla$ is an $H$-connection if and only if for some, and hence any, $u\in P$, i.e. $u:\bR^n\rightarrow \mfg$, we have $u^*\nabla:=u^{-1}\circ \nabla_{u(\cdot)} u(\cdot)\in (\bR^n)^*\otimes \mfh$.
\end{itemize}
\end{remark}
\begin{lemma}\label{le:torsionfreeleftinvariant}
Let $G$ be a Lie group and $P$ be a left-invariant $H$-structure on $G$. Then $P$ is torsion-free if and only if $P$ admits a left-invariant torsion-free $H$-connection.
\end{lemma}
\begin{proof}
As the direction ``$\Leftarrow$'' is clear by definition, let us assume that $P$ is torsion-free. By definition, we then have a torsion-free $H$-connection $\nabla$. Let $(X_1,\ldots,X_n)$ be a global left-invariant section of $P$ and let $\nabla^L$ be the connection on $G$ uniquely defined by
\begin{equation*}
(\nabla^L_{X_i} X_j)(g):=d(l_g)_e((\nabla_{X_i} X_j)(e)).
\end{equation*}
Then the connection one-form $\omega^L$ of $\nabla^L$ with respect to $(X_1,\ldots,X_n)$ satisfies $\omega^L(g)=\omega(e)$, where $\omega$ is the connection form of $\nabla$ with respect to $(X_1,\ldots,X_n)$. As $\omega$ has values in $\mfh$, the same is true for $\omega^L$, i.e. $\nabla^L$ is an $H$-connection. Moreover, using the torsion-free property of $\nabla$ and the fact that commutators of left-invariant vector fields are again left-invariant, we get
\begin{equation*}
\begin{split}
(\nabla^L_{X_i} X_j)(g)-(\nabla^L_{X_j} X_i)(g)&=d(l_g)_e\bigl((\nabla_{X_i} X_j)(e)-(\nabla_{X_j} X_i)(e)\bigr)=d(l_g)_e\left([X_i,X_j](e)\right)\\
&=[X_i,X_j](g),
\end{split}
\end{equation*}
and so that $\nabla^L$ is also torsion-free. This finishes the proof.
\end{proof}
One may wonder if the same is true in the flat case, i.e. if an $H$-structure $P$ on a Lie algebra $\mfg$ is flat if and only if $P$ admits a compatible left-invariant flat torsion-free connection. This is not the case. A counterexample may be given by the $\GL(n,\bR)$-structure $\cF(\mfg)$ of a Lie algebra $\mfg$ which does not admit a flat torsion-free connection. Such examples may be found, e.g., in \cite{Be}. 

Thus, these kind of counterexamples justify to give those $H$-structures on Lie algebras $\mfg$ which admit a left-invariant flat torsion-free $H$-connection a special name:
\begin{definition}
An $H$-structure $P$ on a Lie algebra $\mfg$ is called \emph{left-invariantly flat} if $P$ admits a compatible left-invariant flat torsion-free connection.
\end{definition}

\subsection{Almost Abelian Lie groups and associated solvmanifolds}
We start with the basic definitions of almost Abelian Lie algebras and Lie groups and of almost Abelian solvmanifolds:
\begin{definition}
\begin{itemize}[wide]
\item
A Lie algebra $\mfg$ is called \emph{almost Abelian} if it admits a codimension one Abelian ideal $\mfu$.
\item
A Lie group $G$ is called \emph{almost Abelian} if its associated Lie algebra is almost Abelian.
\item
A compact manifold $M$ is called an \emph{almost Abelian solvmanifold} if it is diffeomorphic to a quotient $\Gamma \backslash G$, where $G$ is a simply-connected almost Abelian Lie group and $\Gamma$ is a \emph{cocompact lattice}, i.e. a discrete subgroup of $G$ such that the quotient $\Gamma\backslash G$ is compact.
\end{itemize}
\end{definition}
\begin{remark}
\begin{itemize}[wide]
\item
By the proof of \cite[Proposition 1]{Fr1}, an $n$-dimensional almost Abelian Lie algebra $\mfg$ has a unique codimension one Abelian ideal $\mfu$ provided that $\mfg$ is neither Abelian nor isomorphic to $\mfh_3\oplus \bR^{n-3}$.
\item
Choosing any $X\in \mfg\setminus \mfu$, the entire Lie bracket of the almost Abelian Lie algebra $\mfg$ is encoded into the endomorphism $f:=\ad(X)|_{\mfu}\in \End(\mfu)$ of the codimension one Abelian ideal.  Note that $f$ is uniquely determined up to non-zero scaling. 
\item Any $n$-dimensional almost Abelian Lie algebra is isomorphic to a semi-direct product $\mfg_f:=\bR^{n-1}\rtimes_f \bR$ for some $f\in \End(\bR^{n-1})$, where $t\in \bR$ acts on $\bR^{n-1}$ by $t\cdot f$. By  \cite[Proposition 1]{Fr1}, two such Lie algebras $\mfg_{f_1}$ and $\mfg_{f_2}$ are isomorphic if and only if $f_2$ is similar to $f_1$ up to non-zero scaling. Hence, $n$-dimensional almost Abelian Lie algebras are classified by the possible Jordan normal forms of $f\in \End(\mfu)\cong \End(\bR^{n-1})$ up to scaling.
\item
Since the Jordan normal form of a nilpotent endomorphism contains only $0$ and $1$ as entries, the Mal'cev criterion \cite{Ma} shows that any simply-connected nilpotent almost Abelian Lie group $G$ admits a cocompact lattice.
\item
If $G$ is any simply-connected almost Abelian Lie group, there is a necessary and sufficient criterion of Bock \cite{Bo} for the existence of a cocompact lattice $\Gamma$ in $G$, namely $G$ admits such a cocompact lattice if and only if there is some $t_0\in \bR\setminus \{0\}$ such that with respect to some basis of $\mfu$, $\exp(t_0 f)\in \GL(\mfu)$ has only integer entries.
\end{itemize}
\end{remark}
We now turn to $H$-structures $P$ on almost Abelian Lie algebras and note that $P$ determines a unique $H$-orbit in $\Gr_{n-1}(\bR^n)$ allowing to define different \emph{types} of $H$-structures:
\begin{definition}
Let $\mfg$ be an $n$-dimensional almost Abelian Lie algebra with codimension one Abelian ideal $\mfu$ and $P$ be an $H$-structure on $\mfg$. Then for any $u\in P$, the $(n-1)$-dimensional subspace $u^{-1}(\mfu)$ of $\bR^n$ determines the same $H$-orbit $[u^{-1}(\mfu)]=u^{-1}(\mfu)\cdot H$ in the Grassmannian $\Gr_{n-1}(\bR^n)$ of hyperplanes in $\bR^n$.

In this situation, we say that $P$ is \emph{of type $[u^{-1}(\mfu)]$}. If $P$ is of type $[\bR^{n-1}]$ with $\bR^{n-1}:=\bR^{n-1}\times \{0\}\subseteq \bR^n$, we call the $H$-structure $P$ \emph{special}. So for a special $H$-structure $P$ there exist adapted frames $u\in P$ with $u^{-1}(\mfu)=\bR^{n-1}$. We call such an adapted frame $u\in P$ then also \emph{special}.
\end{definition}
\begin{remark}
Note that all $H$-structures on an almost Abelian Lie algebra are special if and only if $H$ acts transitively on the Grassmannian $\mathrm{Grass}_{n-1}(\bR^n)$ of $(n-1)$-dimensional subspaces. The connected Lie subgroups of $\GL(n,\bR)$ acting transitively on $\mathrm{Grass}_{n-1}(\bR^n)$ may be found in \cite[6.1. Theorem, 6.17. Theorem]{Kr}:
\end{remark}
\subsection{Torsion-free $H$-structures on almost Abelian Lie algebras}
In this subsection, we define a subspace $\cF_{\mfh}$ of $\End(\bR^{n-1})$ such that a given special $H$-structure $P$ is torsion-free if and only if, under a suitable identification of $\mfu$ with $\bR^{n-1}$, the endomorphims $f$ of $\End(\mfu)$ is in $\cF_{\mfh}$. As not every $H$-structure is special, the following observation will be useful to obtain results also for non-special $H$-structures:
\begin{lemma}\label{le:reductiontospecialHstructure}
Let $\mfg$ be an $n$-dimensional almost Abelian Lie algebra and $P$ be an $H$-structure on $\mfg$ which is of type $[U]$. Choose some $T\in \GL(\mfu)$ with $T(U)=\bR^{n-1}$. Then the $THT^{-1}$-structure $P\circ T^{-1}$ is special. Moreover, $P$ is torsion-free or flat, respectively, if and only if $P\circ T^{-1}$ is torsion-free or flat, respectively.
\end{lemma}
\begin{proof}
Follows directly from Lemma \ref{le:conjugacyclassesofH}.
\end{proof}
Due to the last lemma, we may restrict for the moment to special $H$-structures. We first will define a subalgebra $\tilde{\mfk}_{\mfh}$ of $\End(\bR^{n-1})$ naturally associated to $\mfh$:
\begin{definition}
Let $\mfh$ be a Lie subalgebra of $\End(\bR^n)$. Then we let
\begin{equation*}
\mfk_{\mfh}:=\left\{F\in \mfh,\ F(\bR^{n-1})\subseteq \bR^{n-1}\right\}\subseteq \mfh
\end{equation*}
denote the subalgebra of $\mfh$ preserving $\bR^{n-1}$ and define the Lie subalgebra $\tilde{\mfk}_{\mfh}$ of $\End(\bR^{n-1})$ by
\begin{equation*}
\tilde{\mfk}_{\mfh}:=\left\{\left. F|_{\bR^{n-1}}\right| F\in \mfk_{\mfh} \right\}.
\end{equation*}
We call $\tilde{\mfk}_{\mfh}$ the \emph{characteristic} subalgebra of $\End(\bR^{n-1})$ associated to $\mfh$ and denote by $K_{\mfh}$ or $\tilde{K}_{\mfh}$, respectively, the associated connected Lie subgroups of $H$ or $\GL(n-1,\bR)$, respectively.
\end{definition}
\begin{example}
Table \ref{table:charsubalgs} lists the characteristic subalgebra $\tilde{\mfk}_{\mfh}$ for various linear subalgebras $\mfh$ of $\End(\bR^n)$. Here, $\mfh$ and $\tilde{\mfk}_{\mfh}$ are embedded into $\End(\bR^n)$ or $\End(\bR^{n-1})$, respectively, in the ``standard way''. In the case $\mfh=\mfg_2^*$, there are essentially $3$ different ways to embed $\mfg_2^*$ into $\bR^7$, namely one where $\bR^6$ has signature $(3,3)$, one where $\bR^6$ has signature $(2,4)$ and one where $\bR^6$ is degenerate, and we indicate which case we consider in the table. In the degenerate case, we use the (implicit) definition of $\mathfrak{g}_2^*$ from \cite{Fr2} as the stabiliser of a certain three-form on $\bR^7$ and we set $J_0:=\left(\begin{smallmatrix} 0 & -1 \\ 1 & 0 \end{smallmatrix} \right)$.
\begin{longtable}[ht]{ccc}
	\caption{Characteristic subalgebras of some linear sualgebras} \\ \hline
	$n$ & $\mfh$ & $\tilde{\mfk}_{\mfh}$ \\
	\hline
	\endhead
	\label{table:charsubalgs}
$n$ & $\mathfrak{so}(n)$ & $\mathfrak{so}(n-1)$ \\
$2m$ & $\mathfrak{sp}(2m,\bR)$ & $\left\{\left.\left(\begin{smallmatrix} A & 0 \\
	                                                                                    w^t & a\end{smallmatrix}\right)\right|A\in \mathfrak{sp}(2m-2,\bR),\, w\in \bR^{2m-2},\, a\in \bR\right\}$ \\
$2m$ & $\mathfrak{gl}(m,\bC)$ & $\left\{\left.\left(\begin{smallmatrix} A & v \\
	                                                                                    0 & a\end{smallmatrix}\right)\right|A\in \mathfrak{gl}(m-1,\bC),\, v\in \bR^{2m-2},\, a\in \bR\right\}$ \\	                                                                                    
$2m$ & $\mathfrak{u}(m)$ & $\mathfrak{u}(m-1)\times \{0\}$ \\
$2m$ & $\mathfrak{su}(m)$ & $\mathfrak{su}(m-1)\times \{0\}$ \\
$4k$ & $\mathfrak{sp}(k)$ & $\mathfrak{sp}(k-1)\times \{0\}^3$ \\
$4k$ & $\mathfrak{sp}(k)\oplus \mathfrak{sp}(1)$ & $\left(\mathfrak{sp}(k-1)\oplus \mathfrak{sp}(1)\right)\times \{0\}^3$ \\
$7$ & $\mfg_2$ & $\mathfrak{su}(3)$ \\
$7$ & $\mfg_2^*$ & $\mathfrak{sl}(3,\bR)$ (signature $(3,3)$) \\
 & & $\mathfrak{su}(1,2)$ (signature $(2,4)$) \\
 & & $\left\{\left.\left(\begin{smallmatrix} -\tr(A) & -2 b & v^t & w^t \\ 0 & 0 & 0 & v^t \\ 0 & -J_0 v & A-\tr(A) I_2  & b I_2 \\ 0 & 0 & 0 & A
	                                                                                    \end{smallmatrix}\right)\right|A\in \mathfrak{gl}(2,\bR),\, v,w\in \bR^{2},\, b\in \bR\right\}$  (degenerate) \\
 $8$ & $\mathfrak{spin}(7)$ & $\mfg_2$
  \end{longtable}
 \end{example}
Next, we define the mentioned subspace $\cF_{\mfh}$:
\begin{definition}
Let $\mfh$ be a Lie subalgebra of $\mathfrak{gl}(n,\bR)$. We first set
\begin{equation*}
\cD_{\mfh}:=\left\{\nabla\in (\bR^n)^*\otimes \mfh\left| \nabla|_{\bR^{n-1}\times \bR^{n-1}}\in S^2 (\bR^{n-1})^*\otimes \bR^n\right.\right\}.
\end{equation*}
For a special $H$-structure $P$, $\cD_{\mfh}$ contains the ``pull-backs'' of $H$-connections along a special adapted frame with the torsion-free condition already being imposed on the codimension one Abelian ideal. Hence, we denote elements in $\cD_{\mfh}$ like connections with $\nabla$ and write $\nabla_u w$ instead of $\nabla(u,w)$,

Next, we choose some $v\in \bR^n\setminus \bR^{n-1}$ and define a map
\begin{equation*}
\cT:=\cT_v:\cD_{\mfh}\rightarrow \Hom(\bR^{n-1},\bR^n),\quad \cT(\nabla):=\left(\nabla_v-\nabla v\right)|_{\bR^{n-1}}=\left(\nabla_v-\nabla v\right)\circ \iota,
\end{equation*}
where $\iota:\bR^{n-1}\rightarrow \bR^n$ is the natural inclusion. Moreover, we identify $\End(\bR{n-1})$ naturally with the subspace of $\Hom(\bR^{n-1},\bR^n)$ consisting of those elements $f\in \Hom(\bR^{n-1},\bR^n)$ for which $\im(f)\subseteq \bR^{n-1}\subseteq \bR^n$ and finally set
\begin{equation*}
\cF_{\mfh}:=\cT(\cT^{-1}(\End(\bR^{n-1})))\subseteq \End(\bR^{n-1}).
\end{equation*}
We note that different choices of $v$ result in an overall scaling of the map $\cT_v$ and so do neither affect the preimage $\cT^{-1}(\End(\bR^{n-1}))$ nor the subspace $\cF_{\mfh}=\cT(\cT^{-1}(\End(\bR^{n-1})))$.
\end{definition}
We remark that $\cF_{\mfh}$ is a $\tilde{K}_{\mfh}$-module:
\begin{proposition}\label{pro:tildeKmodule}
Let $H$ be a Lie subgroup of $\GL(n,\bR)$ with Lie algebra $\mfh$. Then $\cF_{\mfh}$ is a $\tilde{K}_{\mfh}$-submodule of $\End(\bR^{n-1})$. 
\end{proposition}
\begin{proof}
Let $\tilde{k}\in \tilde{K}_{\mfh}$ and $F\in \cF_{\mfh}$. By definition, there is $h\in H$ with $h(\bR^{n-1})=\bR^{n-1}$ and $h|_{\bR^{n-1}}=\tilde{k}$.

Again, by definition, there is $\nabla\in \cD_{\mfh}$ such that $\cT_{e_n}(\nabla)\in \End(\bR^{n-1})$ and $\cT_{e_n}(\nabla)=F$. Then, since $h$ preserves $\bR^{n-1}$, the pullback $h^*\nabla=h^{-1}(\nabla_{h(\cdot)} h(\cdot))$ is in $\cD_{\mfh}$ as well and, choosing $v=h^{-1}(e_n)$, we have 
\begin{equation*}
\begin{split}
\cT_v(h^*\nabla)(u)&=h^{-1}(\nabla_{h(v)} h(u)-\nabla_{h(u)} h(v))=h^{-1}(\nabla_{e_n} \tilde{k}(u)-\nabla_{\tilde{k}(u)} e_n)\\
&=h^{-1}(\cT_{e_n}(\nabla)(\tilde{k}(u)))=(\tilde{k}^{-1}\circ F\circ \tilde{k})(u)
\end{split}
\end{equation*}
for all $u\in \bR^{n-1}$. Hence, also $\tilde{k}^{-1}\circ F\circ \tilde{k}$ is in $\cF_{\mfh}$ and $\cF_{\mfh}$ is a $\tilde{K}_{\mfh}$-module.
\end{proof}
Next, we show that a special $H$-structure $P$ is torsion-free if and only if $f$ is in $\cF_{\mfh}$ and give also a sufficient criterion for the left-invariant flatness of $P$:
\begin{proposition}\label{pro:specialHstrstf}
Let $P$ be a special $H$-structure on an almost Abelian Lie algebra $\mfg$ with codimension one Abelian ideal $\mfu$. Moreover, let $X\in \mfg\setminus \mfu$, set $f:=\ad(X)|_{\mfu}$ and use a fixed special adapted frame to identify $\mfg$ with $\bR^n$, $\mfu$ with $\bR^{n-1}$ and so $f$ with an endomorphism of $\bR^{n-1}$. Then:
\begin{enumerate}[(a)]
\item
If $f\in \tilde{\mfk}_{\mfh}$, then $P$ is left-invariantly flat.
\item
$P$ is torsion-free if and only if $f\in \cF_{\mfh}$.
\end{enumerate}
\end{proposition}
\begin{proof}
\begin{enumerate}
\item[(a)]
First of all, since $f\in \tilde{\mfk}_h$, there is by definition of the characteristic subalgebra of $\End(\bR^{n-1})$ an element $F\in \mfh$ with $F(\bR^{n-1})\subseteq \bR^{n-1}$ and
\begin{equation*}
F|_{\bR^{n-1}}=f
\end{equation*}
We define now a left-invariant connection $\nabla$ on $\mfg\cong \bR^n$ by setting $\nabla_u:=0$ for all $u\in \bR^{n-1}$ and
\begin{equation*}
\nabla_{e_n}:=\lambda\, F
\end{equation*}
for $\lambda\in \bR^*$ being the unique non-zero scalar such that $\ad(e_n)|_{\bR^{n-1}}=\lambda\, f$.

Obviously, with this definition, $\nabla$ is an $H$-connection and $R^{\nabla}=0$, i.e. $\nabla$ is flat. Next, $T^{\nabla}(u_1,u_2)=0$ for all $u_1,u_2\in \bR^{n-1}$ as $\bR^{n-1}$ is Abelian and we also have
\begin{equation*}
\begin{split}
T^{\nabla}(e_n,u)&=\nabla_{e_n}(u)-[e_n,u]=\lambda\, F(u)-\lambda\, f(u)=\lambda\, f(u)-\lambda\, f(u)=0.
\end{split}
\end{equation*}
Thus, $\nabla$ is also torsion-free and so a left-invariant flat torsion-free $H$-connection, i.e. $P$ is left-invariantly flat.
\item[(b)]
First, let $P$ be torsion-free. By Lemma \ref{le:torsionfreeleftinvariant}, there is a left-invariant torsion-free $H$-connecton $\nabla$. Using the fixed special adapted frame, $\nabla$ is an element of $(\bR^n)^*\otimes \mfh$. Since $\nabla$ is torsion-free, we have
\begin{equation*}
0=T^{\nabla}(u_1,u_2)=\nabla_{u_1} u_2-\nabla_{u_2} u_1-[u_1,u_2]=\nabla_{u_1} u_2-\nabla_{u_2} u_1
\end{equation*}
for all $u_1,u_2\in \bR^{n-1}$, i.e. $\nabla\in \cD_{\mfh}$. Using our identifications, we have $X\in\bR^n\setminus \bR^{n-1}$ and so
\begin{equation*}
\cT_X(\nabla)(u)=\nabla_X u-\nabla_u X=T^{\nabla}(X,u)+[X,u]=f(u),
\end{equation*}
for all $u\in \bR^{n-1}$, i.e. $\cT_X(\nabla)=f\in \End(\bR^{n-1})$. Thus, $f\in \cF_{\mfh}$.

Conversely, assume that $f\in \cF_{\mfh}$. Then $f=\cT_X(\nabla)$ for some $\nabla\in \cD_{\mfh}$ with $\cT_X(\nabla)\subseteq \End(\bR^{n-1})$. The condition $\nabla\in (\bR^n)^*\otimes \mfh$ says that $\nabla$ is an $H$-connection and the condition $f=\cT_X(\nabla)$ implies
\begin{equation*}
[X,u]=f(u)=\cT_X(\nabla)(u)=\nabla_X u-\nabla_u X,
\end{equation*}
and so $T^{\nabla}(X,u)=0$ for all $u\in \bR^{n-1}$. Since $\nabla|_{\bR^{n-1}\times \bR^{n-1}}\in S^2 (\bR^{n-1})^*\otimes \bR^n$, we also have
\begin{equation*}
T^{\nabla}(u_1,u_2)=\nabla_{u_1} u_2-\nabla_{u_2} u_1-[u_1,u_2]=0
\end{equation*}
for all $u_1,u_2\in \bR^{n-1}$. Thus, in total, $T^{\nabla}=0$. Consequently, $\nabla$ is a torsion-free $H$-connection on $\mfg\cong \bR^n$ and so $P$ is torsion-free.
\end{enumerate}
\end{proof}
Proposition \ref{pro:specialHstrstf} and Lemma \ref{le:flatequivalent} directly imply 
\begin{corollary}\label{co:flat=tf}
Let $H$ be a linear subgroup of $\GL(n,\bR)$. Then always $\tilde{\mfk}_{\mfh}\subseteq \cF_{\mfh}$. Moreover, if $ \cF_{\mfh}\subseteq \tilde{\mfk}_{\mfh}$, then $\cF_{\mfh}=\tilde{\mfk}_{\mfh}$ and a special $H$-structure $P$ on an almost Abelian Lie algebra $\mfg$ is torsion-free if and only if it is left-invariantly flat, which is the case if and only if $f\in \End(\mfu)\cong \End(\bR^{n-1})$ defined as in Proposition \ref{pro:specialHstrstf} is contained in $\tilde{\mfk}_{\mfh}$.
\end{corollary}
 We note that there are linear subgroups $H$ for which the flatness and the torsion-free condition for a (special) $H$-structure on almost Abelian Lie algebras is equivalent without having $\tilde{\mfk}_{\mfh}= \cF_{\mfh}$, cf. Example \ref{ex:tf=flatbutbigger} below. 
 
We end this section by computing $\cF_{\mfh}$ for the case of the real symplectic Lie algebra $\mathfrak{sp}(2m,\bR)$:
\begin{example}
We show that
\begin{equation*}
\cF_{\mathfrak{sp}(2m,\bR)}=\tilde{\mfk}_{\mathfrak{sp}(2m,\bR)}=\left\{\left.\begin{pmatrix} A & 0 \\
	                                                                                    w^t & a\end{pmatrix}\right|A\in \mathfrak{sp}(2m-2,\bR),\ w\in \bR^{2m-2},\ a\in \bR\right\}.
\end{equation*}
where $\mathfrak{sp}(2m,\bR)=\left\{\left. A\in \bR^{2m\times 2m}\right|A.\omega_0=0\right\}$ for $\omega_0:=e^{12}+\ldots+e^{2m-1\, 2m}$, thus reproving the well-known characterisation of symplectic almost Abelian Lie algebras, cf., e.g., \cite[Proposition 4.1]{LW}.

For this, let $\nabla\in \cT^{-1}(\End(\bR^{n-1}))\subseteq\cD_{\mfh}$ be given. Then
\begin{equation*}
\omega_0(\nabla_u e_{2m},v)=-\omega_0(e_{2m},\nabla_u v)=-\omega_0(e_{2m},\nabla_v u)=\omega_0(\nabla_v e_{2m},u)=-\omega_0(u,\nabla_v e_{2m})
\end{equation*}
for all $u,v\in \bR^{2m-1}$. Thus,
\begin{equation*}
\nabla e_{2m}=\begin{pmatrix}
	          A_2 & J_0 z_2 & x_2\\
	          w_2^t & a_2 & b_2\\
	          z_2^t & c_2 &  d_2
\end{pmatrix}
\end{equation*}
for certain $A_2\in \mathfrak{sp}(2m-2,\bR)$, $w_2,x_2,z_2\in \bR^{2m-2}$ and $a_2,b_2,c_2,d_2\in \bR$ and where $J_0$ is the standard complex structure on $\bR^{2m-2}$. As $\nabla_{e_{2m}}\in \mathfrak{sp}(2m,\bR)$ and $\nabla_{e_{2m}}-\nabla e_{2m}$ preserves $\bR^{2m-1}$, we have 
\begin{equation*}
	\nabla_{e_{2m}}=\begin{pmatrix}
		A_1 & J_0 z_2 & -J_0 w_1\\
		w_1^t & a_1 & b_1\\
		z_2^t & c_2 &  -a_1
	\end{pmatrix}
\end{equation*}
for certain $A_1\in \mathfrak{sp}(2m-2,\bR)$, $w_1\in \bR^{2m-2}$ and $a_1,b_1\in \bR$. Then
\begin{equation*}
\cT(\nabla)=\begin{pmatrix}
	A_1-A_2 & 0 \\
	(w_1-w_2)^t & a_1-a_2
\end{pmatrix}
\end{equation*}
is in $\tilde{\mfk}_{\mathfrak{sp}(2m,\bR)}$ and so $\cF_{\mathfrak{sp}(2m,\bR)}\subseteq\tilde{\mfk}_{\mathfrak{sp}(2m,\bR)}$. Hence, by Corollary \ref{co:flat=tf}, we have $\cF_{\mathfrak{sp}(2m,\bR)}=\tilde{\mfk}_{\mathfrak{sp}(2m,\bR)}$ as claimed.
\end{example}
\begin{remark}\label{re:almostnilpotent}
A natural extension of the class of almost Abelian Lie algebras is given by the class of \emph{almost nilpotent} Lie algebras $\mfg$, i.e. $\mfg$ has a nilpotent codimension one ideal $\mfn$. For this class of Lie algebras, we have the following strong sufficient condition ensuring that a given \emph{special} $H$-structure $P$ is torsion-free, where special here means again that one has an adapted basis $(X_1,\ldots,X_n)$ such that $(X_1,\ldots,X_{n-1})$ is a basis of $\mfn$:

If $(X_1,\ldots,X_{n-1})\cdot B$ is a torsion-free $B$-structure on $\mfn$, where
\begin{equation*}
	B:=\left\{\left. h|_{\bR^{n-1}}\right|h\in H,\ h(\bR^{n-1})\subseteq \bR^{n-1},\ h(e_n)=1\right\}\subseteq \GL(n-1,\bR),
\end{equation*}
and $f:=\ad(X_n)|_{\mfn}\in \mathrm{Der}(\mfn)$ may be identified via $(X_1,\ldots,X_{n-1})$ with an element in $\tilde{\mfk}_{\mfh}$, then $P$ is a torsion-free $H$-structure.

This result simply follows from that fact that if $\tilde{\nabla}$ is a torsion-free $B$-connection on $\mfn\cong \bR^{n-1}$, one easily checks that $\nabla$, defined by
\begin{equation*}
\nabla_{u_1} u_2:=\tilde{\nabla}_{u_1} u_2,\ \nabla_{u_1} X_n:=0,\ \nabla_{X_n}:=F
\end{equation*}
for all $u_1,u_2\in \bR^{n-1}$, is a torsion-free $H$-connection on $\mfg\cong \bR^n$. Here, $F\in H$ is chosen such that $F(\bR^{n-1})\subseteq \bR^{n-1}$ and $F|_{\bR^{n-1}}=f$,
\end{remark}
\begin{convention}
From now on, if we have given an almost Abelian Lie algebra $\mfg$, we always denote by $\mfu$ a codimension one Abelian ideal, fix always some $X\in \mfg\setminus \mfu$ and set $f:=\ad(X)|_{\mfu}\in \End(\mfu)$. We recall that $f$ is only defined uniquely up to non-zero scaling, i.e. as an element in $\mathbb{P}(\End(\mfu))$, however, this kind of non-uniqueness will play no role in our considerations.
\end{convention}
Now if an $H$-structure is no longer special, Proposition \ref{pro:specialHstrstf} together with Lemma \ref{le:reductiontospecialHstructure} yields the following result on the flatness and torsion-free condition for $P$:
\begin{theorem}\label{th:arbrtyHstrtfflat}
	Let $P$ be an $H$-structure of type $[U]$ on an almost Abelian Lie algebra $\mfg$. Choose $T\in \GL(n,\bR)$ with $T(U)=\bR^{n-1}$ and $u\in P$ with $u(U)=\mfu$. If we use $u\circ T^{-1}$ to identify $\bR^{n}$ with $\mfg$ and $\bR^{n-1}$ with $\mfu$, then the following assertions are true:
	\begin{enumerate}[(a)]
		\item If $f\in \tilde{\mfk}_{T\mfh T^{-1}}$, then $P$ is left-invariantly flat.
		\item $P$ is torsion-free if and only if $f\in \cF_{T\mfh T^{-1}}$.
	\end{enumerate}
\end{theorem}
\begin{remark}\label{re:nonspecial}
We notice that if we use the basis $(T^{-1}(e_1),\ldots,T^{-1}(e_{n-1}))$ to identify $U$ with $\bR^{n-1}$, then we have the identity
\begin{equation*}
\tilde{\mfk}_{T\mfh T^{-1}}=\left\{F|_U\in \End(U)\cong \bR^{(n-1)\times (n-1)}\left| F\in \mfh,\quad F(U)\subseteq U\right.\right\}.
\end{equation*}
Similarly, we may use $T$ to identify $\cD_{T\mfh T^{-1}}$ with
\begin{equation*}
\left\{\left.\nabla\in (\bR^n)^*\otimes \mfh\right| \nabla_{U\times U}\in S^2 U^*\otimes \bR^n\right\}.
\end{equation*}
Then, for any $v\in \bR^n\setminus U$, we may define the map $\cT=\cT_v:\cD_{T\mfh T^{-1}}\rightarrow \End(U)$ analogously to above and then identify $\cF_{T\mfh T^{-1}}$ with $\cT(\cT^{-1}(\End(\bR^{n-1})))$.
\end{remark}
So far we always had given a specific $H$-structure $P$ on an almost Abelian Lie algebra $\mfg$ and wanted to know if $P$ is torsion-free. One may, however not specify $P$ in front and simply ask the question whether a given almost Abelian Lie algebra $\mfg$ admits a torsion-free $H$-structure $P$ (of any type) or not. For this kind of question, we deduce from Theorem \ref{th:arbrtyHstrtfflat} the following result:
\begin{theorem}\label{th:almostAbeliantf}
	Let $\mfg$ be an almost Abelian Lie algebra. Choose a set of representatives $(U_{\alpha})_{\alpha\in A}$ for the orbits of the action of $H$ on $\mathrm{Gr}_{n-1}(\bR^{n})$. Moreover, choose for each $\alpha\in A$ some $T_{\alpha}\in \GL(n,\bR)$ with $T_{\alpha}(U_{\alpha})=\bR^{n-1}$.
	
	Then $\mfg$ admits a torsion-free $H$-structure if and only if there exists $v\in \mathrm{Iso}(\bR^n,\mfg)$ such that
	\begin{equation*}
		f\in \bigcup_{\alpha\in A} v\circ \cF_{T_{\alpha} H T_{\alpha}^{-1}}\circ v^{-1}|_{\bR^{n-1}}.
	\end{equation*}
	Moreover, if $f \in   v\circ \cF_{T_{\alpha_0} H T_{\alpha_0}^{-1}}\circ v^{-1}|_{\bR^{n-1}}$ for some $v\in \mathrm{Iso}(\bR^n,\mfg)$ and some $\alpha_0\in A$, then a torsion-free $H$-structure $P$ on $\mfg$ is given by
	\begin{equation*}
		P:=(v\circ T_{\alpha_0})\cdot H.
	\end{equation*}
\end{theorem}
\begin{proof}
If $\mfg$ admits a torsion-free $H$-structure $P$, then there is some $\alpha_0\in A$ such that $P$ is of type $[U_{\alpha_0}]$ and so Theorem \ref{th:arbrtyHstrtfflat} (b) tells us that if we choose any $u\in P \subseteq \mathrm{Iso}(\bR^n,\mfg)$ and use $v:=u\circ T_{\alpha}^{-1} \in\mathrm{Iso}(\bR^n,\mfg)$ to identify $\bR^n$ with $\mfg$ and $\bR^{n-1}$ with $\mfu$, $f$ has to lie in $\cF_{T_{\alpha_0}\mfh T_{\alpha_0}^{-1}}$, which means exactly that $v^{-1}\circ f\circ v|_{\bR^{n-1}}\in \cF_{T_{\alpha} H T_{\alpha}^{-1}}$, i.e.
 $f\in v\circ \cF_{T_{\alpha_0} H T_{\alpha_0}^{-1}}\circ v^{-1}|_{\bR^{n-1}}$.
 
Conversely, if $f\in \bigcup_{\alpha\in A} v\circ \cF_{T_{\alpha} H T_{\alpha}^{-1}}\circ v^{-1}|_{\bR^{n-1}}$,
 then there exists some $\alpha_0\in A$ such that $f\in v\circ \cF_{T_{\alpha_0} H T_{\alpha_0}^{-1}}\circ v^{-1}|_{\bR^{n-1}}$, i.e. $v^{-1}\circ f\circ v|_{\bR^{n-1}}\in  \cF_{T_{\alpha_0} H T_{\alpha_0}^{-1}}$. Now $u:=v\circ T_{\alpha_0}$ is an adapted frame of the $H$-structure $P:=(v\circ T_{\alpha_0})\cdot H$ and $u\circ T_{\alpha_0}^{-1}=v$, so that, by Theorem \ref{th:arbrtyHstrtfflat} (b), $P$ is torsion-free.
\end{proof}
\section{Subalgebras commuting with an endomorphism}\label{sec:mfhcommutingwithendo}
In this section, we consider linear subalgebras $\mfh$ which commute with an endomorphism $A$ of $\bR^n$, i.e.
\begin{equation*}
\mfh\subseteq \mathfrak{gl}(A):=\left\{F\in \mathfrak{gl}(n,\bR)\left| AF=FA\right. \right\}.
\end{equation*}
\begin{remark}
If $A$ is an endomorphism of $\bR^n$, then a $\GL(A)$-structure $P$ on an $n$-dimensional manifold $M$ is equivalent to an endomorphism field $\mathcal{A}$ on $M$ which is pointwise isomorphic to $A$. In this case, a connection $\nabla$ is compatible with $P$ if and only if $\nabla \mathcal{A}=0$. Moreover, $P$ is flat if and only if there exist local coordinates around every point identifying $\mathcal{A}$ with $A$.

Properties of flat (usually called \emph{integrable}) endomorphism fields $\mathcal{A}$ as above have been studied intensively in the literature and it has been shown that flatness is always equivalent to torsion-freeness in this situation and that flatness is also equivalent to the vanishing of the \emph{Nijenhuis tensor} $N_{\mathcal{A}}$
\begin{equation*}
N_{\mathcal{A}}(X,Y)=[\cA(X),\cA(Y)]-\cA([\cA(X),Y]+[X,\cA(Y)])-\cA^2([X,Y])
\end{equation*}
of $\cA$ together with the integrability of certain subbundles defined using the nilpotent part of $\mathcal{A}$, cf., e.g. \cite{LL} or \cite{Tho}.
\end{remark}
Many $\GL(A)$-structures have special names and we recall four of them here:
\begin{definition}
Let $\mathcal{A}$ be an endomorphism field on an $n$-dimensional manifold $M$ pointwise isomorphic to $A\in \End(\bR^n)$. We say that $\mathcal{A}$ is
\begin{itemize}
\item an \emph{almost complex structure} if $A=J_0:=\diag(M_i,\ldots, M_i)$, where $M_i:=\left(\begin{smallmatrix}0 & -1 \\ 1 & 0 \end{smallmatrix}\right)$,
\item 
an \emph{almost product structure of signature $(p,n-p)$} if $A=P_0:=\diag(I_p,I_{n-p})$ for $p,n-p\geq 1$,
\item 
an \emph{almost para-complex structure} if $n=2m$ and $\mathcal{A}$ is an almost product structure of type $(m,m)$,
\item
and an \emph{almost tangent structure} if $n=2m$ and $A=T_0:=\left(\begin{smallmatrix} 0 & 0 \\ I_m & 0 \end{smallmatrix}\right)$.
\end{itemize}
We skip the word ``almost'' if the underlying $\GL(A)$-structure $P$ is torsion-free, which, as said above, is equivalent to $P$ being flat, and in all considered cases also to the vanishing of the associated Nijenhuis tensor, i.e. to $N_{\mathcal{A}}=0$.

Note that if $M=V$ is a vector space and the tensor-field $A$ is constant, i.e. $A_v=A_w$ for all $v,w\in V$, then the torsion-freeness of the underlying $\GL(A)$-structure is trivially fulfilled and we also skip the word ``almost'' then.
\end{definition}
\begin{remark}
Note that almost complex structures may also be characterised by $\mathcal{A}^2=-\id_{TM}$, almost product structures by $\mathcal{A}^2=\id_{TM}$ and $\mathcal{A}\neq \pm \id_{TM}$ and almost tangent structures by $n=2m$, $\mathcal{A}^2=0$ and $\ker(\mathcal{A})=\mathrm{im}(\cA)$.
\end{remark}
We come now back to arbitrary $H$-structures with $H\subseteq \GL(A)$ and show now that the map $\cT$ simplifies a lot provided that $\bR^{n-1}$ is \emph{not} $A$-invariant, a condition automatically satisfied if $A$ defines an almost complex structure. We also prove that if $\mfh$ is even $A$-invariant, then we obtain $\tilde{\mfk}_{\mfh}=\cF_{\mfh}$, and so then a special $H$-structure is torsion-free if and only if it is flat if and only if $f\in \tilde{\mfk}_{\mfh}$:
\begin{theorem}\label{th:commutingendo}
Let $A$ be an endomorphism of $\bR^n$ and let $\mfh$ be a Lie subalgebra of $\mathfrak{gl}(A)$. Assume that $\bR^{n-1}$ is not $A$-invariant. Then:
\begin{enumerate}[(a)]
	\item The map $\cT=\cT_{Av}:\cD_{\mfh}\rightarrow (\bR^{n-1})^*\otimes \bR^n$ is given by
	\begin{equation*}
		\cT=\nabla_{Av}\circ \iota- A\circ \nabla_{v}\circ \iota
	\end{equation*}
	for some $v\in \bR^{n-1}$ with $Av\notin \bR^{n-1}$.
	\item 
	One has
	\begin{equation*}
		\tilde{\mfk}_\mfh\subseteq \cF_{\mfh}\subseteq \tilde{\mfk}_{\mfh+A\mfh}:=\left\{F|_{\bR^{n-1}}\left| F\in \mfh+A\mfh,\quad F(\bR^{n-1})\subseteq \bR^{n-1}\right.\right\}
	\end{equation*}
\item
If $\mfh=A\mfh$, then $\tilde{\mfk}_\mfh=\cF_{\mfh}$. Consequently, in this case, a special $H$-structure $P$ on an almost Abelian Lie algebra $\mfg$ is left-invariantly flat if and only if it is torsion-free if and only if $f\in\tilde{\mfk}_{\mfh}$.
\end{enumerate}
\begin{proof}
Let $\nabla\in \cD_{\mfh}$ be given. As $\bR^{n-1}$ is not $A$-invariant, there is some $v\in \bR^{n-1}$ with $Av\notin \bR^{n-1}$. Then, using that $\nabla_{u_1} u_2=\nabla_{u_2} u_1$ for all $u_1,u_2\in \bR^{n-1}$ and that $\nabla_w$ comutes with $A$ for all $w\in \bR^n$, we get
\begin{equation*}
\cT(\nabla)=\nabla_{Av}\circ \iota-\nabla Av\circ \iota=\nabla_{Av}\circ \iota- A\circ \nabla_{\iota(\cdot)} v=\nabla_{Av}\circ \iota- A\circ \nabla_{v} \iota(\cdot)=\nabla_{Av}\circ \iota- A\circ \nabla_{v}\circ \iota,
\end{equation*}
i.e. (a) holds. Now if $\nabla\in \cT^{-1}(\End(\bR^{n-1}))$, we have $\nabla_{Av},\nabla_v\in \mfh$ and so $\nabla_{Av}\circ \iota- A\circ \nabla_{v}\circ \iota \in \mfh+A\mfh$, which implies that $\cT(\nabla)\in \tilde{\mfk}_{\mfh+A\mfh}$, i.e. $\cF_{\mfh}\subseteq \tilde{\mfk}_{\mfh+A\mfh}$, thus proving (b). Finally, (c) follows directly from (b).
\end{proof}
\end{theorem}
\subsection{Product structures}
In this subsection, we provide a characterisation of product structures on almost Abelian Lie algebras. In fact, we show that any almost Abelian Lie algebra admits product structures of arbitrary signature. 

In order to apply our set-up to the existence question of product structures of signature $(p,n-p)$ on almost Abelian Lie algebras, we need to determine first the orbits of the $\GL(P_0)$-action on $\mathrm{Grass}_{n-1}(\bR^n)$.
\begin{lemma}\label{le:orbitsproductstr}
The $\GL(P_0)$-action on $\mathrm{Grass}_{n-1}(\bR^n)$ has three orbits represented by the following three $(n-1)$-dimensional subspaces of $\bR^n$:
\begin{equation*}
\begin{split}
U_1&:=\bR^{n-1},\quad U_2:=\spa{e_1,\ldots,e_{p-1},e_{p+1},\ldots, e_n},\\ U_3&:=\spa{e_1,\ldots,e_{p-1},e_{p+1},\ldots,e_{n-1},e_p+e_n}
\end{split}
\end{equation*}
\end{lemma}
\begin{proof}
Let $U$ be an $(n-1)$-dimensional subspace of $\bR^n$. Note that the $\GL(P_0)$-action preserves the dimensions $d_{\pm}=\dim(U\cap (\bR^n)_{\pm})$, where $\bR^n_{\pm}$ is the $\pm$-eigenspace of $P_0$. Consequently, $U_1$, $U_2$ and $U_3$ define different orbits. Moreover, we know that $d_+\in \{p-1,p\}$, $d_-\in \{n-p-1,n-p\}$ and $d_+ + d_-\leq n-1$. Thus, $(d_+,d_-)\in \{(p,n-p-1),(p-1,n-p), (p-1,n-p-1)\}$ and we only need to show that, depending on the numbers $(d_+,d_-)$, the subspace $U_1$, $U_2$ or $U_3$, respectively, lie in the $\GL(P_0)$-orbit of $U$. For this goal, observe that
\begin{equation*}
\GL(P_0)=\left\{\left. \diag(A,B) \right|A\in \GL(p,\bR),\ B\in \GL(n-p,\bR)\right\}.
\end{equation*}
Now we consider the different possible values for $(d_+,d_-)$:

Assume first that $(d_+,d_-)=(p,n-p-1)$. Then $\bR^n_+=\spa{e_1,\ldots,e_p}\subseteq U$ and we may use the $\GL(P_0)$-action to bring the $n-p-1$-dimensional subspace $U\cap \bR^n_-$ of $\bR^n_-=\spa{e_{p+1},\ldots,e_n}$ to the space $\spa{e_{p+1},\ldots,e_{n-1}}$ and so $U_1$ is, in this case in the same orbit as $U$.

If $(d_+,d_-)=(p-1,n-p)$, then we may argue similarly to show that $U_2=\spa{e_2,\ldots,e_n}$ is in the same orbit as $U$.

Finally, if $(d_+,d_-)=(p-1,n-p-1)$, we may use the $\GL(P_0)$-action to assume that $U\cap \bR^{n}_{+}=\spa{e_1,\ldots,e_{p-1}}$ and $U\cap \bR^{n}_{-}=\spa{e_{p+1},\ldots,e_{n-1}}$. Thus,
$U=\spa{e_1,\ldots,e_{p-1},e_{p+1},\ldots,e_{n-1},\lambda e_p+\mu e_n}$ for certain $\lambda,\mu\in \bR^*$ and we may apply $A=\diag\left(I_{p-1},\tfrac{1}{\lambda},I_{n-p-1},\tfrac{1}{\mu}\right)\in \GL(P_0)$ to bring $U$ to $U_3$.
\end{proof}
We now obtain the following characterisation of product structures on almost Abelian Lie algebras:
\begin{theorem}\label{th:productstr}
	Let $\mfg$ be an $n$-dimensional almost Abelian Lie algebra. Then $\mfg$ admits a product structure of signature $(p,q)$ ($p+q=n$) and of type $[U_1]$, $[U_2]$ or $[U_3]$, respectively, if and only if there exists a basis $(X_1,\ldots,X_{n-1})$ of $\mfu$ such that with respect to that basis $f$ equals
	\begin{equation*}
		\begin{pmatrix} A_1 & 0 \\ B_1 & C_1 \end{pmatrix},\quad \begin{pmatrix} A_2 & B_2 \\ 0 & C_2 \end{pmatrix}\quad \mathrm{or}\quad	\begin{pmatrix} A_3 & 0 & u_3 \\
			0 & B_3 & w_3\\
			0& 0 & a_3 
		\end{pmatrix}
	\end{equation*}
	respectively, where $A_1\in \bR^{p\times p}$, $B_1\in \bR^{(q-1)\times p}$, $C_1\in \bR^{(q-1)\times (q-1)}$, $A_2\in \bR^{(p-1)\times (p-1)}$, $B_2\in \bR^{(p-1)\times q}$, $C_2\in \bR^{q\times q}$, $A_3\in \bR^{(p-1)\times (p-1)}$, $B_3\in \bR^{(q-1)\times (q-1)}$, $u_3\in \bR^{p-1}$, $w_3\in \bR^{q-1}$ and $a_3\in \bR$ are arbitrary. In these cases, $(X_1,\ldots,X_{n-1}, X)$,\linebreak $(X_1,\ldots,X_{p-1},X,X_{p+1},\ldots,X_{n-1})$ or $(X_1,\ldots,X_{p-1},X_{p+1},\ldots,X_{n-1},X_p+X)$,
	respectively, is an adapted basis for a product structure of signature $(p,n-p)$ and type $[U_1]$, $[U_2]$ or $[U_3]$, respectively.
\end{theorem}
\begin{proof}
For notational simplicity, we set $H:=\GL(P_0)$.

We start by considering the case that $P$ has type $[U_3]$ noting that this is the only case for which $U_j$ is not $P_0$-invariant. In this case, we define $T_3\in \End(\bR^n)$ by $T_3(e_i):=e_i$ for $i=1,\ldots,p-1,p+1,\ldots,n$ and $T_3(e_p+e_n):=e_p$ and note that then $T_3(U_3)=\bR^{n-1}$. Then, with respect to the basis 
\begin{equation*}
(e_1,\ldots,e_{p-1},e_{p+1},\ldots,e_{n-1},e_p+e_n,e_n),
\end{equation*}
the elements in $\mfh$ are exactly those $n\times n$-matrices which are of the form
\begin{equation*}
\begin{pmatrix} A_3 & 0 & u_3 &  0\\
	            0 & B_3 & w_3 & w_3\\
	            v_3^T & 0 & a_3 & 0 \\
	            -v_3^T & x_3^T & b_3-a_3 & b_3
\end{pmatrix}	            
\end{equation*}
for certain $A_3\in \bR^{(p-1)\times (p-1)}$, $u_3,v_3\in \bR^{p-1}$, $B_3\in \bR^{(q-1)\times (q-1)}$,  $w_3,x_3\in \bR^{q-1}$ and $a_3, b_3\in \bR$. Hence, using Remark \ref{re:nonspecial}, we see that the elements in $\tilde{\mfk}_{T_3\mfh T_3^{-1}}$ are of the form
\begin{equation*}
	\begin{pmatrix} A_3 & 0 & u_3 \\
		0 & B_3 & w_3\\
		0& 0 & a_3 
	\end{pmatrix}	            
\end{equation*}
with respect to the basis $(e_1,\ldots,e_{p-1},e_{p+1},\ldots,e_{n-1},e_p+e_n)$ of $U$ and so Theorem \ref{th:commutingendo} (c) yields the result in that case.

Next, we look at type $[U_1]$ and have to determine the space $\cF_{T_1 \mfh T_1^{-1}}=\cF_{\mfh}$ here. So let $\nabla\in \cD_{\mfh}$ be given. Since $\nabla_{e_n}\in \mfh$, we have
\begin{equation*}
	\nabla_{e_n}=\begin{pmatrix} A & 0 & 0 \\
		0 & \tilde{C} & v \\ 
		0 & w^t & b
	\end{pmatrix}
\end{equation*}
for certain $A\in \bR^{p\times p}$, $\tilde{C}\in \bR^{(q-1)\times (q-1)}$, $v,w\in \bR^{q-1}$, $b\in \bR$.
Next, let $Y\in \bR^n$ be arbitrary. Since $\nabla_Y$ commutes with $P_0$, $\nabla_Y$ preserves both $\bR^n_+$ and $\bR^n_-$ and so $\nabla_Y e_n\in \bR^n_-$. Consequently,
\begin{equation*}
	\nabla e_n=\begin{pmatrix}  0 & 0 & 0 \\
		\tilde{B} & \hat{B} & u \\ 
		x^t & y^t & c
	\end{pmatrix}
\end{equation*}
for certain $\tilde{B}\in \bR^{(q-1)\times p}$, $\hat{B}\in \bR^{(q-1)\times (q-1)}$, $x\in \bR^p$, $y,u\in \bR^{q-1}$, $c\in \bR$. Hence,
\begin{equation*}
\cT(\nabla)=
	\begin{pmatrix}
	 A & 0 \\
	 \tilde{B} & \tilde{C}-\hat{B} \\
	 x^t & w^t-y^t
\end{pmatrix},
\end{equation*}
so that
\begin{equation*}
\cF_{\mfh}\subseteq \left\{\left.\begin{pmatrix} A_1 & 0 \\ B_1 & C_1 \end{pmatrix}\right|A_1\in \bR^{p\times p}, \, B_1\in \bR^{(q-1)\times p}, \, C_1\in \bR^{(q-1)\times (q-1)} \right\}.
\end{equation*}
Conversely, if
\begin{equation*}
F:=\begin{pmatrix} A_1 & 0 \\ B_1 & C_1 \end{pmatrix}
\end{equation*}
with $A_1\in \bR^{p\times p}$, $B_1\in \bR^{(n-p-1)\times p}$ and $C_1\in \bR^{(n-p-1)\times (n-p-1)}$ is given,
we define $\nabla\in (\bR^n)^*\otimes \mfh$ by $\nabla_u v:=0$ for all $u\in \bR^{n-1}$, 
\begin{equation*}
\nabla_{e_i} e_n:=\left(\begin{smallmatrix}  0 \\ -B_1 e_i \\ 0 \end{smallmatrix}\right),\quad 
\nabla_{e_j} e_n:=\left(\begin{smallmatrix}  0 \\ -C_1 e_{j-p} \\ 0 \end{smallmatrix}\right),\quad
\nabla_{e_n} e_i:=\left(\begin{smallmatrix}  A_1 e_i \\ 0  \\ 0 \end{smallmatrix}\right),\quad \nabla_{e_n} e_j:=0,
\end{equation*}
for $i=1,\ldots,p$, $j=p+1,\ldots,n-1$ and $\nabla_{e_n} e_n:=0$. Then $\nabla\in \cD_{\mfh}$ and
\begin{equation*}
\cT(\nabla)=F\in \End(\bR^{n-1})
\end{equation*}
which proves the assertion in this case.

Finally, the assertion for type $[U_2]$ follows directly from the characterisation for type $[U_1]$ since $P$ is a product structure of signature $(p,q)$ and type $[U_2]$ if and only if $-P$ is a product structure of signature $(q,p)$ and type $[U_1]$.
\end{proof}
As an immediate corollary of Theorem \ref{th:productstr}, we get that any almost Abelian Lie algebra admits a product structure:
\begin{corollary}\label{co:alwaysproduct}
	An almost Abelian Lie algebra $\mfg$ admits product structure $P$ of any possible signature $(p,q)$ with $\mfu$ being $P$-invariant. In particular, any even-dimensional almost Abelian Lie algebra $\mfg$ admits a para-complex structure $P$ with $\mfu$ being $P$-invariant.
\end{corollary}
\begin{proof}
Consider first the case that $p$ is even. In this case, we use the lower block triangular real Jordan normal form of $f$. Putting first at the upper left side, the Jordan blocks corresponding to pairs of non-real complex-conjugated eigenvalues and then the ones corresponding to real eigenvalues, the Jordan normal form of $f$ is of the form
\begin{equation*}
\begin{pmatrix} A_1 & 0 \\ B_1 & C_1 \end{pmatrix}
\end{equation*}
for certain $A_1\in \bR^{p\times p}$, $B_1\in \bR^{(n-p-1)\times (n-p)}$, $C_1\in \bR^{(n-p-1)\times (n-p-1)}$ and so $\mfg$ admits a product structure of signature $(p,n-p)$ and type $[U_1]$, and so with $P$-invariant $\mfu$, by Theorem \ref{th:productstr}.

Next, consider the case that $p-1$ is even. In this case, we use the upper block triangular real Jordan normal form of $f$. Putting again first at the upper left, the Jordan blocks corresponding to pairs of non-real complex-conjugated eigenvalues and then the Jordan blocks corresponding to real eigenvalues, the Jordan normal form of $f$ is of the form
 \begin{equation*}
	\begin{pmatrix} A_2 & B_2 \\ 0 & C_2 \end{pmatrix}
\end{equation*}
for certain $A_2\in \bR^{(p-1)\times (p-1)}$, $B_2\in \bR^{(p-1)\times (n-p)}$, $C_2\in \bR^{(n-p)\times (n-p)}$. Thus, in this case, Theorem \ref{th:productstr} yields the existence of a product structure $P$ of signature $(p,n-p)$ of type $[U_2]$, and so with $P$-invariant $\mfu$.
\end{proof}
\begin{remark}
In \cite{ABDO}, the authors classified all four-dimensional solvable Lie algebras admitting a para-complex structure. In particular, they showed that all four-dimensional almost Abelian Lie algebras admit a para-complex structure. In this sense, Corollary \ref{co:alwaysproduct} extends the result in the almost Abelian case to any (even) dimension.
\end{remark}
We note that almost product structures provide examples of $H$-structures on almost Abelian Lie algebras for which torsion-freeness is the same as flatness but $\cF_{\mfh}\neq \tilde{\mfk}_{\mfh}$:
\begin{example}\label{ex:tf=flatbutbigger}
Since in general an almost product structure is flat if and only if it is torsion-free, we know that a special $\GL(P_0)$-structure of signature $(p,n-p)$ on an almost Abelian Lie algebra is torsion-free if and only if it is flat. However, in this case, we have
\begin{equation*}
\begin{split}
\tilde{\mfk}_{\mfh}&=\left\{\left.\begin{pmatrix} A & 0 \\ 0 & C\end{pmatrix}\right|A\in \GL(p,\bR),\ C\in \GL(n-p-1,\bR)\right\}\\
&\neq \left\{\left.\begin{pmatrix} A & 0 \\ B & C \end{pmatrix}\right|A\in \GL(p,\bR),\ B\in \bR^{(n-p-1)\times p},\ C\in \GL(n-p-1,\bR)\right\}=\cF_{\mfh}
\end{split}
\end{equation*}
showing that the condition $\tilde{\mfk}_{\mfh}=\cF_{\mfh}$ is not necessary for torsion-freeness being equivalent to flatness.
\end{example}
\subsection{Tangent structures}
We start with the determination of all $\GL(T_0)$-orbits in $\mathrm{Grass}_{2m-1}(\bR^{2m})$:
\begin{lemma}\label{le:orbitstangentstrr}
	The $\GL(T_0)$-action on $\mathrm{Grass}_{2m-1}(\bR^{2m})$ has two orbits represented by the following two $(2m-1)$-dimensional subspaces of $\bR^{2m-1}$:
	\begin{equation*}
			U_1:=\bR^{2m-1},\quad U_2:=\spa{e_1,\ldots,e_{m-1},e_{m+1},\ldots,e_{2m}}
	\end{equation*}
\end{lemma}
\begin{proof}
Let $U$ be an $(2m-1)$-dimensional subspace of $\bR^{2m}$ and observe that $d_U:=\dim(\ker(T_0)\cap U)\in \{m-1,m\}$ is an invariant of the $\GL(T_0)$-action and that $d_{U_1}=m-1$ while $d_{U_2}=m$. Moreover, observe that
\begin{equation*}
\GL(T_0)=\left\{\left.\begin{pmatrix} A & 0 \\ B & A\end{pmatrix}\right|A\in \GL(m,\bR),\quad B\in \bR^{m\times m}\right\}.
\end{equation*}
Now consider first the case $d_U=m-1$ and let $W$ be an $m$-dimensional complementary subspace of $\ker(T_0)\cap U$ in $U$. As $W\cap \ker(T_0)=\{0\}$, we may use the $\GL(T_0)$-action to bring $W$ to the subspace $\spa{e_1,\ldots,e_m}$ and may then use an element in $\GL(T_0)$ with $B=0$ to bring $U\cap \ker(T_0)$ to $\spa{e_{m+1},\ldots,e_{m-1}}$. Hence, then $U=\spa{e_1,\ldots,e_{2m-1}}=\bR^{2m-1}=U_1$.

If $d_U=m$, we have $\ker(T_0)=\ker(T_0)\cap U$ and there exists an $(m-1)$-dimensional complementary subspace $W$ of $\ker(T_0)$ in $U$. But then we may use the action of $\GL(T_0)$ to bring $W$ to the subspace $\spa{e_1,\ldots,e_{m-1}}$ and so have $U=\spa{e_1,\ldots,e_{m-1},e_{m+1},\ldots,e_{2m}}=U_2$.
\end{proof}
\begin{theorem}\label{th:tangentstr}
	Let $\mfg$ be an $n$-dimensional almost Abelian Lie algebra and $X\in \mfg\backslash \mfu$ arbitrary. Then $\mfg$ admits a tangent structure of type $[U_1]$ or $[U_2]$, respectively, if and only if there exists a basis $(X_1,\ldots,X_{n-1})$ of $\mfu$ such that with respect to that basis $f$ equals
	\begin{equation*}
		\begin{pmatrix} A_1 & v_1 & 0 \\ 0 & a_1 & 0 \\ B_1 & w_1 & A_1 \end{pmatrix}\ \mathrm{ or } \ \begin{pmatrix}
			A_2 &  0 & v_2 \\
			B_2 &  C_2 & w_2
		\end{pmatrix}
	\end{equation*}
	respectively, where $A_1\in \bR^{(m-1)\times (m-1)}$, $B_1\in \bR^{(m-1)\times (m-1)}$, $v_1,w_1\in \bR^{m-1}$, $a_1\in \bR$, $A_2\in \bR^{(m-1)\times (m-1)}$, $B_2,C_2\in \bR^{m\times (m-1)}$, $v_2\in \bR^{m-1}$, $w_2\in \bR^m$ are arbitrary. In these cases, $(X_1,\ldots,X_{2m-1}, X)$ or $(X_1,\ldots,X_{m-1},X,X_{m+1},\ldots,X_{2m-1})$,
	respectively, is an adapted basis for a tangent structure of type $[U_1]$ or $[U_2]$, respectively.
\end{theorem}
\begin{proof}
Let $H:=\GL(T_0)$. Note first that $U_1=\bR^{2m-1}$ is not $T_0$-invariant and so Theorem \ref{th:commutingendo} (c) yields that $\mfg$ admits a tangent structure of type $[U_1]$, i.e a special tangent structure if and only if $f$ may be identified with an element in $\tilde{\mfk}_{\mfh}$ with respect to some basis of $\mfu$. This is the claimed assertion in this case.

So let us now consider the case of type $[U_2]$ and let $\nabla\in \cD_{T_2\mfh T_2^{-1}}$ for $T_2\in \GL(2m,\bR)$ defined by $T_2(e_i):=e_i$ for $i=1,\ldots,m-1$, $T_2(e_j):=e_{j-1}$ for $j=m+1,\ldots,2m-1$, $T_2(e_m):=e_{2m}$ and $T_2(e_{2m}):=e_{2m-2}$. We use the identifications mentioned in Remark \ref{re:nonspecial} and observe first that $e_m\notin U_2$ and that $\nabla_{e_m}\in \mfh$ preserves $\ker(T_0)\subseteq U_2$. Thus, in particular, $\nabla_{e_m}$ maps $\spa{e_{m+1},\ldots,e_{2m-1}}$ into $\ker(T_0)$. Now we argue that also $\nabla e_m$ maps $\spa{e_{m+1},\ldots,e_{2m-1}}$ into $\ker(T_0)$. To show this assertion, let $u\in  \spa{e_{m+1},\ldots,e_{2m-1}}\subseteq U_2$. Then there is some $\tilde{u}\in \spa{e_1,\ldots,e_{m-1}}\subseteq U_2$ such that $T_0 \tilde{u}=u$. Hence,
\begin{equation*}
T_0\nabla_u e_m=\nabla_u e_{2m}=\nabla_{e_{2m}} T_0 \tilde{u}=T_0 \nabla_{e_{2m}} \tilde{u}=T_0 \nabla_{\tilde{u}} e_{2m}=T_0 \nabla_{\tilde{u}} T_0 e_{m}= T_0^2 \nabla_{\tilde{u}} e_{m}=0.
\end{equation*} 
Thus, $\nabla_u e_m\in \ker(T_0)$ and so any element in $\cT(\cT^{-1}(\End(\bR^{n-1})))$ for $\cT=\cT_{e_m}$ has the form
\begin{equation*}
\begin{pmatrix}
	A_2 &  0 & v_2 \\
	B_2 &  C_2 & w_2
\end{pmatrix}
\end{equation*}
for certain $A_2\in \bR^{(m-1)\times (m-1)}$, $B_2,C_2\in \bR^{m\times (m-1)}$, $v_2\in \bR^{m-1}$, $w_2\in \bR^m$
with respect to the basis $(e_1,\ldots,e_{m-1},e_{m+1},\ldots,e_{2m-1},e_{2m})$ of $U$.

Now let 
\begin{equation*}
	F:=\begin{pmatrix}
		A_2 &  0 & v_2 \\
		B_2 &  C_2 & w_2
	\end{pmatrix}
\end{equation*}
for certain $A_2\in \bR^{(m-1)\times (m-1)}$, $B_2,C_2\in \bR^{m\times (m-1)}$, $v_2\in \bR^{m-1}$, $w_2\in \bR^m$ be given. Define $\nabla\in \cD_{T_2 \mfh T_2^{-1}}$ by $\nabla_u v:=0$ for all $u,v\in \spa{e_1,\ldots,e_{m-1},e_{m+1},\ldots,e_{2m-1}}$,
\begin{equation*}
\nabla_{e_i} e_m:= \left(\begin{smallmatrix} -A_2 e_i \\ 0 \\ 0  \end{smallmatrix}\right),\  \nabla_{e_i} e_{2m}:=T_0 \nabla_{e_i} e_m,\
\nabla_{e_j} e_m:= \left(\begin{smallmatrix} 0 \\ 0 \\ -C_2 e_{j-m}  \end{smallmatrix}\right),\ \nabla_{e_j} e_{2m}:=0
\end{equation*}
for $i=1,\ldots,m-1$, $j=m+1,\ldots,2m-1$,
\begin{equation*}
\nabla_{e_{2m}} e_m :=  \left(\begin{smallmatrix} -v_2 \\ 0 \\ -w_2  \end{smallmatrix}\right),\quad \nabla_{e_{2m}} e_{2m}:=T_0 \nabla_{e_{2m}} e_m,\quad \nabla_{e_m} e_i:= \left(\begin{smallmatrix} 0 \\ 0 \\ B_2 e_i  \end{smallmatrix}\right)
\end{equation*}
for  $i=1,\ldots,m-1$ and $\nabla_{e_m} e_l:=0$ for $l=m,\ldots,2m$. Then $\cT(\nabla)=F$, which shows the assertion in the type $[U_2]$ case and finishes the proof.
\end{proof}
Theorem \ref{th:tangentstr} implies that a $2m$-dimensional almost Abelian Lie algebra admits a tangent structure $T$ with $T$-invariant codimension one Abelian ideal $\mfu$ if and only if $f$ is a block triangular matrix with one diagonal block of size $m$ and one of size $m-1$. Hence, argueing similarly to the proof of Corollary \ref{co:alwaysproduct}, we get:
\begin{corollary}\label{co:alwaystangent}
Any even-dimensional almost Abelian Lie algebra $\mfg$ admits a tangent structure $T$ with $T$-invariant codimension one Abelian ideal $\mfu$.
\end{corollary}
\subsection{Complex structures}
\subsubsection{Complex subalgebras}
Finally, we consider the case that $A$ is a complex structure on $\bR^{2m}$ and then write $J$ instead of $A$. 
We begin with \emph{complex subalgebras}:
\begin{definition}
Let $J$ be a complex structure on $\bR^{2m}$. A $J$-invariant subalgebra $\mfh$ of $\mathfrak{gl}(2m,\bR)$ is called a \emph{complex subalgebra}. We note that complex subalgebras are exactly the $\bR$-Lie subalgebras of $\mathfrak{gl}(2m,\bR)$ which also have a compatible structure as a Lie algebra over $\bC$.
\end{definition}
For complex subalgebras, Theorem \ref{th:commutingendo} (c) yields the following result:
\begin{corollary}\label{co:complexsubalgebra}
	Let $\mfh$ be a complex subalgebra. Then an $H$-structure $P$ on $\mfg$ is flat if and only if it is torsion-free, which is the case if and only if for some $u\in P$ and some $T\in \GL(2m,\bR)$ with $(u\circ T^{-1})(\bR^{n-1})=\mfu$, one has $f\in \tilde{\mfk}_{T \mfh T^{-1}}$.
	\end{corollary}
Let us give a few examples, some which were known before and so we provide here a new proof of the characterisation of the torsion-free condition in these cases:
\begin{example}\label{ex:complexsubalg}
	\begin{itemize}
		\item[(a)]
		We start by taking $\mfh=\mathfrak{gl}(J_0)=:\mathfrak{gl}(m,\bC)$. Noting that $H=\GL(m,\bC)$ acts transitively on $\mathrm{Grass}_{2m-1}(\bR^{2m})$, any $\GL(m,\bC)$-structure is special. Observing that
		\begin{equation*}
			\cF_{\mathfrak{gl}(m,\bC)}=\tilde{\mfk}_{\mathfrak{gl}(m,\bC)}=\left\{\left.\begin{pmatrix} A & v \\ 0 & a \end{pmatrix}\right|A\in \mathfrak{gl}(m-1,\bC),\ v\in \bR^{2m-2},\ a\in \bR\right\},
		\end{equation*}
	 Corollary \ref{co:complexsubalgebra} gives back the characterisation of integrable almost complex structures on almost Abelian Lie algebras from \cite[Lemma 6.1]{LRV}.
		\item[(b)] $\mfh=\mathfrak{sl}(m,\bC)$: Note that a torsion-free $\mathrm{SL}(m,\bC)$-structure is nothing but a pair $(J,\nu)$ consisting of an integrable almost complex structure $J$ and a holomorphic volume form $\nu$. Moreover, observing that again $H=\mathrm{SL}(m,\bC)$ acts transitively on $\mathrm{Grass}_{2m-1}(\bR^{2m})$, Corollary \ref{co:complexsubalgebra} implies that an $\mathrm{SL}(m,\bC)$-structure $(J,\nu)$ on an almost Abelian Lie algebra $\mfg$ is torsion-free if and only if
		\begin{equation*}
			f\in \tilde{\mfk}_{\mathfrak{sl}(m,\bC)}=\left\{\left.\begin{pmatrix} A & v \\ 0 & -\tr(A) \end{pmatrix}\right|A\in \mathfrak{gl}(m-1,\bC),\quad \tr(JA)=0,\quad \ v\in \bR^{2m-2}\right\}.
		\end{equation*}
		\item[(c)]
		Next, let $m=2k$, i.e. $n=4k$, and take $\mfh=\mathfrak{sp}(J_0,\omega_0)=:\mathfrak{sp}(2k,\bC)$ with
		\begin{equation*}
			\omega_0:=\sum_{i=1}^k e^{4k-3}\wedge e^{4k}+ e^{4k-2}\wedge e^{4k-1},
		\end{equation*}
		i.e. with $(J_0,\omega_0)$ being the standard complex symplectic structure on the vector space $\bR^{4k}$.
		
		 Again, the group $H=\mathrm{Sp}(2k,\bC)$ acts transitively on $\mathrm{Grass}_{4k-1}(\bR^{4k})$ and here we have
		\begin{equation*}
			\begin{split}
			\cF_{\mathfrak{sp}(2k,\bC)}&=\tilde{\mfk}_{\mathfrak{sp}(2k,\bC)}\\
			&=\left\{\left.\begin{pmatrix} A & 0 & 0 & v \\ \omega(Ju,\cdot)  & a & 0 & b \\ \omega(u,\cdot) & 0 & a & c \\ 0 & 0 & 0 & -a \end{pmatrix}\right|A\in \mathfrak{sp}(2k-2,\bC),\ v\in \bR^{4k-4},\ a,b,c\in \bR\right\},
			\end{split}
		\end{equation*}
	
		Hence, Corollary \ref{co:complexsubalgebra} yields the characterisation of complex symplectic structures on almost Abelian Lie algebras from \cite[Theorem 3.10]{BFrLT}.
		\item[(d)]
		We note that Corollary \ref{co:complexsubalgebra} also applies to the complex pseudo-Riemannian Berger holonomy algebras, i.e. $\mathfrak{so}(n,\bC)$, $(\mfg_2)_{\bC}$, $\mathfrak{sp}(p,\bC)\oplus \mathfrak{sl}(2,\bC)$ and $\mathfrak{spin}(7)_{\bC}$ and gives, in particular, that a pseudo-Riemannian metric with holonomy in one of these groups is automatically flat.
	\end{itemize}
\end{example}
\subsubsection{Totally real subalgebras}
Here, we consider subalgebras of $\mathfrak{gl}(J)$ for some complex structure $J$ on $\bR^{2m}$ which are totally real in the following sense:
\begin{definition}
Let $J$ be a complex structure on $\bR^{2m}$. A real subalgebra $\mfh$ of $\mathfrak{gl}(J)$ is called \emph{totally real} if $\mfh\cap J\mfh=\{0\}$.
\end{definition}
We will give some important classes of examples of totally real subalgebras at the end of this subsection and in the following subsections. Before, we do this, we will provide a description of $\cF_{\mfh}$ for (almost) any totally real subalgebra $\mfh$ and for this and also later use will need to introduce a \emph{tableau} $\cK_{\mfh}$ naturally associated to $\mfh$, where \emph{tableau} simply means a linear subspace of some space of homomorphisms, and of the \emph{first prolongation} $\cK_{\mfh}^{(1)}$ of $\cK_{\mfh}$:
\begin{definition}
Let $\mfh$ be a subalgebra of $\mathfrak{gl}(n,\bR)$. The \emph{tableau} $\cK:=\cK_{\mfh}$ \emph{associated to} $\mfh$ is the linear subspace of $\Hom(\bR^{n-1},\bR^n)$ defined by
\begin{equation*}
 \cK_{\mfh}:=\left\{\left. F|_{\bR^{n-1}}\right|F\in \mfh\right\}.
\end{equation*}
The \emph{first prolongation} $\cK^{(1)}:=\cK_{\mfh}^{(1)}$ of $\cK_{\mfh}$ is defined as
\begin{equation*}
\cK_{\mfh}^{(1)}:=\left((\bR^{n-1})^*\otimes \cK\right)\cap (S^2 (\bR^{n-1})^*\otimes \bR^{n-1}).
\end{equation*}
\end{definition}
The importance of the first prolongation $\cK^{(1)}$ of the associated tableau stems from the following observation:
\begin{lemma}\label{le:cDincK1}
Let $\nabla\in \cD_{\mfh}$ for some subalgebra $\mfh$ of $\mathfrak{gl}(n,\bR)$. Then $\nabla|_{\bR^{n-1}\times \bR^{n-1}}\in \cK^{(1)}_{\mfh}$.
\end{lemma}
In the case of a totally real subalgebra $\cK^{(1)}_{\mfh}$ has a very easy structure:
\begin{lemma}\label{le:cK1totallyreal}
Let $\mfh\subseteq \mathfrak{gl}(J)$ be a totally real subalgebra, where $J$ is a complex structure on $\bR^{2m}$. Then
\begin{equation*}
\cK^{(1)}_{\mfh}\subseteq S^2 (\bR^{2m-1}_J)^0\otimes \bR^n,
\end{equation*}
where $\bR^{2m-1}_J:=\bR^{2m-1}\cap J\bR^{2m-1}$ and $(\bR^{2m-1}_J)^0\subseteq (\bR^{2m-1})^*$ is the one-dimensional annihilator of $\bR^{2m-1}_J$.
\end{lemma}
\begin{proof}
Let $\tilde{\nabla}\in \cK^{(1)}_{\mfh}$ be given. Then there exists $\hat{\nabla}\in (\bR^{2m-1})^*\otimes \mfh$ with
$\tilde{\nabla} w=\hat{\nabla} w$ for all $w\in \bR^{2m-1}$. Let now $u\in (\bR^{2m-1})_J$ and $w\in \bR^{2m-1}$ be given. Then $Ju\in \bR^{2m-1}$ as well and we get
\begin{equation*}
\hat{\nabla}_{u} w=\hat{\nabla}_w u=-\hat{\nabla}_w J^2u=-J\hat{\nabla}_w Ju=-J\hat{\nabla}_{Ju} w,
\end{equation*}
which shows that $\hat{\nabla}_u=-J\hat{\nabla}_{Ju}$ and so $\hat{\nabla}_{u}\in \mfh\cap J\mfh=\{0\}$. Thus,
$\tilde{\nabla}_u=0$, which for symmetry reasons, shows $\cK^{(1)}_{\mfh}\subseteq S^2 (\bR^{2m-1}_J)^0\otimes \bR^n$.
\end{proof}
For the formulation of the result for $\cF_{\mfh}$ for a totally real subalgebra $\mfh$, we need to consider certain subalgebras of $\mfh$ which all vanish on $(\bR^{2m-1})_J$:
\begin{definition}
Let $J$ be a complex structure on $\bR^{2m}$ and $\mfh\subseteq \mathfrak{gl}(J)$ be a linear subalgebra. Then we set
\begin{equation*}
\begin{split}
\mfh_2&:=\left\{F\in \mfh\left|F|_{\bR^{2m-1}_J}=0\right.\right\},\quad \mfh_2^{\bR^{2m-1}}:=\left\{F\in \mfh\left|F|_{\bR^{2m-1}_J}=0,\ F(\bR^{2m-1})\subseteq \bR^{2m-1}\right.\right\},\\
\mfh_2^{\bR^{2m-1}_J}&:=\left\{F\in \mfh\left|F|_{\bR^{2m-1}_J}=0,\ F(\bR^{2m-1})\subseteq \bR_J^{2m-1} \right.\right\}.
\end{split}
\end{equation*}
Note that we have the inclusions
\begin{equation*}
\mfh_2^{\bR^{2m-1}_J}\subseteq  \mfh_2^{\bR^{2m-1}}\subseteq \mfh_2
\end{equation*}
and that the dimension may rise at most by $1$ in each inclusion. We further note that $J\mfh_2^{\bR^{2m-1}_J}$ preserves the subspace $\bR^{2m-1}$ and so $J\mfh_2^{\bR^{2m-1}_J}|_{\bR^{2m-1}}$ is a subspace of $\End(\bR^{2m-1})$.

Finally, we call $\mfh$ 
\begin{itemize}
\item \emph{of type $(I)$} if $\mfh_2=\mfh_2^{\bR^{2m-1}_J}$,
\item \emph{of type $(II)$} if $\mfh_2^{\bR^{2m-1}_J}\neq \mfh_2^{\bR^{2m-1}}$ and $\mfh_2^{\bR^{2m-1}}= \mfh_2$,
\item \emph{of type $(III)$} if $\mfh_2^{\bR^{2m-1}_J}= \mfh_2^{\bR^{2m-1}}$ and $\mfh_2^{\bR^{2m-1}}\neq \mfh_2$
\item and \emph{of type $(IV)$} if $\mfh_2^{\bR^{2m-1}_J}\neq\mfh_2^{\bR^{2m-1}}$ and $\mfh_2^{\bR^{2m-1}}\neq \mfh_2$.
\end{itemize}
\end{definition}
We are now able to prove the following result on $\cF_{\mfh}$ in the totally real case:
\begin{theorem}\label{th:totallyreal}
Let $J$ be a complex structure on $\bR^{2m}$ and $\mfh\subseteq \mathfrak{gl}(J)$ be a totally real linear subalgebra. Then:
\begin{enumerate}[(a)]
	\item If $\mfh$ is of type $(I)$, then 
	\begin{equation*}
		\cF_{\mfh}=\tilde{\mfk}_{\mfh}\oplus \left.J\mfh_2^{\bR^{2m-1}_J}\right|_{\bR^{2m-1}}.
	 \end{equation*}
	\item If $\mfh$ is of type $(II)$ and additionally any $F\in \mfh$ with $F(\bR^{2m-1}_J)\subseteq \bR^{2m-1}$ satisfies $F(\bR^{2m-1})\subseteq \bR^{2m-1}$, then 
		\begin{equation*}
		\cF_{\mfh}=\tilde{\mfk}_{\mfh}\oplus  \left.J\mfh_2^{\bR^{2m-1}_J}\right|_{\bR^{2m-1}}.
	\end{equation*}
    \item If $\mfh$ is of type $(III)$, then
    \begin{equation*}
    	\cF_{\mfh}=\left(\tilde{\mfk}_{\mfh}\oplus \left.J\mfh_2^{\bR^{2m-1}_J}\right|_{\bR^{2m-1}}\right)+\spa{(JF-\lambda F)|_{\bR^{2m-1}}}
    \end{equation*}
for $F\in \mfh_2\setminus\mfh_2^{\bR^{2m-1}}$ and $\lambda\in \bR$ such that $(JF-\lambda F)(\bR^{2m-1})\subseteq \bR^{2m-1}$.
	\item If $\mfh$ is of type $(IV)$, then
	\begin{equation*}
	\cF_{\mfh}=\left(\tilde{\mfk}_{\mfh}\oplus \left.J\mfh_2^{\bR^{2m-1}_J}\right|_{\bR^{2m-1}}\right)+\spa{(F_2-JF_1)|_{\bR^{2m-1}},JF_2|_{\bR^{2m-1}}}
	\end{equation*}
for $F_1\in \mfh_2^{\bR^{2m-1}}\setminus\mfh_2^{\bR^{2m-1}_J}$ and $F_2\in  \mfh_2 \setminus\mfh_2^{\bR^{2m-1}}$ such that
$(F_2-J F_1)(\bR^{2m-1})\subseteq \bR^{2m-1}$ and such that $JF_2(\bR^{2m-1})\subseteq \bR^{2m-1}$.
\end{enumerate}
\end{theorem}
\begin{proof}
Choose $v\in \bR^{2m-1}\setminus \bR^{2m-1}_J$.

First of all, let $\tilde{H}\in\tilde{\mfk}_{\mfh}\oplus \left.J\mfh_2^{\bR^{2m-1}_J}\right|_{\bR^{2m-1}}$ be given. Then $\tilde{H}=\tilde{H}_1-J\tilde{H}_2$ with $\tilde{H_1}\in \tilde{\mfk}_{\mfh}$ and $\tilde{H}_2\in  \left.J\mfh_2^{\bR^{2m-1}_J}\right|_{\bR^{2m-1}}$ and there are $H_1,H_2\in \mfh$ with $H_1(\bR^{2m-1})\subseteq \bR^{2m-1}$ and $H_2\in \mfh_2^{\bR^{2m-1}_J}$ with $H_i|_{\bR^{2m-1}}=\tilde{H}_i$, $i=1,2$. Defining $\nabla\in (\bR^{2m})^*\otimes \mfh$ by $\nabla_u:=0$ for all $u\in \bR^{2m-1}_J$, $\nabla_v:=H_1$ and $\nabla_{Jv}:=H_2$, we see that $\nabla\in \cD_{\mfh}$ and that $\cT(\nabla)=\tilde{H}\in \End(\bR^{n-1})$. This shows that always
\begin{equation*}
\tilde{\mfk}_{\mfh}\oplus \left.J\mfh_2^{\bR^{2m-1}_J}\right|_{\bR^{2m-1}}\subseteq \cF_{\mfh}.
\end{equation*}

Now conversely, let $\tilde{H}\in \cF_{\mfh}$. Then, by Theorem \ref{th:commutingendo} (a), there is some $\nabla\in \cD_{\mfh}$ and some $v\in \bR^{2m-1}\setminus \bR^{2m-1}_J$ such that
\begin{equation*}
\tilde{H}=(\nabla_{Jv}-J\nabla_v)|_{\bR^{2m-1}}.
\end{equation*}
and $\tilde{H}(\bR^{2m-1})\subseteq \bR^{2m-1}$.

We set $H_1:=\nabla_{Jv}\in \End(\bR^n)$ and $H_2:=\nabla_v\in \End(\bR^n)$
By Lemma \ref{le:cDincK1} and Lemma \ref{le:cK1totallyreal}, we have $H_2(\bR^{2m-1}_J)=0$ so that $H_1(\bR^{2m-1}_J)\subseteq \bR^{2m-1}$ and so, since $H_1 J=JH_1$, even $H_1(\bR^{2m-1}_J)\subseteq \bR^{2m-1}_J$.
Moreover, note that $H_2\in \mfh_2$. We discuss now individually the different cases mentioned in the theorem:
\begin{enumerate}[(a)]
	\item Assume first that $\mfh_2=\mfh_2^{\bR^{2m-1}_J}$. Then $H_2\in \mfh_2^{\bR^{2m-1}_J}$ and so $JH_2$ preserves $\bR^{2m-1}$. Thus, also $H_1$ has to preserve $\bR^{2m-1}$, i.e. we have $H_1|_{\bR^{2m-1}}\in \tilde{\mfk}_{\mfh}$.
	Thus, $H=H_1|_{\bR^{2m-1}}-H_2|_{\bR^{2m-1}}\in  \tilde{\mfk}_{\mfh}\oplus J\mfh_2^{\bR^{2m-1}_J}$, which proves the assertion in this case.
	\item Assume now that $\mfh_2^{\bR^{2m-1}_J}\neq \mfh_2^{\bR^{2m-1}}$ but $\mfh_2^{\bR^{2m-1}}= \mfh_2$ and that, additionally, any $F\in \mfh$ with $F(\bR^{2m-1}_J)\subseteq \bR^{2m-1}_J$ satisfies $F(\bR^{2m-1})\subseteq \bR^{2m-1}$. Then $H_1(\bR^{2m-1})\subseteq \bR^{2m-1}$ and so we must have $JH_2(\bR^{2m-1})\subseteq \bR^{2m-1}$.
	Consequently, $H_2\in  \mfh_2^{\bR^{2m-1}_J}$ and the argumentation in part (a) shows $H\in \tilde{\mfk}_{\mfh}\oplus J\mfh_2^{\bR^{2m-1}_J}$, i.e. the statement holds in this case.
	\item Now assume that $\mfh_2^{\bR^{2m-1}_J}= \mfh_2^{\bR^{2m-1}}$ but $\mfh_2^{\bR^{2m-1}}\neq \mfh_2$ and let $F\in \mfh_2\setminus\mfh_2^{\bR^{2m-1}}$. Then, since $F(v)\notin \bR^{2m-1}$, there is some $\lambda_1\in \bR$ such that $(H_1-\lambda_1 F)(v)\in \bR^{2m-1}$ and so $G_1:=H_1-\lambda_1 F$ preserves $\bR^{2m-1}$ and $G_1|_{\bR^{2m-1}}\in \tilde{\mfk}_{\mfh}$. Moreover, there is some $\lambda\in \bR$ such that $(JF-\lambda F)(v)\in \bR^{2m-1}$, i.e. $JF-\lambda F$ preserves $\bR^{2m-1}$. As $H_2\in \mfh_2$, we may write
	$H_2=G_2+\lambda_2 F$ with $G_2\in \mfh_2^{\bR^{2m-1}_J}$ and $\lambda_2\in \bR$ and so
	\begin{equation*}
	\tilde{H}=G_1|_{\bR^{2m-1}}+\lambda_1 F|_{\bR^{2m-1}}-JG_2|_{\bR^{2m-1}}-\lambda_2 JF|_{\bR^{2m-1}}
	\end{equation*}
	and in order that the right hand side preserves $\bR^{2m-1}$, we must have $\lambda_1 F(v)-\lambda_2 JF(v)\in \bR^{2m-1}_J$, i.e. $\lambda_1=\lambda \lambda_2$. Consequently,
	\begin{equation*}
	\tilde{H}\in
\left(\tilde{\mfk}_{\mfh}\oplus \left.J\mfh_2^{\bR^{2m-1}_J}\right|_{\bR^{2m-1}}\right)+\spa{(JF-\lambda F)|_{\bR^{2m-1}}}
	\end{equation*}
and so
	\begin{equation*}
	\cF_{\mfh}\subseteq
	\left(\tilde{\mfk}_{\mfh}\oplus \left.J\mfh_2^{\bR^{2m-1}_J}\right|_{\bR^{2m-1}}\right)+\spa{(JF-\lambda F)|_{\bR^{2m-1}}}.
\end{equation*}
The converse inclusion follows easily by writing down, for any element $F$ of the space on the right hand side, an element $\nabla\in \cD_{\mfh}$ wich maps under $\cT$ to $F$.
\item 
Finally, assume that $\mfh_2^{\bR^{2m-1}_J}\neq\mfh_2^{\bR^{2m-1}}$ and $\mfh_2^{\bR^{2m-1}}\neq \mfh_2$. Choose any $F_1\in \mfh_2^{\bR^{2m-1}}\setminus\mfh_2^{\bR^{2m-1}_J}$ and $\tilde{F}_2\in \mfh_2 \setminus\mfh_2^{\bR^{2m-1}}$. By considering $F_2=\tilde{F}_2-\lambda F_1$ for appropriate $\lambda\in \bR$ instead of $\tilde{F}_2$, we may assume that $JF_2$ preserves $\bR^{2m-1}$. By scaling $F_1$ appropriately,
we may also assume that $F_2-J F_1$ maps $\bR^{2m-1}$ into $\bR^{2m-1}_J\subseteq \bR^{2m-1}$.
 Next, as in the proof of part (c), there exists some $\lambda_1\in \bR$ such that $G_1:=H_1-\lambda_1 F_2$ preserves $\bR^{2m-1}$ and so $G_1|_{\bR^{2m-1}}\in \tilde{\mfk}_{\mfh}$. Moreover, we may find $\mu_1,\mu_2\in \bR$ and $G_2\in \mfh_2^{\bR^{2m-1}_J}$ such that 
\begin{equation*}
H_2=G_2+\mu_1 F_1+\mu_2 F_2
\end{equation*}
Consequently,
\begin{equation*}
\tilde{H}=G_1|_{\bR^{2m-1}}+\lambda_1 F_2|_{\bR^{2m-1}}-JG_2|_{\bR^{2m-1}}-\mu_1 JF_1|_{\bR^{2m-1}}-\mu_2 JF_2|_{\bR^{2m-1}}.
\end{equation*}
As the right hand side has to preserve $\bR^{2m-1}$, we must have $\mu_1=\lambda_1$ and so get
	\begin{equation*}
	\tilde{H}\in
	\left(\tilde{\mfk}_{\mfh}\oplus \left.J\mfh_2^{\bR^{2m-1}_J}\right|_{\bR^{2m-1}}\right)+\spa{(F_2-J F_1)|_{\bR^{2m-1}},JF_2|_{\bR^{2m-1}}},
\end{equation*}
i.e.
\begin{equation*}
	\cF_{\mfh}\subseteq
\left(\tilde{\mfk}_{\mfh}\oplus \left.J\mfh_2^{\bR^{2m-1}_J}\right|_{\bR^{2m-1}}\right)+\spa{(F_2-JF_1)|_{\bR^{2m-1}},JF_2|_{\bR^{2m-1}}}.
\end{equation*}
The converse inclusion follows again easily by constructing explicitly, for any given element element $\tilde{H}$ of the space on the right hand side, an element $\nabla\in \cD_{\mfh}$ wich maps under $\cT$ to $\tilde{H}$. 
\end{enumerate}
\end{proof}
\begin{remark}
Theorem \ref{th:totallyreal} shows that if $\mfh$ is totally real and of type (I), (III) or (IV), or of type (II) with the additional property that any $F\in \mfh$ with $F(\bR^{2m-1}_J)\subseteq \bR^{2m-1}$ even satisfies $F(\bR^{2m-1})\subseteq \bR^{2m-1}$, then $\cF_{\mfh}\subseteq \tilde{\mfk}_{\mfh}\oplus (\bR^{2m-1}_J)^0\otimes \bR^{2m-1}$. So in these cases, the elements to be added to $\tilde{\mfk}_{\mfh}$ in order to obtain $\cF_{\mfh}$ are all of rank one and have common codimension one kernel $\bR^{2m-1}_J$. In Example \ref{ex:totallyrealtypes} below, we will see that this is, in general, not the case if $\mfh$ is a totally real subalgebra of type (II) without the just mentioned additional property.
\end{remark}
\begin{example}\label{ex:totallyrealtypes}
We give examples of totally real linear subalgebras $\mfh$ of all four types and in the case of type (II) also one which satisfy the additional property and one which does not satisfy this additional property in order to show that the types are not void. In all cases, let $J$ be the standard complex structure on $\bR^{2m}$ so that $\bR^{2m-1}_J=\bR^{2m-2}$:
\begin{itemize}
	\item Type (I): Let $\mfh_0$ be any totally real linear subalgebra of $(\bR^{2m-2},J)$, set $W:=\spa{e_1,e_3,\ldots,e_{2m-3}}$ and
	\begin{equation*}
	\mfh:=\left\{\left.\left(\begin{smallmatrix} A & w & Jw \\ 0 & 0 & 0 \\ 0 & 0 & 0 \end{smallmatrix}\right)\right|A\in \mfh_0,\ w\in W\right\}.
	\end{equation*}
Then $\mfh$ is a totally real subalgebra with
\begin{equation*}
	\mfh_2=\left\{\left.\left(\begin{smallmatrix} 0 & w & Jw \\ 0 & 0 & 0 \\ 0 & 0 & 0 \end{smallmatrix}\right)\right| w\in W \right\}=\mfh_2^{\bR^{2m-1}_J},\quad \tilde{\mfk}_{\mfh}=\left\{\left.\left(\begin{smallmatrix} A & w  \\ 0 & 0 \end{smallmatrix}\right)\right|A\in \mfh_0,\, w\in W \right\}.
\end{equation*}
Hence, $\mfh$ is of type (I) and
\begin{equation*}
\cF_{\mfh}=\left\{\left.\left(\begin{smallmatrix} A & v  \\ 0 & 0 \end{smallmatrix}\right)\right|A\in \mfh_0,\ v\in \bR^{2m-2} \right\}
\end{equation*}
by Theorem \ref{th:totallyreal} since $W\oplus JW=\bR^{2m-2}$.
\item Type (II) with additional property: Again, let $\mfh_0$ be any totally real linear subalgebra of $(\bR^{2m-2},J)$ and $W$ as before but now set
\begin{equation*}
	\mfh:=\left\{\left.\left(\begin{smallmatrix} A & w & Jw \\ 0 & \lambda & 0 \\ 0 & 0 & \lambda \end{smallmatrix}\right)\right|A\in \mfh_0,\ w\in W,\ \lambda\in \bR \right\}.
\end{equation*}
Then $\mfh$ is a totally real subalgebra with
\begin{equation*}
	\mfh_2=\left\{\left.\left(\begin{smallmatrix} 0 & w & Jw \\ 0 & \lambda & 0 \\ 0 & 0 & \lambda \end{smallmatrix}\right)\right| w\in W,\ \lambda\in \bR \right\}=\mfh_2^{\bR^{2m-1}}\neq \mfh_2^{\bR^{2m-1}_J}=\left\{\left.\left(\begin{smallmatrix} 0 & w & Jw \\ 0 & 0 & 0 \\ 0 & 0 & 0 \end{smallmatrix}\right)\right| w\in W\right\}.
\end{equation*}

Hence, $\mfh$ is of type (II) and the additional property in Theorem \ref{th:totallyreal} (b) is satisfied since any element $F\in \mfh$ preserves $\bR^{2m-1}$.
As
\begin{equation*}
\tilde{\mfk}_{\mfh}=\left\{\left.\left(\begin{smallmatrix} A & 0  \\ 0 & \lambda \end{smallmatrix}\right)\right|A\in \mfh_0,\, \lambda\in \bR \right\},
\end{equation*}
Theorem \ref{th:totallyreal} yields
\begin{equation*}
	\cF_{\mfh}=\left\{\left.\left(\begin{smallmatrix} A & v  \\ 0 & \lambda \end{smallmatrix}\right)\right|A\in \mfh_0,\ v\in \bR^{2m-2},\ \lambda\in \bR \right\}.
\end{equation*}
\item Type (II) without additional property: Let $J_0:=\left(\begin{smallmatrix} 0 & -1 \\ 1 & 0 \end{smallmatrix}\right)$, $m=2$ and consider
\begin{equation*}
\mfh=\spa{\diag(J_0,J_0),\diag(0,I_2)}.
\end{equation*}
Then $\mfh_2=\spa{(0,I_0)}=\mfh_2^{\bR^3}\neq \mfh_2^{\bR^3_J}=\{0\}$, i.e. $\mfh_2$ is of type (II). However, $F=\diag(J_0,J_0)$ satisfies $F(\bR^3_J)\subseteq \bR^3$ but not $F(\bR^3)\subseteq \bR^3$. Hence, we cannot apply Theorem \ref{th:totallyreal} to compute $\cF_{\mfh}$. In fact, we show now that in this case, 
\begin{equation*}
\cF_{\mfh}\neq \tilde{\mfk}_{\mfh}\oplus \left.J\mfh_2^{\bR^3_J}\right|_{\bR^3}
\end{equation*}
For this, we first note that 
\begin{equation*}
\tilde{\mfk}_{\mfh}\oplus \left.J\mfh_2^{\bR^3_J}\right|_{\bR^3}=\left\{\left.\left(\begin{smallmatrix} 0 & 0 & 0 \\ 0 & 0 & 0 \\ 0 & 0 & \lambda \end{smallmatrix}\right)\right|\lambda\in \bR \right\}.
\end{equation*}
Next, let $\nabla\in \cD_{\mfh}$ with $\nabla\in \cT^{-1}(\End(\bR^3))$ be given. As $\nabla\in (\bR^4)^*\otimes \mfh$, there exist $\alpha,\beta\in (\bR^4)^*$ such that
\begin{equation*}
\nabla=\alpha\otimes \diag(J_0,J_0)+\beta\otimes \diag(0,I_2).
\end{equation*}
Write now $\alpha=\alpha_0+\mu e^4$, $\beta=\beta_0+\nu e^4$ with $\alpha_0,\beta_0\in \spa{e^1,e^2,e^3}$, $\mu,\nu\in \bR$. Then the condition $\nabla\in \cD_{\mfh}$ yields
\begin{equation*}
-\alpha_0(e_1) e_1=\nabla_{e_1} e_2=\nabla_{e_2} e_1=\alpha_0(e_2) e_2,
\end{equation*}
i.e. $\alpha_0(e_1)=\alpha_0(e_2)=0$. But then
\begin{equation*}
\alpha_0(e_3) J_0 u=\nabla_{e_3} u=\nabla_u e_3=\beta_0(u) e_3
\end{equation*}
for all $u\in \spa{e_1,e_2}$, i.e. $\alpha_0(e_3)=0$ and $\beta_0(e_1)=\beta_0(e_2)=0$. Consequently, $\alpha_0=0$ and $\beta_0=\tau e^3$ for some $\tau\in \bR$ and so
\begin{equation*}
\nabla=\mu e^4\otimes \diag(J_0,J_0)+(\nu e^3+\tau e^4)\otimes \diag(0,I_2).
\end{equation*}
But then
\begin{equation*}
\begin{split}
\cT_{e_4}(\nabla)&=\mu \diag(J_0,J_0)|_{\bR^3}+\tau \diag(0,I_2)|_{\bR^3}+\mu e^4\otimes e_3|_{\bR^3}-(\nu e^3+\tau e^4)|_{\bR^3}\otimes e_4
\\
&=\left(\begin{smallmatrix} 0 & -\mu & 0 \\
                           \mu  & 0 & 0 \\
                            0 & 0 & \tau  \\
                            0 & 0 & \mu+\nu \end{smallmatrix}\right),
\end{split}
\end{equation*}
and the condition $\nabla\in \cT^{-1}(\End(\bR^3))$ forces $\nu=-\mu$. Thus,
\begin{equation*}
\cF_{\mfh}=\left\{\left.\left(\begin{smallmatrix} 0 & -\mu & 0 \\ \mu & 0 & 0 \\ 0 & 0 & \lambda \end{smallmatrix}\right)\right|\lambda,\mu\in \bR \right\}\neq\left\{\left.\left(\begin{smallmatrix} 0 & 0 & 0 \\ 0 & 0 & 0 \\ 0 & 0 & \lambda \end{smallmatrix}\right)\right|\lambda\in \bR \right\}= \tilde{\mfk}_{\mfh}\oplus \left.J\mfh_2^{\bR^3_J}\right|_{\bR^3}=\tilde{\mfk}_{\mfh}.
\end{equation*}
We note that the elements which are added to $\tilde{\mfk}_{\mfh}$ in order to obtain $\cF_{\mfh}$ are not of the form $e^3\otimes u$ for some $u\in \bR^3$. In fact, they even all have rank two.
\item Type (III): Here, take again any totally real subalgebra $\mfh_0$ of $(\bR^{2m-2},J)$ and set
\begin{equation*}
\mfh=\left\{\left.\diag(A,\lambda J_0)\right|A\in \mfh_0,\ \lambda\in \bR\right\}.
\end{equation*}
Then $\mfh_2=\spa{\diag(0,J_0)}\neq \{0\}=\mfh_2^{\bR^{2m-1}}=\mfh_2^{\bR^{2m-1}_J}$, i.e. $\mfh_2$ is of type (III). In this case, Theorem \ref{th:totallyreal} (c) yields 
\begin{equation*}
	\cF_{\mfh}=\left\{\left.\left(\begin{smallmatrix} A & 0  \\ 0 & \lambda \end{smallmatrix}\right)\right|A\in \mfh_0,\ \lambda\in \bR \right\}
\end{equation*}
by taking $F=\diag(0,J_0)$ and so $\lambda=0$.
\item Type (IV): Finally, take again any totally real subalgebra $\mfh_0$ of $(\bR^{2m-2},J)$ but now set
\begin{equation*}
	\mfh=\left\{\left.\diag(A,\lambda J_0+\mu I_2)\right|A\in \mfh_0,\ \lambda,\mu\in \bR\right\}.
\end{equation*}
Here, 
\begin{equation*}
\mfh_2=\spa{\diag(0,I_2),\diag(0,J_0)}\neq \mfh_2^{\bR^{2m-1}}=\spa{\diag(0,I_2)}\neq \{0\}=\mfh_2^{\bR^{2m-1}_J}
\end{equation*}
and so $\mfh$ is of type (IV). Taking $F_1=\diag(0,I_2)$ and $F_2=\diag(0,J_0)$, Theorem \ref{th:totallyreal} (d) gives us
\begin{equation*}
	\cF_{\mfh}=\left\{\left.\left(\begin{smallmatrix} A & 0  \\ 0 & \lambda \end{smallmatrix}\right)\right|A\in \mfh_0,\ \lambda\in \bR \right\}.
\end{equation*}
\end{itemize}
\end{example}

We note that all non-zero elements in $\mfh_2$ have rank two. Hence, if any non-zero element in $\mfh$ has at least rank three, then $\cF_{\mfh}=\tilde{\mfk}_{\mfh}$. We give such linear subalgebras a special name in analogy to so-called \emph{elliptic} subalgebras, whose definition we also recall here:
\begin{definition}
Let $\mfh\subseteq \mathfrak{gl}(n,\bR)$ be a Lie subalgebra. Then $\mfh$ is called \emph{elliptic} if $\mfh$ does not contain any non-zero element of rank at most one. Moreover, $\mfh$ is called \emph{super-elliptic} if $\mfh$ does not contain any non-zero element of rank at most two.
\end{definition}
With this definition at hand and noting that if $\mfh$ is super-elliptic, then surely all conjugated subalgebras are super-elliptic as well, Theorem \ref{th:totallyreal} implies:
\begin{corollary}\label{co:superelliptictotallyreal}
Let $\mfh$ be a super-elliptic totally real subalgebra. Then $\cF_{\mfh}=\tilde{\mfk}_{\mfh}$ and an $H$-structure is torsion-free if and only if it is left-invariantly flat.
\end{corollary}
An important class of super-elliptic totally real subalgebras is provided by the following class of subalgebras:
\begin{definition}
\begin{itemize}
	\item  A \emph{hypercomplex} structure on $\bR^{4k}$ is a triple $(I,J,K)$ of complex structures on $\bR^{4k}$ satisfying $IJ=-JI=K$.
	\item
	An \emph{almost hypercomplex structure} on a manifold $M$ is a triple $(I,J,K)$ of almost complex structures on $M$ which is at each $p\in M$ a hypercomplex structure on $T_p M$.  Equivalently, an almost hypercomplex structure is a $\GL(I_0,J_0,K_0)$-structure for
	\begin{equation*}
		\GL(I_0,J_0,K_0):=\left\{\left.F\in \GL(4k,\bR)\right| [F,I_0]=[F,J_0]=[F,K_0]=0\right\},
	\end{equation*}
	where $I_0,J_0,K_0$ is the standard hypercomplex structure on $\bR^{4k}$.
	\item A \emph{hypercomplex structure} on a manifold $M$ is an almost hypercomplex structure $(I,J,K)$ such that the associated $\GL(I_0,J_0,K_0)$-structure $P$ is torsion-free.
	\item A \emph{hypercomplex subalgebra} $\mfh$ is a real subalgebra of $\mathfrak{gl}(I,J,K)$ for some almost hypercomplex structure $I,J,K$ on $\bR^{4k}$.
\end{itemize}
\end{definition}
\begin{corollary}\label{co:hypercomplex}
Let $\mfh$ be a hypercomplex subalgebra. Then $\mfh$ is super-elliptic and totally real with respect to any induced complex structure on $\bR^{4k}$. Consequently, $\cF_{\mfh}=\tilde{\mfk}_{\mfh}$ and an $H$-structure is torsion-free if and only if it is left-invariantly flat.
\end{corollary}
\begin{proof}
Any non-zero element $F\in\mfh$ has rank at least four since if $Z$ is in the image of $F$, then also $IZ$, $JZ$ and $KZ$ are in the image of $F$. Hence, $\mfh$ is super-elliptic. Moreover, $\mfh$ is totally real since if $F\in \mfh$ such that, e.g., also $IF\in \mfh$, then we obtain
\begin{equation*}
\begin{split}
F(Y)&=-F(K^2Y)=-KF(KY)=-IJF(KY)=J (IF)(KY) \\
&=IF(JKY)=IF(IY)=F(I^2 Y)=-F(Y),
\end{split}
\end{equation*}
i.e. $F(Y)=0$ for all $Y\in \bR^{4k}$ and so $F=0$, i.e. $\mfh\cap I\mfh=\{0\}$. But then the last assertions follow from Corollary \ref{co:superelliptictotallyreal}.
\end{proof}
\begin{remark}
Corollary \ref{co:cK1=0elliptic} below will give us another proof of the assertion in Corollary \ref{co:hypercomplex}
\end{remark}
Corollary \ref{co:hypercomplex} reproves two known results in the literature, where in both cases the group $H$ acts transitively on $\mathrm{Grass}_{4k-1}(\bR^{4k})$:
\begin{example}
\begin{itemize}
\item In the case $H=\GL(I,J,K)$, Corollary \ref{co:hypercomplex} reproves the characterisation of the almost Abelian Lie algebras admitting a hypercomplex structure in \cite[Theorem 3.2]{AB1} and also reproves the result that they are actually all flat, cf. \cite[Proposition 3.7]{AB1}.
\item In the case $H=\mathrm{Sp}(I,J,K)$, Corollary \ref{co:hypercomplex} gives back the classification of almost Abelian hyperk\"ahler Lie algebras from \cite[Proposition 3.2]{BDFi}.
\end{itemize}
\end{example}
\subsubsection{Hyperparacomplex subalgebras}
We begin with the main definition:
\begin{definition}
	\begin{itemize}
		\item  A \emph{hyperparacomplex} structure on a $2m$-dimensional vector space $V$ is a triple $(J,E,K)$ of endomorphisms of $V$ consisting of a complex structure $J$ and two para-complex structures $E$ and $K$ on $V$ satisfying $JE=-EJ=K$. The \emph{standard} hypercomplex structure on $\bR^{2m}$ is the hyperparacomplex structure  $J_0,E_0,K_0)$ uniquely defined by $J_0 e_i=e_{m+i}$, $E_0 e_i=e_i$ and $E_0 e_{m+i}=-e_{m+i}$ for $i=1,\ldots,m$.
		\item
		An \emph{almost hyperparacomplex structure} on a manifold $M$ is a triple $(J,E,K)$ of endomorphism fields which is pointwise a hyper-paracomplex structure on $T_p M$. Equivalently, an almost hyperparacomplex structure is a \linebreak $\GL(J_0,E_0,K_0)$-structure for
		\begin{equation*}
		\GL(J_0,E_0,K_0):=\left\{\left.F\in \GL(4k,\bR)\right| [F,J_0]=[F,E_0]=[F,K_0]=0\right\},
		\end{equation*}
		where $J_0,E_0,K_0$ is the standard hyperparacomplex structure on $\bR^{2m}$.
		\item 
		A \emph{hyperparacomplex structure} is an almost hyperparacomplex structure for which the associated $\GL(J_0,E_0,K_0)$-structure is torsion-free.
		\item A \emph{hyperparacomplex subalgebra} $\mfh$ is a real subalgebra of $\mathfrak{gl}(J,E,K)$ for some hyper paracomplex structure $J,E,K$ on $\bR^{4k}$.
	\end{itemize}
\end{definition}
\begin{remark}
	\begin{itemize}
		\item 
	Hyperparacomplex structures are also called \emph{product structures} since they may, alternatively, be defined as a pair $(J,E)$ of a complex structure $J$ and a product structure $E$ with $JE=-EJ$. In this case, $E$ is automatically a paracomplex structure on the underlying manifold $M$. Note that, in contrast to the hypercomplex case, the dimension of $M$ need not to be divisible by four but only by two.
	\item
	Any hyperparacomplex structure $(J,E,K)$ on $\bR^{2m}$ equips $\bR^{2m}$ with the splitting $\bR^{2m}=\bR^{2m}_+\oplus \bR^{2m}_-$ with $\bR^{2m}_{\pm}$ being the $\pm 1$-eigenspaces of $E$. Then $J$ identifies $\bR^{2m}_-$ with $\bR^{2m}_+$. Under this identification, we have
	 \begin{equation*}
	 	\mathfrak{gl}(J,E,K)=\left\{\left.\diag(A,A)\right|A\in \mathfrak{gl}(m,\bR)\right\}=\Delta \mathfrak{gl}(m,\bR).
	 \end{equation*}
	\end{itemize}
\end{remark}
Next, we show that hypercomplex subalgebras are totally real:
\begin{lemma}
Let $\mfh\subseteq \mathfrak{gl}(J,E,K)$ be a hyperparacomplex subalgebra. Then $\mfh$ is totally real with respect to $J$.
\end{lemma}
\begin{proof}
Let $F\in \mfh$ be given such that $JF\in \mfh$ as well. Then
\begin{equation*}
\begin{split}
F&=-J^2 F=-JFJ=-JFJE^2=-JFKE=-KJFE=-JEJFE\\
=&EJ^2FE=-EFE=-FE^2=-F
\end{split}
\end{equation*}
and so $F=0$. Hence, $\mfh$ is totally real.
\end{proof}
\begin{definition}
Let $\mfh$ be a hyperparacomplex subalgebra $\mfh$. Then
\begin{equation*}
\mfh=\left\{\left.\diag(A,A)\right|A\in \tilde{\mfh}\right\}=\Delta \tilde{\mfh}
\end{equation*}
 for some subalgebra $\tilde{\mfh}$ of $\mathfrak{gl}(m,\bR)$ and we say that $\mfh$ is \emph{induced by} $\tilde{\mfh}$.

If $\tilde{\mfh}=\mathfrak{so}(m)$, we call $\Delta \mathrm{O}(m)$-structures $P$ also \emph{almost K\"ahler-K\"unneth structures}, whereas if $m=2k$ and $\tilde{\mfh}=\mathfrak{sp}(2k,\bR)$, then an $\Delta\mathrm{Sp}(2k,\bR)$-structure $P$ is called \emph{almost hyper para-K\"ahler} or \emph{almost hypersymplectic} structure. As usual, the word ``almost'' is skipped if $P$ is torsion-free in both cases.
\end{definition}
Observing that if $\tilde{\mfh}$ is elliptic, then $\mfh$ is super-elliptic, Corollary \ref{co:superelliptictotallyreal} implies:
\begin{theorem}\label{th:hyperparacomplexelliptic}
Let $\mfh$ be a hyperparacomplex subalgebra induced by an elliptic subalgebra $\tilde{\mfh}$ of $\mathfrak{gl}(m,\bR)$. Then $\cF_{\mfh}=\tilde{\mfk}_{\mfh}$ and so an $H$-structure on an almost Abelian Lie algebra is torsion-free if and only if it is left-invariantly flat.
\end{theorem}
\begin{remark}
Theorem \ref{th:hyperparacomplexelliptic} show that, in particular, any K\"ahler-K\"unneth structure on an almost Abelian Lie algebra is flat, a result which is known to be true for any K\"ahler-K\"unneth structure on any manifold by \cite{HKo}.
\end{remark}
We do now concentrate on the case $H=\GL(J_0,E_0,K_0)$ and aim at getting a characterisation of the almost Abelian Lie algebras admitting a torsion-free $H$-structures, i.e. a hyperparacomplex structure. Instead of determining all possible orbits of the $\GL(J,E,K)$-action on $\mathrm{Grass}_{2m-1}(\bR^{2m})$, which would lead to an infinite number of such orbits, we only distinguish hyperparacomplex subalgebras $\mfh$ conjugated to $\mathfrak{gl}(J_0,E_0,K_0)$ according to whether $\bR^{2m-1}_J$ is $E$-invariant or not and first determine in both cases some properties of $\mfh$ and $\mfh_2$:
\begin{lemma}\label{le:hyperparacomplexmfh2}
Let $(J,E,K)$ be a hyperparacomplex structure on $\bR^{2m}$ and let $\mfh:=\mathfrak{gl}(J,E,K)$.
\begin{enumerate}[(a)]
	\item If $\bR^{2m-1}_J$ is $E$-invariant, then $\mfh$ is of type $(II)$ and any element $F\in \mfh$ with $F(\bR^{2m-1}_J)\subseteq \bR^{2m-1}_J$ satisfies $F(\bR^{2m-1})\subseteq \bR^{2m-1}$.
	\item If $\bR^{2m-1}_J$ is not $E$-invariant, then $\mfh_2=\{0\}$.
\end{enumerate}
\end{lemma}
\begin{proof}
\begin{enumerate}[(a)]
	\item Let $\bR^{2m-1}_J$ be $E$-invariant. We first show that then $\bR^{2m-1}_J$ admits an $E$- and 
	$J$-invariant complement in $\bR^{2m}$. To show this, take any $v\in \bR^{2m}\setminus \bR^{2m-1}_J$. We may 
	decompose $v=v_+ + v_-$ with $Ev_{\pm}=\pm v_{\pm}$ and must have $v_+\notin \bR^{2m-1}_J$ or
	$v_-\notin  \bR^{2m-1}_J$ as if both vectors are in $\bR^{2m-1}_J$, then $v=v_+ + v_-\in  \bR^{2m-1}_J$ as well, a contradiction. In any case, we have found $w\in  \bR^{2m}\setminus \bR^{2m-1}_J$ which is an eigenvector for $E$. But then $U_2:=\spa{w,Jw}$ is an $E$- and $J$-invariant complement of $\bR^{2m-1}_J$ in $\bR^{2m}$. Now since $F$ commutes with $J$, there are $a,b\in \bR$ and some $J$-equivariant map $G:U_2\rightarrow \bR^{2m-1}_J$ such that
	\begin{equation*}
	F(w)=aw+b Jw+G(w),\quad F(Jw)=-b w+a Jw +G(Jw).
	\end{equation*}
Since $F$ also commutes with $E$ and if $\epsilon\in\{1,-1\}$ denotes the eigenvalue of $w$ with respect to $E$, we have
\begin{equation*}
\begin{split}
a w+b\, Jw+ G(w)&=F(w)=F(\epsilon\, Ew)=\epsilon E F(w)=\epsilon E (aw+b Jw+G(w))\\
&=a w- b Jw+ \epsilon E G(w),
\end{split}
\end{equation*}
which yields that $G$ is $E$-invariant and $b=0$. Hence,
\begin{equation*}
F(u)=a u+G(u)
\end{equation*}
for all $u\in U_2$. In particular, $\mfh_2=\mfh_2^{\bR^{2m-1}}$ and any element $F\in \mfh$ with $F(\bR^{2m-1}_J)\subseteq \bR^{2m-1}_J$ satisfies $F(\bR^{2m-1})\subseteq \bR^{2m-1}$. Finally, $\mfh_2^{\bR^{2m-1}}\neq \mfh_2^{\bR^{2m-1}_J}$, and so $\mfh$ is of type $(II)$, since $\tilde{F}\in \End(\bR^{2m})$, defined by $\tilde{F}(\bR^{2m-1}_J)=0$ and $\tilde{F}(u)=u$ for $u\in U_2$ with $U_2$ chosen as above, is an element in $\mfh_2^{\bR^{2m-1}}\backslash \mfh_2^{\bR^{2m-1}_J}$.
\item Now assume that $\bR^{2m-1}_J$ is not $E$-invariant and let $F\in \mfh_2$ be given. Moreover, let $u\in \bR^{2m-1}_J$ be such that $w:=Eu\notin  \bR^{2m-1}_J$. Then $\spa{w,Jw}$ is a complement of $\bR^{2m-1}_J$ in $\bR^{2m}$ and 
\begin{equation*}
F(w)=F(E^2 w)=E F(Ew)=EF(u)=0.
\end{equation*}
and so also $F(Jw)=0$, which shows that $F=0$, i.e. $\mfh_2=\{0\}$.
\end{enumerate}
\end{proof}
\begin{theorem}\label{th:hyperparacomplexclass}
Let $\mfh$ be a $2m$-dimensional almost Abelian Lie algebra. Then:
\begin{enumerate}[(a)]
\item
 $\mfh$ admits a hyperparacomplex structure $(J,E,K)$ for which $\mfu_J$ is $E$-invariant if and only if 
 there exist $A\in \bR^{(m-1)\times (m-1)}$, $w_1,w_2\in \bR^{m-1}$ and $a\in \bR$ with
 \begin{equation*}
 f=\begin{pmatrix} A & 0 & w_1 \\ 0 & A & w_2 \\ 0 & 0 & a \end{pmatrix}
 \end{equation*}
 with respect to some basis $(X_1,\ldots,X_{m-1},Y_1,\ldots,Y_{m-1},V)$ of $\mfu$. 
 
In this case, there is  a hyperparacomplex structure $(J,E,K)$ for which $\mfu_J$ is $E$-invariant such $\spa{X_1,\ldots,X_{m-1}}$ is in the $+1$-eigenspace of $E$, $\spa{Y_1,\ldots,Y_{m-1}}$ is in the $-1$-eigenspace of $E$ and $J(X_i)=Y_i$ for $i=1,\ldots,m-1$.
\item 
$\mfh$ admits a hyperparacomplex structure $(J,E,K)$ for which $\mfu_J$ is not $E$-invariant if and only if 
there exist $A\in \bR^{(m-2)\times (m-2)}$, $u_1,u_2\in \bR^{m-2}$ and $u\in \bR$ such that
\begin{equation*}
	f=\begin{pmatrix}
		A & 0 & u_1 & -u_2 & u_1 \\
		0& A & u_2 & u_1 & -u_2 \\
		0 & 0 & a & 0 & 0 \\
		0 & 0 & 0 & a & 0 \\
		0 & 0 & 0 & 0 & a
	\end{pmatrix} 
\end{equation*}
with respect to some basis $(X_1,\ldots,X_{m-2},Y_1,\ldots,Y_{m-2},V_1,V_2,V_3)$ of $\mfu$. In this case, there is a hyperparacomplex structure $(J,E,K)$ for which $\mfu_J$ is not $E$-invariant such that $\spa{X_1,\ldots,X_{m-2}}$ is in the $+1$-eigenspace of $E$, $\spa{Y_1,\ldots,Y_{m-2}}$ is in the $-1$-eigenspace of $E$, $J(X_i)=Y_i$ for $i=1,\ldots,m-1$ and $J(V_1)=V_2$ and $E(V_1)=V_3$.
\end{enumerate}
\end{theorem}
\begin{proof}
\begin{enumerate}[(a)]
	\item 
Let $(J,E,K)$ be a hyperparacomplex structure on $\mfg$ with $\mfu_J$ being $E$-invariant. Let $[U]$ be the type of $(J,E,K)$ and let $u$ be the adapted frame of $P$ identifying $\bR^{2m}$ with $\mfg$ and $U$ with $\mfu$. We note that under this identification, $(J,E,K)$ get the standard hyperparacomplex structure $(J_0,E_0,K_0)$ on $\bR^{2m}$. As $U_J$ endowed with the restrictions of the hypercomplex structure $(J_0,E_0,K_0)$ is a hyperparacomplex vector space of dimension $2m-2$, we may find an element $T\in \GL(2m,\bR)$ with $T(U_J)=\bR^{2m-2}$, $T(U)=\bR^{2m-1}$ and such that the hyperparacomplex structure $(\tilde{J}_0,\tilde{E}_0,\tilde{K}_0)$ defined by $\tilde{A}_0:=T A_0 T^{-1}$ for $A\in\{J,E,K\}$ satisfies $\tilde{J}_0 e_i=e_{m-1+i}$, $\tilde{E}_0 e_i=e_i$ and $\tilde{E}_0 e_{m-1+i}=-e_{m-1+i}$ for $i=1,\ldots,m-1$. Now $e_{2m-1}$ is, in general, not an eigenvector of $\tilde{E}_0$ but may be written as $e_{2m-1}=\lambda_1 w-\lambda_2 Jw$ for $w\in \bR^{2m}\setminus \bR^{2m-2}$ with $\tilde{E}_0 w=w$. We note that this is equivalent to $\lambda_1 e_{2m-1}+\lambda_2 J e_{2m-1}$ being an eigenvector of $\tilde{E}_0$ with eigenvalue $1$ and that then
\begin{equation*}
\tilde{\mfk}_{T\mfh T^{-1}}= \left\{\left.\left(\begin{smallmatrix} A & 0 & u_1 \\ 0 & A & u_2 \\ 0 & 0 & a \end{smallmatrix}\right)\right|A\in \GL(m-1,\bR), u_1,u_2\in \bR^{m-1} \textrm{ with } \lambda_2 u_1+\lambda_1 u_2=0, \, a\in \bR\right\}.
\end{equation*}
Recalling that $(T\mfh T^{-1})_2^{\bR^{2m-1}_J}|_{\bR^{2m-1}}$ is contained in $\tilde{\mfk}_{T\mfh T^{-1}}$, we see that $(T\mfh T^{-1})_2^{\bR^{2m-1}_J}|_{\bR^{2m-1}}$ consists of exactly the elements in $\tilde{\mfk}_{T\mfh T^{-1}}$ which satisfy $A=0$ and $a=0$. Thus
\begin{equation*}
J(T\mfh T^{-1})_2^{\bR^{2m-1}_J}|_{\bR^{2m-1}}=\left\{\left.\left(\begin{smallmatrix} 
	0 & 0 & -u_2 \\ 0 & 0 & u_1 \\ 0 & 0 & 0 \end{smallmatrix}\right)\right| u_1,u_2\in \bR^{m-1} \textrm{ with } \lambda_2 u_1+\lambda_1 u_2=0, \right\}
\end{equation*}
Hence, the result follows from Theorem \ref{th:totallyreal} and Lemma \ref{le:hyperparacomplexmfh2}, which imply
\begin{equation*}
\begin{split}
	\cF_{T\mfh T^{-1}}&=\tilde{\mfk}_{T\mfh T^{-1}}\oplus J(T\mfh T^{-1})_2^{\bR^{2m-1}_J}|_{\bR^{2m-1}}\\
	&=\left\{\left.\left(\begin{smallmatrix} A & 0 & w_1 \\ 0 & A & w_2 \\ 0 & 0 & a \end{smallmatrix}\right)\right|A\in \GL(m-1,\bR), w_1,w_2\in \bR^{m-1}, \, a\in \bR\right\}
\end{split}
\end{equation*}
\item 
Let $(J,E,K)$ be a hyperparacomplex structure on $\mfg$ with $\mfu_J$ not being $E$-invariant. Then 
$\mfu_{J,E}:=\mfu_J\cap E\mfu_J$ is $J$- and $E$-invariant and has dimension $2m-4$. Note that since $\mfu_J$ is not $E$-invariant, we may choose $v\in \mfu_J\setminus \mfu_{J,E}$ with $Ev\notin \mfu_J$. Then $v,Jv,Ev, Kv$ are linearly independent and $\spa{v,Jv,Ev,Kv}$ is a complement of $\mfu_{J,E}$ in $\mfg$. By choosing $v$ appropriately, we may assume that $\mfu_{J,E}\oplus \spa{v,Jv,Kv}=\mfu$.

Next, denote by $[U]$ be the type of 
$(J,E,K)$ and identify as in the proof of part (a) with an appropriate adatped frame $u$ the spaces $\bR^{2m}$ with $\mfg$, $\bR^{2m-1}$ with $\mfu$ and the hyperparacomplex structure $(J,E,K)$ with $(J_0,E_0,K_0)$. Here, we may choose $T\in \GL(2m,\bR)$ such that $T(U_{J,E})=\spa{e_1,\ldots,e_{2m-4}}$, $T(u^{-1}(v))=e_{2m-3}$, $T(u^{-1}(Jv))=e_{2m-2}$, $T(u^{-1}(Ev))=e_{2m-1}$, $T(u^{-1}(Kv))=e_{2m}$ and such that the hyperparacomplex structure $(\tilde{J}_0,\tilde{E}_0,\tilde{K}_0)$ defined as in part (a) satisfies $\tilde{J}_0 e_i=e_{m-2+i}$, $\tilde{E}_0 e_i=e_i$ and $\tilde{E}_0 e_{m-2+i}=-e_{m-2+i}$ for $i=1,\ldots,m-2$. We note that then $\tilde{J}_0 e_{2m-3}=e_{2m-2}$, $\tilde{E}_0(e_{2m-3})=e_{2m-1}$ and $\tilde{K}_0(e_{2m-3})=e_{2m}$.

Let now $F\in T\mfh T^{-1}$ with $F(\bR^{2m-1})\subseteq \bR^{2m-1}$. Then $F$ has to preserve $\bR^{2m-4}:=\spa{e_1,\ldots,e_{2m-4}}$ since this space equals $\bR^{2m-1}_{\tilde{E}_0,\tilde{J}_0}$ for the induced hyperparacomplex structure. Moreover, if we set $U_3:=\spa{e_{2m-3}, e_{2m-2},\linebreak e_{2m-1}}$ and denote by $\tilde{F}$ the projection of $F|_{U_3}:U_3\rightarrow \bR^{2m-1}$ to $U_3$ along $\bR^{2m-4}$, we get
\begin{equation*}
\tilde{F}(e_{2m-3})=a e_{2m-3}+b  e_{2m-2}+c e_{2m-1}
\end{equation*}
for certain $a,b,c\in \bR$. Then
\begin{equation*}
	\bR^{2m-1}\ni\tilde{F}(e_{2m-2})=\tilde{F}(\tilde{J}_0 e_{2m-3})=\tilde{J}_0\tilde{F}(e_{2m-3})=-b e_{2m-3}+a e_{2m-2}+c e_{2m}
\end{equation*}
forcing $c=0$. Moreover,
\begin{equation*}
	\bR^{2m-1} \ni \tilde{F}(e_{2m-1})=\tilde{F}(\tilde{E}_0 e_{2m-3})=\tilde{E}_0\tilde{F}(e_{2m-3})=a e_{2m-1}-b e_{2m}
\end{equation*}
forcing $b=0$. This shows that
\begin{equation*}
	F=\begin{pmatrix}
		A & 0 & u_1 & -u_2 & u_1 & u_2 \\
		0& A & u_2 & u_1 & -u_2 & u_1\\
		0 & 0 & a & 0 & 0 & 0\\
		0 & 0 & 0 & a & 0  & 0 \\
		0 & 0 & 0 & 0 & a & 0\\
		0 & 0 & 0 & 0 & 0 & a
	\end{pmatrix} 
\end{equation*}
for certain $A\in \bR^{(m-2)\times (m-2)}$, $u_1,u_2\in \bR^{m-2}$, and the assertion follows.
\end{enumerate}
\end{proof}
\begin{example}
	Let $m=2$. In \cite{BV}, all four-dimensional Lie algebras admitting a hyperparacomplex structure have been determined. Let us show that our classification of four-dimensional almost Abelian Lie algebras admitting a hyperparacomplex structure coincides with the classification of \cite{BV} in the almost Abelian case.
	
	For this, note that by Theorem \ref{th:hyperparacomplexclass}, a four-dimensional almsot Abelian Lie algebra admits a hyperparacomplex structure if and only if 
		\begin{equation*}
			f=\begin{pmatrix} a & 0 & w_1 \\ 0 & a & w_2 \\ 0 & 0 & b \end{pmatrix}
		\end{equation*}
		for certain $a,b,w_1,w_2\in \bR$ with respect to some basis of $\mfu$. The possible Jordan normal forms up to non-zero scaling of these endomorphisms of $\bR^3$ are	
		\begin{equation*}
		0,\ \diag(1,1,\lambda),\ \diag(0,0,1),\ \left(\begin{smallmatrix} 0 & 0 & 0 \\ 0 & 0 & 1 \\ 0 & 0 & 0 \end{smallmatrix} \right), \left(\begin{smallmatrix} 1 & 0 & 0 \\ 0 & 1& 1 \\ 0 & 0 & 1 \end{smallmatrix} \right).
	\end{equation*}
	and these correspond exactly to the Lie algebras called (PHC1), (PCH9) for $\lambda\neq 0$ where $c\neq 1$ or $c=1$ and $a=b=0$ in \cite{BV} and (PCH3) for $\lambda=0$, (PCH5), (PCH4) and (PCH9) for $c=1$ and $(a,b)\neq (0,0)$, respectively, which are exactly the almost Abelian Lie algebras admitting a hyperparacomplex structure according to \cite{BV}.
\end{example}
We end this section by identifying the \emph{flat} hyperparacomplex structures, i.e. those where the underlying $\GL(J_0,E_0,K_0)$-structure is flat, among all hyperparacomplex structures:
\begin{corollary}\label{co:flathyperparacomplex}
Let $\mfg$ be an almost Abelian Lie algebra and $(E,J,K)$ be a hyperparacomplex structure on $\mfg$. Then:
\begin{enumerate}[(a)]
	\item  Let $\mfu_J$ be $E$-invariant. Then $(J,E,K)$ is flat if and only if
	\begin{equation*}
f=\begin{pmatrix}
			A & 0 & w_1 \\
			0 & A & w_2 \\
			0 & 0 & a
		\end{pmatrix}
	\end{equation*}
for some $A\in \mathfrak{gl}(m-1,\bR)$, $w_1,w_2\in \bR^{m-1}$ and $a\in \bR$ such that if $(\lambda,\mu)\in \{(0,0)\}$ are chosen so that $\lambda V+ \mu JV$ is an eigenvector of $E$ with eigenvalue $1$ for some $V\in \mfu\setminus \mfu_J$, then $\mu w_1+\lambda w_2$ is either zero or an eigenvector of $A$ with eigenvalue $2a$.

\item If $\mfu_J$ is not $E$-invariant, then $(J,E,K)$ is flat.
\end{enumerate}
\end{corollary}
\begin{proof}	
Part (b) follows directly from the fact that by the proof of Theorem \ref{th:hyperparacomplexclass} (b), we have $\cF_{\mfh}=\tilde{\mfk}_{\mfh}$ for $\mfh:=\mathfrak{gl}(\tilde{J},\tilde{E},\tilde{K})$ with $(\tilde{J},\tilde{E},\tilde{K})$ being a hyperparacomplex structure on $\bR^	{2m}$ for which $\bR^{2m}_{\tilde{J}}$ is not $J$-invariant.

So let us assume that $\mfu_J$ is $J$-invariant. We note that by \cite[Proposition 5.1]{ASa}, $\mfg$ admits a unique torsion-free $H$-connection $\nabla$ on $\mfg$ and so $P$ is flat if and only if $\nabla$ is flat. By the proof of Theorem \ref{th:hyperparacomplexclass} (a), we may identify $\bR^{2m}$ with $\mfg$ and $\bR^{2m-1}$ with $\mfu$  and $\bR^{2m-2}$ with $\mfu_J$ in such a way that $\lambda e_{2m-1}+\mu J e_{2m-1}$ is an eigenvector of $E$ with eigenvalue $1$ and that then $\nabla$ is given by $\nabla_u=0$ for all $u\in \bR^{2m-2}$ and
\begin{equation*}
	\nabla_{e_{2m-1}}=
		\left(\begin{smallmatrix}
		0 & 0 & u_1 & -u_2 \\
		0 & 0 & u_2 & u_1 \\
		0 & 0 & 0 & 0 \\
		0 & 0 & 0 & 0
		\end{smallmatrix}\right),\quad
		\nabla_{Je_{2m-1}}=\left(\begin{smallmatrix}
			A & 0 & v_1 & -v_2 \\
			0 & A & v_2 & v_1 \\
			0 & 0 & a & 0 \\
			0 & 0 & 0 & a
		\end{smallmatrix}\right)
	\end{equation*}
for certain $A\in \mathfrak{gl}(m-1,\bR)$, $u_1,u_2,v_1,v_2\in \bR^{m-1}$, $a\in \bR$ with $\mu u_1+\lambda u_2=0$ and $\mu v_1+\lambda v_2=0$.

Note that so
\begin{equation*}
f=\left(\begin{smallmatrix}
	A & 0 & v_1+u_2 \\
	0 & A & v_2-u_1 \\
	0 & 0 & a
\end{smallmatrix}\right).
\end{equation*}
Moreover,
\begin{equation*}
[\nabla_{Je_{2m-1}},\nabla_{e_{2m-1}}]=
	\left(\begin{smallmatrix}
		0 & 0 & A u_1-a u_1 & - Au_2 +a u_2 \\
		0 & 0 & A u_2- au_2 & A u_1- a u_1 \\
		0 & 0 & 0 & 0 \\
		0 & 0 & 0 & 0
	\end{smallmatrix}\right)
\end{equation*}
and this has to be equal to
\begin{equation*}
\nabla_{[Je_{2m-1},e_{2m-1}]}=\nabla_{f(e_{2m-1})}=\nabla_{a e_{2m-1}}=\left(\begin{smallmatrix}
	0 & 0 & a u_1 & -a u_2 \\
	0 & 0 & a u_2 & a u_1 \\
	0 & 0 & 0 & 0 \\
	0 & 0 & 0 & 0
\end{smallmatrix}\right).
\end{equation*}
This is satisfied if and only if $u_1$ and $u_2$ are either zero or eigenvectors of $A$ with eigenvalue $2a$, from which the assertion directly follows.
\end{proof}
\begin{example}
The non-flatness of the hyperparacomplex structure given in \cite[Example 6.3]{ASa} may be explained by Corollary \ref{co:flathyperparacomplex}. In this example, the authors consider the four-dimensional almost Abelian Lie algebra with basis $(X_1,X_2,X_3,X_4)$ such that $(X_1,X_2,X_3)$ is a basis of $\mfu$,
\begin{equation*}
f:=\ad(X_4)|_{\mfu}=\begin{pmatrix}
	-1 & 0 & -2 \\
	0 & -1 &  0\\
	0 & 0 & 1
\end{pmatrix}
\end{equation*}
and the hyperparacomplex structure $(J,E,K)$ on $\mfg$ defined by $\mfg_+=\spa{X_1,X_4}$, $\mfg_-=\spa{X_2,X_3}$ and $JX_1=X_2$, $JX_3=-X_4$. In this case, we may choose $\lambda=0$ and $\mu=1$ so that $\nabla$ is flat only if $1\cdot -2=-2$ would be zero or an eigenvector of $-1$ with eigenvalue $2$. As this is not the case, $(J,E,K)$ is non-flat here.
\end{example}
\subsubsection{Unitary subalgebras}
Again, we begin with the basic definition:
\begin{definition}
Let $(g,J)$ be a Hermitian structure on $\bR^{2m}$. Then we define
\begin{equation*}
\mathfrak{u}(g,J):=\left\{\left.F\in \mathfrak{gl}(J)\right| g(Fv,w)=-g(v,Fw) \textrm{ for all }v,w\in \bR^{2m}\right\}.
\end{equation*}
A subalgebra $\mfh$ of $\mathfrak{u}(g,J)$ is called a \emph{unitary} subalgebra.
\end{definition}
Unitary subalgebras are totally real:
\begin{lemma}\label{le:unitarytotallyreal}
A unitary subalgebra $\mfh$ is totally real.
\end{lemma}
\begin{proof}
Let $F\in \mfh$ be given. Then
\begin{equation*}
g(JFv,w)=-g(Fv,Jw)=g(v,FJw)=g(v,JFw)
\end{equation*}
for all $v,w\in \bR^{2m}$. Hence, $JF$ is not skew-symmetric and so not in $\mathfrak{u}(g,J)$ unless $F=0$. This shows that $\mfh$ is totally real.
\end{proof}
For a unitary subalgebra, we get the following explicit formula for $\cF_{\mfh}$ in the various different cases:
\begin{theorem}\label{th:unitarysubalgebras}
Let $\mfh$ be a unitary subalgebra. Then:
\begin{enumerate}[(a)]
	\item If $\bR^{2m-1}$ is non-degenerate and $v\in \bR^{2m-1}$ is orthogonal to $(\bR^{2m-1})_J$, then
	\begin{equation*}
\cF_{\mfh}=\tilde{\mfk}_{\mfh}
	\end{equation*}
	if $v\otimes (Jv)^b-Jv\otimes v^b\notin \mfh$ and 
		\begin{equation*}
	\cF_{\mfh}=\tilde{\mfk}_{\mfh}\oplus \spa{v\otimes v^b}
		\end{equation*}
	if $v\otimes (Jv)^b-Jv\otimes v^b\in \mfh$.
    \item If $(\bR^{2m-1})_J$ is degenerate and $0\neq v\in (\bR^{2m-1})^{\perp}$, then
    \begin{equation*}
    	\cF_{\mfh}=\tilde{\mfk}_{\mfh}
    \end{equation*}
    if $v^b\otimes Jv-(Jv)^b\otimes v\notin \mfh$ and 
    \begin{equation*}
    	\cF_{\mfh}=\tilde{\mfk}_{\mfh}\oplus \spa{(Jv)^b|_{\bR^{2m-1}}\otimes Jv}
    \end{equation*}
    if $v^b\otimes Jv-(Jv)^b\otimes v\in \mfh$.
\end{enumerate}
\end{theorem}
\begin{proof}
\begin{enumerate}[(a)]
	\item First of all, observe that if $\bR^{2m-1}$ is non-degenerate, then the same is true for $\bR^{2m-1}_J$ as otherwise, there would exist some $v\in \bR^{2m-1}_J$ with $g(v,u)=0$ for all $u\in bR^{2m-1}_J$. But then also $g(Jv,u)=-g(v,Ju)=0$ for all $u\in \bR^{2m-1}$ and so some non-zero linear combination of $v$ and $Jv$ would be orthogonal to $\bR^{2m-1}$ contradicting the non-degeneracy of $\bR^{2m-1}$.
	
Now let $F\in \mfh_2$ be given. As $F$ is zero on $(\bR^{2m-1})_J$ it preserves the orthogonal complement $\spa{v,Jv}$ of that space and so $\mfh_2=\spa{v\otimes (Jv)^b-Jv\otimes v^b}$ if $v\otimes (Jv)^b-Jv\otimes v^b\in \mfh$ and, otherwise, $\mfh_2=0$. Hence, the asserion follows from Theorem \ref{th:totallyreal} since in $\mfh_2^{\bR^{2m-1}_J}=\{0\}$ and in the first case, $\mfh$ is of type $(III)$, whereas in the second case, it is of type $(I)$.
	\item Here, we observe that $(\bR^{2m-1})^{\perp}\subseteq \bR^{2m-1}_J$, and so also $\bR^{2m-1}_J$ is degenerate since otherwise $\bR^{2m-1})^{\perp}\cap \bR^{2m-1}_J=\{0\}$ and then for any non-zero element $0\neq v\in \bR^{2m-1})^{\perp}$ we have $\bR^{2m-1}\oplus \spa{Jv}=\bR^{2m}$. However, since $v$ is also orthogonal to $Jv$, this would imply that $v$ is orthogonal to $\bR^{2m}$, contradicting that $g$ is a pseudo-Riemannian metric on $\bR^{2m}$.
	
Then the assertion is surely clear if $\mfh_2=\{0\}$. So assume that $\mfh_2\neq \{0\}$ and let $0\neq F\in \mfh_2$.  Moreover, let $0\neq v\in (\bR^{2m-1})^{\perp}$. Then $F=\alpha\otimes Jw+\alpha\circ J \otimes w$ for some $\alpha\in (\bR^{2m-1})_J^0$ and some $w\in \bR^{2m}$. As $(\bR^{2m-1})_J^0=\spa{v^b, (Jv)^b}$, we may, w.l.o.g., assume that $\alpha=v^b$. But then, since $F$ is skew-symmetric, we must have 
	\begin{equation*}
	F=\lambda \left(v^b\otimes Jv-(Jv)^b\otimes v\right)
	\end{equation*}
for some $\lambda\in \bR^*$. In particular, $\mfh_2=\mfh_2^{\bR^{2m-1}_J}=\spa{v^b\otimes Jv-(Jv)^b\otimes v}$, i.e. $\mfh$ is of type $(I)$. Since
\begin{equation*}
	J (v^b\otimes Jv-(Jv)^b\otimes v)=-v^b\otimes v-(Jv)^b\otimes Jv
\end{equation*}
preserves $\bR^{2m-1}$ and $v^b|_{\bR^{2m-1}}=0$,
 the assertion follows from Theorem \ref{th:totallyreal} (a).
\end{enumerate}
\end{proof}
\begin{remark}
If $\mfh$ is a unitary subalgebra with the associated metric being Riemannian, then $\cF_{\mfh}=\tilde{\mfk}_{\mfh}$ if $v\otimes (Jv)^b-Jv\otimes v^b\notin \mfh$
and $\cF_{\mfh}=\tilde{\mfk}_{\mfh}\oplus \spa{v\otimes v^b}$ if $v\otimes (Jv)^b-Jv\otimes v^b\in \mfh$, where $v$ is as in Theorem \ref{th:unitarysubalgebras} (a).
\end{remark}
\begin{example}\label{ex:upn-p}
\begin{itemize}
	\item For $\mfh=\mathfrak{u}(m)$, we may choose $v=e_{2m-1}$ and get $\cF_{\mfu(m)}=\mfu(m-1)\oplus \spa{e_{2m-1}\otimes \spa{e^{2m-1}\otimes e_{2m-1}}$. As $\mathrm{U}(m)$ acts transitively on $\mathrm{Grass}_{2m-1}(\bR^{2m})$, this reproves the classification of K\"ahler structures in \cite{LW}.
	\item For $\mfh=\mathfrak{su}(m)$, the element $e^{2m-1}\otimes e_{2m}-e^{2m}\otimes e_{2m-1}\notin \mathfrak{su}(m)$ and so $\cF_{\mathfrak{su}(m)}=\mathfrak{su}(m-1)}$. Note that this result also follows from the fact that $\mathfrak{su}(m)$ is a super-elliptic totally real subalgebra. Moreover, note that since $\mathrm{SU}(m)$ acts transitively on $\mathrm{Grass}_{2m-1}(\bR^{2m})$, an $2m$-dimensional almost Abelian Lie algebra $\mfg$ admits a Calabi-Yau structure if and only if $f\in  \mathfrak{su}(m-1)$.
    \item For $\mfh=\mathfrak{u}(p,m-p)$, the cases (a) and (b) in Theorem \ref{th:unitarysubalgebras} correspond to
    \begin{equation*}
    \begin{split}
    	\cF_{\mfh}&=\left\{\left.\diag(A,a)\right|A\in \mathfrak{u}(p-1,m-p)\cup\mathfrak{u}(p,m-p-1),\ a\in \bR \right\},\\
    	 \cF_{\mfh}&=\left\{\left.\left(\begin{smallmatrix} A &  &  & \\ & a & & b \\ & & a & c \\ & & & -a \end{smallmatrix}\right)\right|A\in \mathfrak{u}(p-1,m-p-1),\ a,b,c\in \bR \right\},
   \end{split}
\end{equation*}
respectively.
\end{itemize}
\end{example}
\section{Subalgebras $\mfh$ with $\cK_{\mfh}^{(1)}$ of special type}\label{sec:cK(1)specialtype}
In this section, we consider arbitrary linear subalgebras $\mfh$ for which the first prolongation
$\cK_{\mfh}^{(1)}$ of the associated tableau $\cK_{\mfh}$ is of a certain very restrictive form. We note that for totally real subalgebras $\mfh$, Lemma \ref{le:cK1totallyreal} shows that $\cK_{\mfh}^{(1)}$ is of a special form and used this special form of $\cK_{\mfh}^{(1)}$, namely $\cK_{\mfh}^{(1)}\subseteq\alpha_0\otimes \alpha_0\otimes \bR^{2m}$ for some $\alpha_0\in (\bR^{2m-1}_J)^0$, to arrive at the explicit description of $\cF_{\mfh}$ for a totally real subalgebra $\mfh$ in Theorem \ref{th:totallyreal}. However, note that the proof of Theorem \ref{th:totallyreal} used also many other particular properties of totally real subalgebras $\mfh$ and of a complex structure $J$ and we do not expect to arrive at a nice description of $\cF_{\mfh}$ for an arbitrary linear subalgebra $\mfh$ assuming simply $\cK_{\mfh}^{(1)}\subseteq\alpha_0\otimes \alpha_0\otimes \bR^{2m}$ for some $\alpha_0\in (\bR^{n-1})^*$, cf. already the great complexity of the result for totally real subalgebra, and so leave out the investigation of this case here.

Hence, we concentrate on linear subalgebras with other special types of $\cK_{\mfh}^{(1)}$, beginning with the case that this first prolongation is even zero and then concentrate on the case that $\cK_{\mfh}^{(1)}= S^2 \cU\otimes z$ for some subspace $\cU$ of $(\bR^{n-1})^*$ and some $z\in \bR^n$, distinguishing between the cases $z\in \bR^{n-1}$ or $z\notin \bR^{n-1}$. These two cases naturally generalise the case of a non-degenerate metric or a degenerate metric subalgebra, respectively, and we will, conversely, show that under mild assumptions, $\mfh$ contains a non-degenerate or degenerate metric subalgebra, respectively. We like to remark that the condition $\cK_{\mfh}^{(1)}= S^2 \cU\otimes z$ is more restrictive then condition $\cK_{\mfh}^{(1)}\subseteq S^2 (\bR^{n-1})^*\otimes z$ since not any subspace of $\cK_{\mfh}^{(1)}\subseteq S^2 (\bR^{n-1})^*\otimes z$ is of the form $S^2 \cU\otimes z$.

\subsection{$\cK_{\mfh}^{(1)}=\{0\}$}\label{subsec:cK1=0}
Before we state the main result in the case $\cK_\mfh^{(1)}=\{0\}$, we discuss the relation of the condition $\cK_\mfh^{(1)}=\{0\}$ to the vanishing of the \emph{first prolongation $\mfh^{(1)}$} of $\mfh$:
\begin{definition}
Let $\mfh$ be a subalgebra of $\mathfrak{gl}(n,\bR)$. Then the \emph{first prolongation} $\mfh^{(1)}$ of $\mfh$ is defined by
\begin{equation*}
\mfh^{(1)}:=\left(S^2 (\bR^n)^*\otimes \bR^n\right)\cap (\bR^n)^*\otimes \mfh.
\end{equation*}
\end{definition}
\begin{remark}
	We recall some well-known properties related to the first prolongation of $\mfh^{(1)}$:
\begin{itemize}
	\item A torsion-free $H$-structure admits a unique torsion-free $H$-connection if and only if $\mfh^{(1)}=\{0\}$.
	\item
	If $\mfh^{(1)}=\{0\}$, then $\mfh$ is elliptic.
	\item
	If $\mfh$ acts irreducibly, then $\mfh^{(1)}=\{0\}$ with the exception of those $\mfh$ mentioned in \cite[Table B]{Br2}.
\end{itemize}
\end{remark}
Assuming the condition $\cK_{\mfh}^{(1)}=\{0\}$, the vanishing of $\mfh^{(1)}$ can be reformulated as follows:
\begin{lemma}\label{le:cK1=0implymfh1=0}
Let $\mfh$ be a subalgebra with $\cK_{\mfh}^{(1)}=\{0\}$. Then $\mfh^{(1)}=\{0\}$ if and only if $\mfh$ is elliptic.
\end{lemma}
\begin{proof}
As remarked above, the condition $\mfh^{(1)}=\{0\}$ always implies that $\mfh$ is elliptic. So let us assume now that $\mfh$ is elliptic and let $\nabla\in \mfh^{(1)}$. Then $\nabla_u v=0$ for all $u,v\in \bR^{n-1}$ since $\cK_{\mfh}^{(1)}=\{0\}$. But so $\nabla_u\in \mfh$ has rank at most one for all $u\in \bR^{n-1}$ and $\mfh$ being elliptic implies $\nabla_u=0$. Consequently,
\begin{equation*}
\nabla_{e_n} u=\nabla_u e_n= 0,
\end{equation*}
for $u\in \bR^{n-1}$, and so $\nabla_{e_n}$ has rank at most one, implying again that actually $\nabla_{e_n}=0$ as well. Thus, $\nabla=0$ and so $\mfh^{(1)}=\{0\}$.
\end{proof}
Conversely, the condition $\mfh^{(1)}=\{0\}$ does, in general, not imply $\cK_{\mfh}^{(1)}=\{0\}$. We may even have $\cK_{\mfh}^{(1)}=\{0\}$ but $\cK_{\tilde{\mfh}}^{(1)}\neq \{0\}$ for some subalgebra $\tilde{\mfh}$ being conjugate to $\mfh$, so the condition $\cK_{\mfh}^{(1)}=\{0\}$ is not invariant under conjugation:
\begin{example}\label{ex:cK1h1}
\begin{itemize}
	\item Let $\mfh$ be a hypercomplex subalgebra.
	Then $\mfh^{(1)}=\{0\}$ and $\cK_{\mfh}^{(1)}=\{0\}$. These equalities follow directly from $\mfh_2=\{0\}$ (due to $\mfh$ being super-elliptic) and so $\cK_{\mfh}^{(1)}=0$. Then $\mfh^{(1)}=0$ follows from Lemma \ref{le:cK1=0implymfh1=0}.
	\item Let $\mfh$ be a hyperparacomplex subalgebra. By \cite[Proposition 5.1]{ASa}, we then have $\mfh^{(1)}=\{0\}$. Now if $\bR^{2m-1}_J$ is not $E$-invariant, then $\mfh_2=\{0\}$ and so we then also have $\cK_{\mfh}^{(1)}=\{0\}$. However, if $\bR^{2m-1}_J$ is $E$-invariant, we may have $\cK_{\mfh}^{(1)}\neq \{0\}$ if $\mfh_2\neq \{0\}$,  e.g. if $\mfh=\mathfrak{gl}(J_0,E_0, K_0)$, in which case one gets $\cK_{\mfh}^{(1)}=\spa{e^{2m-1}\otimes e^{2m-1}\otimes e_{2m-1}}$.
	\item
	If $\mfh=\mathfrak{so}(p,n-p)$, then $\mfh^{(1)}=\{0\}$ but 
	\begin{equation*}
	\cK^{(1)}_{\mfh}=S^2 (\bR^{n-1})^*\otimes e_n
	\end{equation*}
	Moreover, we may choose a conjugate $\tilde{\mfh}$ of $\mfh$ (if $p\geq 1$ and $n-p-1\geq 1$) for which $\bR^{n-1}$ is degenerate and such that
	\begin{equation*}
	\cK^{(1)}_{\tilde\mfh}=S^2 (\bR^{n-1})^*\otimes e_{n-1}.
	\end{equation*}
\end{itemize}
\end{example}
Allthough giving a general assertion on when $\mfh^{(1)}=\{0\}$ implies $\cK_{\mfh}^{(1)}=\{0\}$ seems not to be possible, we show that this implication is true when $\mfh$ is a super-elliptic subalgebra of $\mathfrak{so}(g)$ for some pseudo-Riemannian metric $g$ on $\bR^n$:
\begin{lemma}\label{le:mfhmetricsuperelliptic}
Let $\mfh$ be a \emph{metric} subalgebra, i.e. a subalgebra of $\mathfrak{so}(g)$ for some pseudo-Riemannian metric $g$ on $\bR^n$. If $\mfh$ is super-elliptic, then $\cK_{\mfh}^{(1)}=\{0\}$.
\end{lemma}
\begin{proof}
Let $\tilde{\nabla}\in \cK_{\mfh}^{(1)}$ be given. Then there is some $\hat{\nabla}\in (\bR^{n-1})^*\otimes \mfh$ with
$\hat{\nabla}_u|_{\bR^{n-1}}=\tilde{\nabla}_u$ for all $u\in \bR^{n-1}$. Then
\begin{equation*}
\begin{split}
g(\hat{\nabla}_u v,w)&=-g(v,\hat{\nabla}_u w)=-g(v,\hat{\nabla}_w u)=g(\hat{\nabla}_w v,u)=g(\hat{\nabla}_v w,u)=-g( w,\hat{\nabla}_v u)\\
&=-g(\hat{\nabla}_u v,w)
\end{split}
\end{equation*}
for all $u,v,w\in \bR^{n-1}$. Consequently, $g(\hat{\nabla}_u v,w)=0$, i.e. $\hat{\nabla}_u v\in (\bR^{n-1})^{\perp}$ for any $v\in \bR^{n-1}$. As $\dim((\bR^{n-1})^{\perp})=1$, this implies that $\hat{\nabla}_u\in \mfh$ has rank at most two. As $\mfh$ is super-elliptic, this implies $\hat{\nabla}=0$ and so $\tilde{\nabla}=0$. Thus, $\cK_{\mfh}^{(1)}=\{0\}$.
\end{proof}
Next, we aim at proving a general result on the form of $\cF_{\mfh}$ in the case that $\cK_{\mfh}^{(1)}=\{0\}$. For this, we set
\begin{equation*}
\mfh_1:=\left\{F\in \mfh|F(\bR^{n-1})=\{0\}\right\},\quad
W_{\mfh}:=\left\{F(e_n)| F\in \mfh_1,\quad F(\bR^n)\subseteq \bR^{n-1}\right\}\subseteq \bR^{n-1}.
\end{equation*}
For the proof of the general result, the following simplification of the computation of $\tilde{\mfk}_{\mfh}$ in the case that not all elements of $\mfh_1$ map $\bR^n$ into $\bR^{n-1}$ will be useful:
\begin{lemma}\label{le:tildemfkequaltildecK}
Let $\mfh$ be a linear subalgebra of $\End(\bR^n)$ which contains an element $F_0\in \mfh_1$ with $v:=F_0(e_n)\notin \bR^{n-1}$. Denote by $\pi_{\bR^{n-1}}$ the projection of $\bR^n$ onto $\bR^{n-1}$ along $\spa{v}$. Then
\begin{equation*}
\tilde{\mfk}_{\mfh}=\left\{\left.\pi_{\bR^{n-1}}\circ F|_{\bR^{n-1}}\right|F\in \mfh\right\}
\end{equation*}
\end{lemma}
\begin{proof}
The inclusion $\tilde{\mfk}_{\mfh}\subseteq \left\{\left.\pi_{\bR^{n-1}}\circ F|_{\bR^{n-1}}\right|F\in \mfh\right\}$ is clear. So let $f\in \left\{\left.\pi_{\bR^{n-1}}\circ F|_{\bR^{n-1}}\right|F\in \mfh\right\}$, i.e. there is some $F\in \mfh$ and some $\alpha\in (\bR^{n-1})^*$ such that
\begin{equation*}
F|_{\bR^{n-1}}=f+\alpha\otimes v.
\end{equation*}
As $\mathrm{im}(F_0)=\spa{v}$ and $\bR^{n-1}=\ker(F_0)$, there exists some $\beta\in (\bR^n)^*$ with $F_0=\beta\otimes v$,
 $\beta(\bR^{n-1})=\{0\}$ and $\beta(v)\neq 0$.
We then set
\begin{equation*}
G:=F-\frac{1}{\beta(v)} [F_0,F]
\end{equation*}
and observe that $G\in \mfh$. Moreover,
\begin{equation*}
\begin{split}
G|_{\bR^{n-1}}&=F|_{\bR^{n-1}}-\frac{1}{\beta(v)} [F_0,F]|_{\bR^{n-1}}=f+\alpha\otimes v-\frac{1}{\beta(v)} F_0(f+\alpha\otimes v)\\
&=f+\alpha\otimes v-\frac{1}{\beta(v)} \beta(f+\alpha\otimes v)\otimes v=f+\alpha\otimes v-\frac{\beta(v)}{\beta(v)}\alpha\otimes v=f,
\end{split}
\end{equation*}
which implies that $G\in \mfk_{\mfh}$, and so $f\in \tilde{\mfk}_{\mfh}$, proving the stated assertion.
\end{proof}
\begin{theorem}\label{th:cK1=0}
Let $\mfh$ be a linear subalgebra with $\cK^{(1)}_{\mfh}=\{0\}$. Then 
\begin{equation*}
		\cF_{\mfh}=\tilde{\mfk}_{\mfh}+ (\bR^{n-1})^*\otimes W_{\mfh}.
	\end{equation*}
\end{theorem}
\begin{proof}
Let $\nabla\in \cT^{-1}(\End(\bR^{n-1}))\subseteq\cD_{\mfh}$ be given. Since $\cK^{(1)}_{\mfh}=\{0\}$, we have $\nabla_u v=0$ for all $u,v\in \bR^{n-1}$ and so $\nabla_u\in \mfh_1$ for all $u\in \bR^{n-1}$.
We distinguish now two cases:

\begin{itemize}
	\item Let us first assume that for any element $F\in \mfh_1$ we have $F(e_n)\in \bR^{n-1}$. In this case, $\nabla_u e_n\in W_{\mfh}\subseteq\bR^{n-1}$ for all $u\in \bR^{n-1}$ and so
	$\nabla\in \cT^{-1}(\End(\bR^{n-1}))$ implies $\nabla_{e_n}(\bR^{n-1})\subseteq \bR^{n-1}$, i.e. $\nabla_{e_n}|_{\bR^{n-1}}\in \tilde{\mfk}_{\mfh}$. Thus, 
	\begin{equation*}
	f:=\cT(\nabla)=\nabla_{e_n}|_{\bR^{n-1}}-\nabla e_n |_{\bR^{n-1}}\in \tilde{\mfk}_{\mfh}+ (\bR^{n-1})^*\otimes W_{\mfh},
	\end{equation*}
which shows $\cF_{\mfh}\subseteq \tilde{\mfk}_{\mfh}+ (\bR^{n-1})^*\otimes W_{\mfh}$. Conversely, let
	\begin{equation*}
	f=f_1+f_2\in \tilde{\mfk}_{\mfh}+ (\bR^{n-1})^*\otimes W_{\mfh}
\end{equation*}
be given. Defining $\nabla\in \cD_{\mfh}$ by
\begin{equation*}
\nabla_u v:=0,\quad \nabla_u e_n:=-f_2(u),\quad \nabla_{e_n}:=F_1.
\end{equation*}
for $F_1\in \mfh$ with $F_1(\bR^{n-1})\subseteq \bR^{n-1}$ and $F_1|_{\bR^{n-1}}=f_1$, we have $\cT(\nabla)=f\in \End(\bR^{n-1})$. This proves the other inclusion $\cF_{\mfh}\supseteq \tilde{\mfk}_{\mfh}+ (\bR^{n-1})^*\otimes W_{\mfh}$.
\item
Assume now that there is some element $F_0\in \mfh_1$ with $v:=F_0(e_n)\notin \bR^{n-1}\notin \bR^{n-1}$ and let $\pi_{\bR^{n-1}}$ denote the projection of $\bR^n$ onto $\bR^{n-1}$ along $\spa{v}$. We may then write $f_2:=\nabla e_n|_{\bR^{n-1}}$ as
\begin{equation*}
f_2=h_2+\alpha_2\otimes v
\end{equation*}	
for certain $h_2\in (\bR^{n-1})^*\otimes W_{\mfh}$ and $\alpha_2\in (\bR^{n-1})^*$ and 
$f_1:=\nabla_{e_n}|_{\bR^{n-1}}$ as
\begin{equation*}
f_1:=h_1+\alpha_1\otimes v
\end{equation*}
for $h_1:=\pi_{\bR^{n-1}}\circ f_1 \in \End(\bR^{n-1})$ and $\alpha_1\in (\bR^{n-1})^*$. Then $h_1\in \tilde{\mfk}_{\mfh}$ due to Lemma \ref{le:tildemfkequaltildecK} and
$\nabla\in \cT^{-1}(\End(\bR^{n-1}))$ forces $\alpha_2=\alpha_1$. But so
\begin{equation*}
f:=\cT(\nabla)=h_1+\alpha_1-h_2-\alpha_2=h_1-h_2\in \tilde{\mfk}_{\mfh}+ (\bR^{n-1})^*\otimes W_{\mfh}.
\end{equation*}
This shows $\cF_{\mfh}\subseteq\tilde{\mfk}_{\mfh}+ (\bR^{n-1})^*\otimes W_{\mfh}$ and the other inclusion follows similarly to the proof of the other inclusion in the first case.
\end{itemize}
\end{proof}
\begin{remark}
Notice that for \emph{any} linear subalgebra $\mfh$, i.e. also one with $\cK_{\mfh}^{(1)}\neq \{0\}$, one always has 
\begin{equation*}
(\bR^{n-1})^*\otimes W_{\mfh}\subseteq \cF_{\mfh}.
\end{equation*}
To prove this, let $f\in(\bR^{n-1})^*\otimes W_{\mfh}$ and define $\nabla\in  \cT^{-1}(\End(\bR^{n-1}))\subseteq \cD_{\mfh}$ by
\begin{equation*}
\nabla_{e_n}:=0,\quad \nabla_u v:=0,\quad \nabla_u e_n:=-f(u)
\end{equation*}
for all $u,v\in \bR^{n-1}$ and observe that $\cT(\nabla)=f$. We notice that then $\nabla$ is even flat and so the associated special $H$-structure on the associated almost Abelian Lie algebra is left-invariantly flat. In particular, if
\begin{equation*}
	\mfh:=\left\{\left.\begin{pmatrix} A & v \\
		0 & 0\\
	\end{pmatrix}\right|A\in \tilde{\mfh},\ v\in \bR^{n-1}
	\right\}
\end{equation*}
for any subalgebra $\tilde{\mfh}$ of $\mathfrak{gl}(n-1,\bR)$, then $\cF_{\mfh}=\End(\bR^{n-1})$ and any special $H$-structure is left-invariantly flat.
\end{remark}
\begin{example}
If $\tilde{\mfh}$ is a subalgebra of $\mathfrak{gl}(m,\bR)$ for some $m\leq n-1$ with $\tilde{\mfh}^{(1)}=\{0\}$, then consider
\begin{equation*}
\mfh:=\left\{\left.\begin{pmatrix} A & 0 & 0 \\
	                         0 & 0 & v \\
	                         0 & 0 & 0
	                        \end{pmatrix}\right|A\in \tilde{\mfh},\ v\in \bR^{n-1-m}
       \right\}.
\end{equation*}
Then Theorem \ref{th:cK1=0} implies
\begin{equation*}
\cF_{\mfh}=\left\{\left.\begin{pmatrix} A & 0 \\
	B & C \\
\end{pmatrix}\right|A\in \tilde{\mfh},\ B\in \bR^{(n-1-m)\times m},\ C\in \bR^{(n-1-m)\times (n-1-m)}
\right\}.
\end{equation*}
\end{example}
We note the following direct consequence of Theorem \ref{th:cK1=0} observing that the ellipticity of $\mfh$ implies $\mfh_1=\{0\}$:
\begin{corollary}\label{co:cK1=0elliptic}
Let $\mfh$ be an elliptic linear subalgebra with $\cK_{\mfh}^{(1)}=\{0\}$. Then $\cF_{\mfh}=\tilde{\mfk}_{\mfh}$ and so a special $H$-structure on an almost Abelian Lie algebra $\mfg$ is torsion-free if and only if it is left-invariantly flat, which is the case if and only if $f\in \tilde{\mfk}_{\mfh}$.
\end{corollary}
\begin{remark}
Corollary \ref{co:cK1=0elliptic} together with Example \ref{ex:cK1h1} reproves Corollary \ref{co:hypercomplex} and Corollary \ref{co:flathyperparacomplex} (b) and may easily be used to reprove Theorem \ref{th:hyperparacomplexclass} (b).
\end{remark}
Corollary \ref{co:cK1=0elliptic} and Lemma \ref{le:mfhmetricsuperelliptic} directly imply:
\begin{corollary}
Let $\mfh$ be a super-elliptic metric subalgebra. Then $\cF_{\mfh}=\tilde{\mfk}_{\mfh}$ and so an $H$-structure $P$ on an almost Abelian Lie algebra is left-invariantly flat if and only if it is torsion-free if and only if $f\in \tilde{\mfk}_{\mfh'}$ for some subalgebra $\mfh'$ which is conjugated to $\mfh$.
\end{corollary}
\begin{example}
\begin{itemize}
	\item 
	All of the special Berger holonomy algebras $\mathfrak{su}(m)$, $\mathfrak{sp}(k)$, $\mathfrak{sp}(k)\mathfrak{sp}(1)$ ($k\geq 2$), $\mfg_2$ and $\mathfrak{spin}(7)$ are super-elliptic metric subalgebras. The corresponding $\tilde{\mfk}_{\mfh}$, cf. also Table \ref{table:charsubalgs}, are given by $\mathfrak{su}(m-1)$, $\mathfrak{sp}(k-1)$, $\mathfrak{sp}(k-1)\mathfrak{sp}(1)$, $\mathfrak{su}(3)$ and $\mfg_2$, respectively. In the case $\mfh=\mfg_2$, this reproves the result of the author from \cite{Fr2}.
	\item
	Any subalgebra $\mfh$ of $\Delta \mathfrak{so}(m)\subseteq \mathfrak{so}(m,m)$ is a super-elliptic metric subalgebra. Hence, Corollary \ref{co:cK1=0elliptic} gives a second proof of Theorem \ref{th:hyperparacomplexelliptic} for the case of hyperparacomplex subalgebras induced by subalgebras of $\mathfrak{so}(m)$.
\end{itemize}
\end{example}
\subsection{$\cK_{\mfh}^{(1)}=S^2 \cU\otimes v$ for $v\notin \bR^{n-1}$}\label{subsec:cK1=S2Uv}
In this subsection, we look at the case that $\cK^{(1)}=S^2 \cU\otimes v$ for some subspace $\cU$ of $(\bR^{n-1})^*$ and some $v\in \bR^n\setminus \bR^{n-1}$. 
Note that so $\bR^n=\bR^{n-1}\oplus \spa{v}$ and we will use this identification to identify elements in $(\bR^{n-1})^*$ with elements in the annihilator $v^0\subseteq (\bR^n)^*$ of $v$. We will first show that \emph{non-degenerate metric} subalgebras satisfy the condition $\cK^{(1)}=S^2 \cU\otimes v$ for some $v\notin \bR^{n-1}$.Note that, conversely, Proposition \ref{pro:nusharp} below shows that any elliptic $\mfh$ with {$\cK_{\mfh}^{(1)}=S^2 \cU\otimes v$ for some $v\notin \bR^{n-1}$ contains a subalgebra which is non-degenerate metric:
\begin{lemma}\label{le:nondegmetriccK1}
Let $\mfh$ be a \emph{non-degenerate} metric subalgebra, i.e. $\mfh$ is a subalgebra of $\mathfrak{so}(g)$ for some pseudo-Riemannian metric $g$ on $\bR^n$ such that $g|_{\bR^{n-1}\times \bR^{n-1}}$ is non-degenerate. Then there exists a subspace $\cU\subseteq (\bR^{n-1})^*$ such that for $v\in (\bR^{n-1})^{\perp}$, $v\neq 0$, we have
\begin{equation*}
\cK_{\mfh}^{(1)}=S^2 \cU\otimes v.
\end{equation*}
\end{lemma}
\begin{proof}
Let $\tilde{\nabla}\in \cK^{(1)}$ be given. It suffices to show that $g(\tilde{\nabla}_{u_1} u_2, u_3)=0$ for all $u_1,u_2,u_3\in \bR^{n-1}$ as then $\tilde{\nabla}_{u_1} u_2\in  (\bR^{n-1})^{\perp}=\spa{v}$ and so the symmetry properties of $\cK^{(1)}$ yield the result. However, $g(\tilde{\nabla}_{u_1} u_2, u_3)=0$ follows by the same computation as in the proof of Lemma \ref{le:mfhmetricsuperelliptic}.
\end{proof}
We observe that for an arbitrary linear subalgebra $\mfh$ with $\cK_{\mfh}^{(1)}=S^2 \cU\otimes v$ the subspace $\cU\subseteq (\bR^{n-1})^*$ may be described as follows:
\begin{lemma}\label{le:charofcU}
Let $\mfh$ be a subalgebra with $\cK_{\mfh}^{(1)}=S^2 \cU\otimes v$ for some subspace $\cU\subseteq (\bR^{n-1})^*$ and some $v\in \bR^n\setminus \bR^{n-1}$. Setting
\begin{equation*}
\mfh_v:=\left\{F\in \mfh|F(\bR^{n-1})\subseteq \spa{v}\right\},
\end{equation*}
we have
\begin{equation*}
\cU=\left\{\alpha\in (\bR^{n-1})^*\left|F|_{\bR^{n-1}}=\alpha\otimes v \textrm{ for some } F\in \mfh_v\right.\right\}.
\end{equation*}
\end{lemma}
\begin{proof}
Let $\alpha\in \cU$ be given. Then $\alpha\otimes \alpha\otimes v\in S^2 \cU\otimes v=\cK_{\mfh}^{(1)}$ and so $\alpha\otimes v\in \cK_{\mfh}$, i.e. there is some $F\in \mfh$ with $F|_{\bR^{n-1}}=\alpha\otimes v$. But then $F\in \mfh_v$ and so $\alpha$ in the space on the right hand side of the claimed equality.

Conversely, if $\alpha\in (\bR^{n-1})^*$ is such that there is some $F\in \mfh_v$ with $F|_{\bR^{n-1}}=\alpha\otimes v$, then $\alpha\otimes v\in \cK_{\mfh}$. Thus, $\alpha\otimes \alpha\otimes v\in \cK_{\mfh}^{(1)}$ and so $\alpha\in  \cU$. This proves the claimed equality.
\end{proof}
The general result on $\cF_{\mfh}$ in Theorem \ref{th:cK1=S2cUv} below distinguishes the case whether $\mfh$ is elliptic or not. To prove this result, we first start by showing some results on the form of $\mfh$ which distinguishes three different cases, two of them which are not elliptic:
\begin{lemma}\label{le:nu}
	Let $\mfh$ be an linear subalgebra with $\cK_{\mfh}^{(1)}=S^2 \cU\otimes v$ for some non-zero subspace $\cU\subseteq (\bR^{n-1})^*$ and some $v\in \bR^n\setminus \bR^{n-1}$. Moreover, let $\gamma\in (\bR^{n-1})^0\setminus \{0\}\setminus \{0\}$. Then:
	\begin{enumerate}
	\item[(i)]
	 Either $\cU\otimes v\subseteq \mfh$
	 \item[(ii)]
	 or $\dim(\cU)=1$, i.e. $\cU=\spa{\alpha}$ for some $\alpha\in (\bR^{n-1})^*\setminus \{0\}$ and there exists some $0\neq \lambda\in \bR$ such that $(\alpha+\lambda \gamma)\otimes v\in \mfh$
	 \item[(iii)]
	 or there exists an injective map $\nu:\cU\rightarrow \bR^{n-1}$ and a map $\rho:\cU\rightarrow \bR$ which are uniquely defined by the property that for $\alpha\in \cU$ the element $\nu(\alpha)\in \bR^{n-1}$ and the real number $\rho(\alpha)\in \bR$ satisfy
	 \begin{equation*}
	 \alpha\otimes v-\gamma\otimes (\nu(\alpha)+\rho(\alpha)\, v)\in \mfh. 
	\end{equation*}
	\end{enumerate}
	Moreover, $\mfh_1\subseteq \spa{\gamma\otimes v}$ in case (i) and $\mfh_1=\{0\}$ in case (ii) and (iii), and $\mfh$ is elliptic if and only if $\mfh$ is as in case (iii).
	\end{lemma}
\begin{proof}
Let us first assume that (i) holds and let $\alpha\in \cU\setminus \{0\}$. Then $F:=\alpha\otimes v\in \mfh$. Moreover, let $\tilde{F}\in \mfh_1$ be given. Then $\tilde{F}=\gamma\otimes (\tilde{u}+a\, v)\in \mfh_1$ for some $\tilde{u}\in \bR^{n-1}$ and $a\in \bR$ and
\begin{equation*}
[F,\tilde{F}](u)=-\tilde{F}(\alpha(u) v)=-\gamma(v)\, \alpha(u)\, (\tilde{u}+a\, v),
\end{equation*}
i.e. $[F,\tilde{F}]|_{\bR^{n-1}}=-\gamma(v)\, \alpha\otimes (\tilde{u}+a\, v)$. Hence $\alpha\otimes  \alpha\otimes (\tilde{u}+a\, v)\in \cK_{\mfh}^{(1)}$ and so we must have $\tilde{u}=0$. This shows $\mfh_1\subseteq \spa{\gamma\otimes v}$.

Let us now assume for the rest of the proof that (i) does not hold, i.e. that there is some element $0\neq\beta\in \cU$ such that $\beta\otimes v\notin \mfh$. Since $\beta\otimes v\in \cK_{\mfh}$, there exist $w_1\in \bR^{n-1}$, $a_1\in \bR$ such that $F_1:=\beta\otimes v-\gamma\otimes (w_1+a_1\, v)\in \mfh$ and such that $w_1+a_1\, v\neq 0$.

We first show the existence of maps $\nu:\cU\rightarrow \bR^{n-1}$ and $\rho:\cU\rightarrow \bR$ uniquely defined by the property that $\alpha\otimes v-\gamma\otimes (\nu(\alpha)+\rho(\alpha)\, v)\in \mfh$. We note that these maps are well-defined if and only if $\mfh_1=\{0\}$. So let us asssume that theses maps are not well-defined and so there is some element $0\neq F_2:=\gamma\otimes (w_2+a_2\, v)\in \mfh_1\subseteq \mfh$.
 But then
\begin{equation*}
[F_1,F_2](u)=-F_2(\beta(u) v)=-\beta(u)\, \gamma(v)\, (w_2+a_2\, v)
\end{equation*}
for all $u\in \bR^{n-1}$, i.e. $[F_1,F_2]|_{\bR^{n-1}}=-\gamma(v)\, \beta\otimes (w_2+a_2\, v)$. Since $[F_1,F_2]\in \mfh$, we get $\beta\otimes \beta\otimes (w_2+a_2\,v)\in \cK_{\mfh}^{(1)}$, and so the condition $\cK_{\mfh}^{(1)}=S^2\cU\otimes v$ forces $w_2=0$. But then $a_2\neq 0$ and $\beta\otimes v\in \mfh$. Hence, also
\begin{equation*}
G_1:=F_1-a_1\, \beta\otimes v=\beta\otimes v-\gamma\otimes w_1\in \mfh
\end{equation*}
 Moreover,
\begin{equation*}
\begin{split}
[F_1,F_2](v)&=F_1(a_2\gamma(v)\, v)+F_2(\gamma(v) (w_1+a_1\, v))\\
&=-a_2\gamma(v)\, w_1-\gamma(v)^2 a_1 a_2\, v+\gamma(v)^2 a_2 a_1\, v=-a_2\gamma(v)\, w_1
\end{split}
\end{equation*}
and so
\begin{equation*}
G_2:=-\frac{1}{a_2 \gamma(v)}[F_1,F_2]=\beta\otimes v+\gamma\otimes w_1\in \mfh.
\end{equation*}
Thus, $\tfrac{1}{2} (G_1+G_2)=\beta\otimes v\in \mfh$, a contradiction. Hence, both $\nu$ and $\rho$ are well-defined and $\mfh_1=\{0\}$.

Next, we want to show that either $\nu\equiv 0$ or $\nu$ is injective. For this, let $\alpha\in \ker(\nu)$. Then $H_1:=\left(\alpha-\rho(\alpha)\gamma\right)\otimes v\in \mfh$.
Moreover, let $H_2:=\beta\otimes v-\gamma\otimes (\nu(\beta)+\rho(\beta) v)\in \mfh$ be given. Then $H:=[H_1,H_2]+\rho(\alpha)\gamma(v) H_2\in \mfh$ and
\begin{equation*}
\begin{split}
H(u)&=([H_1,H_2]+\rho(\alpha)\gamma(v) H_2)(u)=H_1(\beta(u) v)-H_2(\alpha(u)v)+\rho(\alpha)\gamma(v)\beta(u) v\\
&=-\beta(u)\rho(\alpha)\gamma(v) v+\alpha(u)\gamma(v)\nu(\beta)+\alpha(u)\rho(\beta)\gamma(v) v+\rho(\alpha)\gamma(v)\beta(u) v\\
&=\gamma(v)\alpha(u)(\nu(\beta)+\rho(\beta) v)
\end{split}
\end{equation*} 
for all $u\in \bR^{n-1}$. Thus, $\alpha\otimes \alpha\otimes (\nu(\beta)+\rho(\beta)\, v)\in \cK_{\mfh}^{(1)}$ and so we must either have $\alpha=0$, i.e. $\ker(\nu)=\{0\}$ and then $\nu$ is injective, or $\nu(\beta)=0$, i.e. $\nu\equiv 0$.

Now if $\nu\equiv 0$, then $(\alpha-\rho(\alpha)\, \gamma)\otimes v\in \mfh$ for any $\alpha\in \cU$. Thus, if $\dim(\ker(\rho))\geq 1$ and so there is some $0\neq \beta\in \ker(\rho)$ such that $\beta\otimes v\in \mfh$, a contraditction. Hence $\ker(\rho)=\{0\}$, which forces $\dim(\cU)=1$ and that $(\alpha+\lambda \gamma)\otimes v\in \mfh$ for $\alpha\in \cU$ with $\cU=\spa{\alpha}$ and for $\lambda:=-\rho(\alpha)\neq 0$. This is case (ii).

Finally, if $\nu$ is injective, we are in case (iii) and to prove the assertion we only need to show that $\mfh$ is elliptic. For this, assume by contradiction that $F:=(\alpha+\lambda \gamma)\otimes (u+\mu v)$ is an element of rank one in $\mfh$ for certain $\alpha\in (\bR^{n-1})^*$, $u\in \bR^{n-1}$ and $a,b\in \bR$. Then $F|_{\bR^{n-1}}=\alpha\otimes (u+\mu v)$ and so $\alpha\otimes \alpha\otimes (u+\mu v)\in \cK_{\mfh}^{(1)}$, forcing $u=0$ and $\alpha\in \cU$. Since $F$ has rank one, we must have $\mu\neq 0$ and so $\alpha\otimes v+\gamma\otimes (-\lambda v)$. Since we are in case (iii), this implies $\nu(\alpha)=0$ and so, as $\nu$ is injective, $\alpha=0$, a contradiction. Thus, $\mfh$ is elliptic.
\end{proof} 
We concentrate now first on the case that $\mfh$ is elliptic, i.e,. on case (iii) in Lemma \ref{le:nu}. In this case, we set
\begin{equation*}
U:=\nu(\cU)\subseteq \bR^{n-1}.
\end{equation*}
We derive first some results on the characteristic subalgebra $\tilde{\mfk}_{\mfh}$ as well as on the bilinear form $(\alpha,\beta)\mapsto \alpha(\nu(\beta))$ on $\cU$:
\begin{lemma}\label{le:bilinearnusymmetric}
Let $\mfh$ be an elliptic linear subalgebra with $\cK_{\mfh}^{(1)}=S^2 \cU\otimes v$ for some non-zero subspace $\cU$ of $(\bR^{n-1})^*$ and some $v\in \bR^n\setminus \bR^{n-1}$. Moreover, let $\beta\in (\bR^{n-1})^0\setminus \{0\}$. Then
\begin{equation*}
\left\{\left. \alpha\otimes \nu(\beta)-\beta\otimes \nu(\alpha)\right|\alpha,\beta\in \cU\right\}\subseteq \tilde{\mfk}_{\mfh},
\end{equation*}
$\rho\equiv 0$ and the bilinear map
\begin{equation*}
\cU\times \cU \ni (\alpha,\beta)\mapsto h(\alpha,\beta):= \alpha(\nu(\beta))\in \bR
\end{equation*}
is symmetric.
\end{lemma}
\begin{proof}
Let linearly independent $\alpha,\beta\in \cU$ be given. W.l.o.g, we may assume that $\alpha\in \ker(\rho)$. Then the commutator $F_3:=[F_1,F_2]\in \mfh$ of
\begin{equation*}
F_1:=\alpha\otimes v-\gamma\otimes \nu(\alpha),\quad F_2:=\beta\otimes v-\gamma\otimes (\nu(\beta)+\rho(\beta)\, v)
\end{equation*}
reads
\begin{equation*}
F_3=\gamma(v)(\alpha\otimes \nu(\beta)-\beta\otimes \nu(\alpha))+\gamma(v)\rho(\beta)\, \alpha\otimes v+\gamma(v)\rho(\beta)\, \gamma\otimes \nu(\alpha)+\gamma\otimes (\beta(\nu(\alpha))-\alpha(\nu(\beta))\, v.
\end{equation*}
Then the element $G:=\frac{1}{\gamma(v)}F_3-\rho(\beta)F_1$ of $\mfh$ is given by
\begin{equation*}
G=\alpha\otimes \nu(\beta)-\beta\otimes \nu(\alpha)+2\rho(\beta)\, \gamma\otimes \nu(\alpha)+\frac{1}{\gamma(v)}\gamma\otimes (\beta(\nu(\alpha))-\alpha(\nu(\beta)))\,v.
\end{equation*}
Hence $G\in \mfk_{\mfh}$ and so $\alpha\otimes \nu(\beta)-\beta\otimes \nu(\alpha)\in \tilde{\mfk}_{\mfh}$. Moreover,
\begin{equation*}
\begin{split}
H\coloneqq & [G,F_1]\\
=&2\rho(\beta)\gamma(v)\,\alpha\otimes \nu(\alpha)+(\beta(\nu(\alpha))-2 \alpha(\nu(\beta)))\, \alpha\otimes v+\alpha(\nu(\alpha))\, \beta\otimes v\\
-&\alpha(\nu(\alpha))\, \gamma\otimes \nu(\beta)+(2\beta(\nu(\alpha))-\alpha(\nu(\beta)))\gamma\otimes \nu(\alpha)\\
-&2\rho(\beta)\alpha(\nu(\alpha))\, \gamma\otimes v.
\end{split}
\end{equation*}
Thus,
\begin{equation*}
(H-\alpha(\nu(\alpha))\, F_2)|_{\bR^{n-1}}=\alpha\otimes \left(2\rho(\beta)\gamma(v)\, \nu(\alpha)+(\beta(\nu(\alpha))-2 \alpha(\nu(\beta)))\, v\right),
\end{equation*}
and so
\begin{equation*}
\alpha\otimes\alpha\otimes \left(2\rho(\beta)\gamma(v)\, \nu(\alpha)+(\beta(\nu(\alpha))-2 \alpha(\nu(\beta)))\, v\right)\in \cK_{\mfh}^{(1)}.
\end{equation*}
As $\nu(\alpha)\neq 0$, this shows $\rho(\beta)=0$, i.e. $\rho\equiv 0$. But then
\begin{equation*}
\begin{split}
H-\alpha(\nu(\alpha))\, F_2&=(\beta(\nu(\alpha))-2 \alpha(\nu(\beta)))\, \alpha\otimes v-(\alpha(\nu(\beta))-2\beta(\nu(\alpha)))\gamma\otimes \nu(\alpha)\in \mfh_v.
\end{split}
\end{equation*}
This shows $\beta(\nu(\alpha))-2 \alpha(\nu(\beta))=\alpha(\nu(\beta))-2\beta(\nu(\alpha))$, i.e. $\alpha(\nu(\beta))=\beta(\nu(\alpha))$.
\end{proof}
The symmetric bilinear form $h$ is not a pseudo-metric on $(\bR^n)^*$ as it is only defined on $\cU\subseteq (\bR^{n-1})^*$ and it is, in general, even degenerate. However, we can extend it to a pseudo-metric $g$ on $(\bR^n)^*$ such that the map $\nu$ gets the associated musical isomorphism $(\cdot)^{\sharp_g}$. Moreover, then one may identify a certain subalgebra $\mfh_g$ of $\mfh$ which is non-degenerate metric with respect to $g$.
 \begin{proposition}\label{pro:nusharp}
 Let $\mfh$ be an elliptic linear subalgebra with $\cK_{\mfh}^{(1)}=S^2 \cU\otimes v$ for some non-zero subspace $\cU$ of $(\bR^{n-1})^*$ and some $v\in \bR^n\setminus \bR^{n-1}$ and let $\gamma\in (\bR^{n-1})^0\setminus \{0\}$ with $\gamma(v)=\epsilon \in \{-1,1\}$ be given. Then there exists a pseudo-metric $g$ on $\bR^n\cong (\bR^n)^*$ such that $g|_{\cU\times \cU}=h$, such that $\nu=(\cdot)^{\sharp_g}|_{\cU}$, such that $g(v,v)=\epsilon$, $v^{\sharp}=\gamma$ and such that
 \begin{equation*}
 \begin{split}
 \mfh_g\coloneqq&\left\{\left. \alpha\otimes \nu(\beta)-\beta\otimes \nu(\alpha)\right|\alpha,\beta\in \cU\right\}\oplus \mfh_v\\
 =&\left\{\left. \alpha\otimes \nu(\beta)-\beta\otimes \nu(\alpha),\ \alpha\otimes v-\gamma\otimes \nu(\alpha)\right|\alpha,\beta\in \cU\right\}\subseteq \mfh\cap \mathfrak{so}(g),
 \end{split}
 \end{equation*}
i.e. $\mfh_g$ is a subalgebra of $\mfh$ which is non-degenerate and metric with respect to $g$.
 \end{proposition}
 \begin{proof}
 By Sylvester's law of inertia, we may choose a basis $\alpha_1,\ldots,\alpha_m$ of $\cU$ such that there exists $l\in \{0,\ldots,m\}$ and $\epsilon_1,\ldots,\epsilon_l\in \{-1,1\}$ with
\begin{equation*}
\alpha_j(\nu(\alpha_i))=\alpha_i(\nu(\alpha_j))=h(\alpha_i,\alpha_j)=\delta_{ij} \epsilon_i,\quad  \alpha_r(\nu(\alpha_k))=\alpha_k(\nu(\alpha_r))=h(\alpha_r,\alpha_k)=0
\end{equation*}
for all $i,j=1,\ldots,l$, $r=l+1,\ldots,m$ and $k=1,\ldots,m$. We note that so $\cU^{\perp_h}=\spa{\alpha_{l+1},\ldots,\alpha_m}$ and that the space $\nu(\cU^{\perp_h})=\spa{\nu(\alpha_{l+1}),\ldots \nu(\alpha_m)}\subseteq \bR^{n-1}$ is annihilated by all elements in $\cU$. As the natural pairing between $(\bR^{n-1})^*$ and $\bR^{n-1}$ is non-degenerate, 
we may find elements $\alpha_{m+1},\ldots,\alpha_{2m-l}\in (\bR^{n-1})^*$ such that
\begin{equation*}
\alpha_{m+s}(\nu(\alpha_i))=0, \quad \alpha_{m+s}(\nu(\alpha_{l+j}))=\delta_{js}
\end{equation*}
for all $i=1,\ldots,l$, $j,s=1,\ldots,m-l$. Necessarily, then $\alpha_1,\ldots, \alpha_{2m-l}$ are linearly independent and the annihilator $U^0$ of $U=\nu(\cU)$ is a subspace complementary to $\spa{\alpha_1,\ldots,\alpha_l,\alpha_{m+1},\ldots \alpha_{2m-l}}$ in $(\bR^{n-1})^*$. Consequently, we may extend the linearly independent set $\alpha_1,\ldots, \alpha_{2m-l}$ to a basis $\alpha_1,\ldots,\alpha_{n-1}$ of $(\bR^{n-1})^*$ such that $\alpha_{l+1},\ldots,\alpha_m,\alpha_{2m-l+1},\ldots,\alpha_{n-1}$ is a basis of $U^0$. Finally we set $\alpha_n:=\gamma$ and define a bilinear symmetric form $g$ on $\bR^n$ by letting
 \begin{equation*}
 \begin{pmatrix}
 \epsilon_1 & & & & & & \\
 & \ddots & & & & & & \\
 & & \epsilon_l & & & &\\
 & & & 0 & I_{m-l} &  &\\
 & & & I_{m-l} & 0 & & \\
 & & & & & I_{n-1-2m+l} & \\
 & & & & & & \epsilon
 \end{pmatrix}
 \end{equation*}
 be the matrix of the bilinear form with respect to the basis $(\alpha_1,\ldots,\alpha_n)$. It is then immediate that $g$ is non-degenerate, i.e. $g$ is a pseudo-metric, and satisfies $g|_{\cU\times \cU}=h$.
 
 Next, we show that
 \begin{equation*}
 g(\alpha,\beta)=\alpha(\nu(\beta))
 \end{equation*}
  for all $\alpha\in (\bR^n)^*$ and $\beta\in \cU$. For this, we note that for $\alpha\in \cU$, we immediately have $g(\alpha,\beta)=h(\alpha,\beta)=\alpha(\nu(\beta))$. Moreover, if $\alpha\in \spa{\alpha_{2m-l+1},\ldots\alpha_n}\subseteq U^0$, then
  $\alpha(\nu(\beta))=0=g(\alpha,\beta)$. So we only have to consider the case that $\alpha\in \spa{\alpha_{m+1},\ldots,\alpha_{2m-l}}$, i.e. $\alpha=\alpha_{m+s}$ for some $s\in \{1,\ldots,m-l\}$. Furthermore, we may assume that $\beta=\alpha_i$ for some $i\in \{1,\ldots,m\}$. If $i\in \{1,\ldots,l\}$, then
  \begin{equation*}
  g(\alpha,\beta)=g(\alpha_{m+s},\alpha_i)=0=\alpha_{m+s}(\nu(\alpha_i))=\alpha(\nu(\beta)),
  \end{equation*}
  whereas if $i=l+j$ for some $j\in \{1,\ldots,m-l\}$, we get
    \begin{equation*}
  g(\alpha,\beta)=g(\alpha_{m+s},\alpha_{l+j})=\delta_{sj}=\alpha_{m+s}(\nu(\alpha_{l+j})=\alpha(\nu(\beta)).
  \end{equation*}
  This show the identity $g(\alpha,\beta)=\alpha(\nu(\beta))$ for all $\alpha\in (\bR^n)^*$ and $\beta\in \cU$. But then
  \begin{equation*}
  g(\beta^{\sharp},\alpha^{\sharp})=g(\beta,\alpha)=\alpha(\nu(\beta))=g(\alpha^{\sharp},\nu(\beta))=g(\nu(\beta),\alpha^{\sharp})
  \end{equation*}
 for all $\alpha\in (\bR^n)^*$ and $\beta\in \cU$, showing that $\beta^{\sharp}=\nu(\beta)$ for all $\beta\in \cU$. This shows that for all $\alpha,\beta\in \cU$ the element
 \begin{equation*}
 \alpha\otimes \nu(\beta)-\beta\otimes \nu(\alpha)=\alpha\otimes \beta^{\sharp}-\beta\otimes \alpha^{\sharp}
 \end{equation*}
 is in $\mathfrak{so}(g)$ and it is also in $\mfh$ by (the proof of) Lemma \ref{le:bilinearnusymmetric}. Moreover,
 \begin{equation*}
g(v,\alpha^{\sharp})=\alpha(v)=0=g(\gamma,\alpha)=g(\gamma^{\sharp},\alpha^{\sharp}),\quad g(v,\gamma^{\sharp})=\gamma(v)=\epsilon=g(\gamma,\gamma)=g(\gamma^{\sharp},\gamma^{\sharp})
 \end{equation*}
 for all $\alpha\in (\bR^{n-1})^*$ and so $\gamma^{\sharp}=v$ and $g(v,v)=g(\gamma^{\sharp},\gamma^{\sharp})=g(\gamma,\gamma)=\epsilon$. Thus, $\alpha\otimes v-\gamma\otimes \nu(\alpha)=\alpha\otimes \gamma^{\sharp}-\gamma\otimes \alpha^{\sharp}\in \mathfrak{so}(g)$ and so $\mfh_g\subseteq \mathfrak{so}(g)\cap \mfh$.
 \end{proof}
\begin{remark}\label{re:gnotunique}
Note that the choice of the pseudo-metric $g$ in Proposition \ref{pro:nusharp} is far from being unique. We also note that while the linear subalgebra $\mfh_g$ of $\mfh$ in Proposition \ref{pro:nusharp} is metric, this does not need to be the case for the entire $\mfh$. For the sake of an explicit counter-example, let $(I_0,J_0,K_0)$ be the standard hypercomplex structure on $\bR^{4k}$. Moreover, take any $n\in \bN$ with $n>4k+1$, let $\cU:=\spa{e^{4k+1},\ldots,e^{n-1}}$, $\nu:\cU\rightarrow \bR^{n-1}$ be defined by $\nu(e^j)=e_j$ for $j=4k+1,\ldots,n-1$ and set
\begin{equation*}
\mfh:=\mathfrak{gl}(I_0,J_0,K_0)\times \left\{\left. \alpha\otimes \nu(\beta)-\beta\otimes \nu(\alpha),\ \alpha\otimes e_n-e^n\otimes \nu(\alpha)\right|\alpha,\beta\in \cU\right\}
\end{equation*}
i.e. the first factor $\mathfrak{gl}(I_0,J_0,K_0)$ acts on $\spa{e_1,\ldots,e_{4k}}$ while the second one acts on $\spa{e_{4k+1},\ldots,e_n}$. Then $\mfh$ is not a metric subalgebra for any pseudo-metric $g$ on $\bR^n$. However, since $\mathfrak{gl}(I_0,J_0,K_0)^{(1)}=\{0\}$, we do, in fact, have
\begin{equation*}
\cK_{\mfh}=S^2 \cU\otimes e_n.
\end{equation*}
\end{remark}
We are finally in the position to determine $\cF_{\mfh}$ for a linear subalgebra $\mfh$ with $\cK_{\mfh}^{(1)}=S^2 \cU\otimes v$
\begin{theorem}\label{th:cK1=S2cUv}
Let $\mfh$ be a linear subalgebra with $\cK_{\mfh}^{(1)}=S^2 \cU\otimes v$ for some non-zero subspace $\cU$ of $(\bR^{n-1})^*$ and some $v\in \bR^{n}\setminus \bR^{n-1}$. Then:
\begin{enumerate}[(a)]
\item
If $\mfh$ is not elliptic, then $\cF_{\mfh}=\tilde{\mfk}_{\mfh}$.
\item
If $\mfh$ is elliptic, then
\begin{equation*}
\cF_{\mfh}=\tilde{\mfk}_{\mfh}+ \spa{\alpha\otimes \nu(\alpha)|\alpha\in \cU}.
\end{equation*}
In this case, $\cU\otimes U\subseteq \cF_{\mfh}$.
\end{enumerate}
\end{theorem}
\begin{proof}
For both cases, let $\nabla\in \cT^{-1}(\End(\bR^{n-1}))\subseteq \cD_{\mfh}$ be given and observe that
\begin{equation*}
	\nabla|_{\bR^{n-1}\times \bR^{n-1}}\in S^2 \cU\otimes v.
\end{equation*}
Now we distinguish the two cases:
\begin{enumerate}[(a)]
\item
If $\mfh$ is not elliptic, then we are either in case (i) or in case (ii) in Lemma \ref{le:nu}. 

In case (i), Lemma \ref{le:nu} yields $\nabla_u\in \spa{\alpha\otimes v|\alpha\in \cU}\oplus \mfh_1\subseteq \spa{\alpha\otimes v|\alpha\in \cU}\oplus \spa{\gamma\otimes v}$ for all $u\in \bR^{n-1}$. Consequently, if $\mfh_1=\{0\}$, we have
$\nabla v|_{\bR^{n-1}}=0$ and the condition $\nabla\in \cT^{-1}(\End(\bR^{n-1}))$ forces $\nabla_v\in \mfk_{\mfh}$ and so $\cT(\nabla)=\nabla_v|_{\bR^{n-1}}\in \tilde{\mfk}_{\mfh}$. Hence, $\cF_{\mfh}\subseteq \tilde{\mfk}_{\mfh}$ and so $\cF_{\mfh}=\tilde{\mfk}_{\mfh}$. 

If $\mfh_1=\spa{\gamma\otimes v}$, then
\begin{equation*}
\cT(\nabla)=\nabla_v|_{\bR^{n-1}}-\nabla v|_{\bR^{n-1}}=\nabla_v|_{\bR^{n-1}}-\beta\otimes v
\end{equation*}
for some $\beta\in (\bR^{n-1})^*$ and so
\begin{equation*}
\cT(\nabla)=\pi_{\bR^{n-1}}\circ\nabla_v|_{\bR^{n-1}}\subseteq  \left\{\left. \pi_{\bR^{n-1}}\circ F|_{\bR^{n-1}}\right|F\in \mfh\right\}=\tilde{\mfk}_{\mfh}
\end{equation*}
by Lemma \ref{le:tildemfkequaltildecK} and with the notations from this lemma, yielding again $\cF_{\mfh}=\tilde{\mfk}_{\mfh}$.

So let us now assume that we are in case (ii). Then $\nabla|_{\bR^{n-1}\times \bR^{n-1}}=\alpha\otimes \alpha\otimes v$ for some $\alpha$ in the one-dimensional space $\cU$, $\mfh_1=\{0\}$ and there exists $\lambda\in \bR^*$ such that
\begin{equation*}
\nabla|_{\bR^{n-1}\times \bR^n}=\alpha\otimes (\alpha+\lambda\gamma)\otimes v.
\end{equation*}
Then $\nabla v|_{\bR^{n-1}}=\lambda \gamma(v)\, \alpha\otimes v$. Setting $U_0:=\ker(\alpha)$, the condition $\nabla\in \cT^{-1}(\End(\bR^{n-1}))$ forces $F(U_0)\subseteq \bR^{n-1}$ for $F:=\nabla_v$. Let $u_0\in \bR^{n-1}\backslash U_0$. Then $\alpha(u_0)\neq 0$ and so $G:=F-\frac{F(u_0)}{\alpha(u_0)} (\alpha+\lambda \gamma)\otimes v\in \mfh$ preserves $\bR^{n-1}$ and so $G|_{\bR^{n-1}}\in \tilde{\mfk}_{\mfh}$. Consequently,
\begin{equation*}
\cT(\nabla)=F|_{\bR^{n-1}}-\lambda \gamma(v)\, \alpha\otimes v=G|_{\bR^{n-1}}+\frac{F(u_0)}{\alpha(u_0)}\alpha\otimes v-\lambda\gamma(v)\, \alpha\otimes v
\end{equation*}
and so the condition $\nabla\in \cT^{-1}(\End(\bR^{n-1}))$ forces $\frac{F(u_0)}{\alpha(u_0)}=\lambda\gamma(v)$ and so gives $\cT(\nabla)=G|_{\bR^{n-1}}\in \tilde{\mfk}_{\mfh}$. Thus, $\cF_{\mfh}\subseteq \tilde{\mfk}_{\mfh}$ and so $\cF_{\mfh}= \tilde{\mfk}_{\mfh}$.
\item
By Sylvester's law of inertia, we have a basis $\alpha_1,\ldots,\alpha_m$ of $\cU$ and $\lambda_1,\ldots,\lambda_m\in \{-1,0,1\}$ with
\begin{equation*}
	\nabla|_{\bR^{n-1}\times \bR^{n-1}}=\sum_{i=1}^m \lambda_i \alpha_i\otimes \alpha_i\otimes v.
\end{equation*}
As $\nabla_u\in \cU\otimes v$ and so $\nabla_u\in \mfh_v$ for all $u\in \bR^{n-1}$, Lemma \ref{le:nu} and Lemma \ref{le:bilinearnusymmetric} imply
\begin{equation*}
	\nabla|_{\bR^{n-1}\times \bR^n}=\sum_{i=1}^m \lambda_i \alpha_i\otimes \left(\alpha_i\otimes v-\gamma\otimes \nu(\alpha_i)\right).
\end{equation*}
Consequently,
\begin{equation*}
\nabla v|{\bR^{n-1}}=-\gamma(v)\sum_{i=1}^m \lambda_i\,\alpha_i\otimes \nu(\alpha_i)\in \spa{\alpha\otimes \nu(\alpha)|\alpha\in \cU}.
\end{equation*}
In particular, $\nabla v$ preserves $\bR^{n-1}$ and so also $\nabla_v$ has to preserve $\bR^{n-1}$. Thus, $\nabla_v|_{\bR^{n-1}}\in \tilde{\mfk}_{\mfh}$ and so
\begin{equation*}
\cT(\nabla)\in \tilde{\mfk}_{\mfh}+ \spa{\alpha\otimes \nu(\alpha)|\alpha\in \cU},
\end{equation*}
which proves the inclusion $\cF_{\mfh}\subseteq \tilde{\mfk}_{\mfh}+ \spa{\alpha\otimes \nu(\alpha)|\alpha\in \cU}$.

For the converse inclusion, let $f=f_1+f_2\in \tilde{\mfk}_{\mfh}+ \spa{\alpha\otimes \nu(\alpha)|\alpha\in \cU}$. Choose $F_1\in \mfh$ with $F_1(\bR^{n-1})\subseteq \bR^{n-1}$ and $F_1|_{\bR^{n-1}}=f_1$ and $\alpha_1,\ldots,\alpha_m\in \cU$ and $\lambda_1,\ldots,\lambda_m\in \bR$ such that
\begin{equation*}
f_2=\sum_{i=1}^m \lambda_i \alpha_i\otimes \nu(\alpha_i).
\end{equation*}
Define now $\nabla\in  \bR^n\otimes \mfh$ by
\begin{equation*}
\nabla|_{\bR^{n-1}\otimes \bR^n}:=-\frac{1}{\gamma(v)}\sum_{i=1}^m \lambda_i \alpha_i\otimes \left(\alpha_i\otimes v-\gamma\otimes \nu(\alpha_i)\right)
\end{equation*}
and $\nabla_v:=F_1$. Then one checks that $f=\cT(\nabla)\in \End(\bR^{n-1})$, i.e.
$f\in \cF_{\mfh}$, proving the other inclusion and so the equality $\cF_{\mfh}=\tilde{\mfk}_{\mfh}+ \spa{\alpha\otimes \nu(\alpha)|\alpha\in \cU}$.

Finally, we need to show that $\cU\otimes U\subseteq \cF_{\mfh}$. For this, let $F\in \cU\otimes U$. Then there exist $\beta_1,\ldots,\beta_l\in \cU$ and $u_1,\ldots,u_l\in U$ such that
\begin{equation*}
F=\sum_{i=1}^m \beta_i\otimes u_i.
\end{equation*}
Now setting $\tau_i:=\nu^{-1}(u_i)$, we may write
\begin{equation*}
\begin{split}
\beta_i\otimes u_i&=\frac{1}{2}\, \left(\beta_i\otimes u_i+\tau_i\otimes \nu(\beta_i)\right)+\frac{1}{2}\, \left(\beta_i\otimes u_i-\tau_i\otimes \nu(\beta_i)\right)\\
&=\frac{1}{2}\, \left(\beta_i\otimes \nu(\tau_i)+\tau_i\otimes \nu(\beta_i)\right)+\frac{1}{2}\, \left(\beta_i\otimes \nu(\tau_i)-\tau_i\otimes \nu(\beta_i)\right).
\end{split}
\end{equation*}
By Lemma \ref{le:bilinearnusymmetric}, we have $\beta_i\otimes \nu(\tau_i)-\tau_i\otimes \nu(\beta_i)\in \tilde{\mfk}_{\mfh}$. Moreover,
\begin{equation*}
\beta_i\otimes \nu(\tau_i)+\tau_i\otimes \nu(\beta_i)=(\beta_i+\tau_i)\otimes (\nu(\beta_i+\tau_i))-\beta_i\otimes \nu(\beta_i)-\tau_i\otimes \nu(\tau_i)\in  \spa{\alpha\otimes \nu(\alpha)|\alpha\in \cU},
\end{equation*}
implying $\beta_i\otimes u_i\in \cF_{\mfh}$ for all $i=1,\ldots,l$. Consequently, $F\in \cF_{\mfh}$ and so $\cU\otimes U\subseteq \cF_{\mfh}$.
\end{enumerate}
\end{proof}
For non-degenerate metric subalgebras, Theorem \ref{th:cK1=S2cUv} yields:
\begin{corollary}\label{co:nondegmetric}
Let $\mfh$ be a non-degenerate metric subalgebra. Then
\begin{equation*}
\cF_{\mfh}=\tilde{\mfk}_{\mfh}\oplus \spa{\alpha\otimes \alpha^{\sharp}|\alpha \in \cU}=\tilde{\mfk}_{\mfh}\oplus \spa{u^b\otimes u |u\in U}
\end{equation*}
for
\begin{equation*}
\cU:=\left\{F(v)^b|F\in \mfh_v\right\}\subseteq (\bR^{n-1})^*,\quad U:=\left\{F(v)|F\in \mfh_v\right\}\subseteq \bR^{n-1}.
 \end{equation*}
\end{corollary}
\begin{proof}
First of all, note that by Lemma \ref{le:nondegmetriccK1}, we have $\cK_{\mfh}^{(1)}=\cU\otimes \spa{v}$ for some $v\perp \bR^{n-1}$ and that by Lemma \ref{le:charofcU}
\begin{equation*}
\cU=\left\{\alpha\in (\bR^{n-1})^*|F|_{\bR^{n-1}}=\alpha\otimes v \textrm{ for some } F\in \mfh_v\right\}.
\end{equation*}
As $F\in \mfh_v$ is of the form $F=\alpha\otimes v-v^b\otimes \alpha^{\sharp}$, we have
\begin{equation*}
\cU:=\left\{F(v)^b|F\in \mfh_v\right\}\subseteq (\bR^{n-1})^*
\end{equation*}
as claimed. Moreover, if $\cU\neq \{0\}$, we choose $\gamma:=v^b$ and then have $\nu(\alpha)=\alpha^{\sharp}$, Thus, the result follows directly from Theorem \ref{th:cK1=S2cUv}. Finally, if $\cU=\{0\}$, then $\cK_{\mfh}^{(1)}=\{0\}$ and the result follows from Corollary \ref{co:cK1=0elliptic} as $\mfh$ is elliptic.
\end{proof}
\begin{example}
\begin{itemize}
	\item Corollary \ref{co:nondegmetric} gives another proof of Theorem \ref{th:unitarysubalgebras} (a) and so of the characterisation of the almost Abelian Lie algebras which are K\"ahler in \cite{LW}.
	\item Corollary \ref{co:nondegmetric} also reproves the characterisation that the author has obtained in \cite{Fr2} of the almost Abelian Lie algebras $\mfg$ admitting a parallel $\G_2^*$-structure with non-degenerate codimension one Abelian ideal $\mfu$. In this case, one checks that always $\mfh_v=0$, hence $\cU=\{0\}$, and $\cF_{\mfh}=\tilde{\mfk}_{\mfh}=\mathfrak{su}(1,2)$ if the signature of $\mfu$ is $(2,4)$ and $\cF_{\mfh}=\tilde{\mfk}_{\mfh}=\mathfrak{sl}(3,\bR)$ if the signature of $\mfu$ is $(3,3)$.
\end{itemize}
\end{example}
\begin{remark}
We note that the flat pseudo-Riemannian almost Abelian Lie algebras with non-degenerate codimension Abelian ideal are known, cf., e.g., \cite[Proposition 4.21]{CFrG}. In this case, $f=f^A+\lambda u^b\otimes u$ with $f^A$ being skew-symmetric, $\lambda\in \bR$ and $u\in\mfu$ with $g(u,u)=0$ and $f^A(u)=0$. Note that so $\mathrm{O}(p,n-p)$-structures provide examples where for the subset $F_{\mfh}$ of endomorphisms of $\bR^{n-1}$ giving rise to special flat $H$-structures on the associate almost Abelian Lie algebra one has
\begin{equation*}
\mathfrak{so}(p,n-1-p)=\tilde{\mfk}_{\mfh}\subsetneq F_{\mfh}\subsetneq \cF_{\mfh}=\End(\bR^{n-1}).
\end{equation*}
Moreover, the result on flat pseudo-Riemannian almost Abelian Lie algebras with non-degenerate $\mfu$ implies that an arbitrary metric non-degenerate special $H$-structure is flat if and only if
\begin{equation*}
f=f_A+\lambda u^b\otimes u
\end{equation*}
for $f_A\in \tilde{\mfk}_{\mfh}$, $\lambda\in \bR$, $u\in \mfu$ with $g(u,u)=0$ and $f^A(u)=0$.
\end{remark}
\subsection{$\cK^{(1)}=S^2 \cU\otimes w$ for $w\in \bR^{n-1}$}\label{subsec:cK1=S2Uw}
In the last subsection, we have seen that the condition $\cK^{(1)}=S^2 \cU\otimes v$ for some $v\notin \bR^{n-1}$ naturally led to distinguishing the case that $\mfh$ is elliptic or not, with the elliptic case resembling some features of non-degenerate metric subalgebras. In this section, we restrict from the beginning to the elliptic case since firstly, the non-elliptic case seems to be much more difficult than the elliptic one and, secondly, we are mostly interested in \emph{degenerate} metric linear subalgebras, which fit into our set-up by Lemma \ref{le:degmetriccK1} below and are always elliptic. In fact, Proposition \ref{pro:cK1S2Uwdegmetricfirst} and Proposition \ref{pro:cK1S2Uwdegmetricsecond}  further below will show that assuming $\cK^{(1)}=S^2 \cU\otimes w$ for some $w\in \bR^{n-1}$ and some subspace $\cU\subseteq (\bR^{n-1})^*$ yields assuming one extra condition the existence of a subalgebra of $\mfh$ which is degenerate metric with respect to some pseudo-Riemannian metric $g$ on $\bR^n$. We note that this result is analogous to the result in Proposition \ref{pro:nusharp} in the previous subsection and note that examples of non-metric subalgebras $\mfh$ with $\cK^{(1)}=S^2 \cU\otimes w$ satisfying even the extra conditions may easily be obtained analogous to Remark \ref{re:gnotunique} and that the $g$ in Proposition \ref{pro:cK1S2Uwdegmetricfirst} and Proposition \ref{pro:cK1S2Uwdegmetricsecond}, is, in general, far from being unique.

Let us start with some basic definitions related to and some properties of the linear subalgebras $\mfh$ under consideration in this subsection. For this, we will from now on assume that $\cK_{\mfh}^{(1)}=S^2 \cU\otimes w$ for some subspace $\cU$ of $(\bR^{n-1})^*$ and some $w\in \bR^{n-1}\setminus \{0\}$. Moreover, we fix some $v\in \bR^n\setminus \bR^{n-1}$ so that we may split $\bR^n=\bR^{n-1}\oplus \spa{v}$ and so also identify $(\bR^{n-1})^*$ naturally with a subspace of $(\bR^n)^*$, namely with the annihilator of $v$. We note that we will surpress the dependence of some of the constructions below from the chosen $v$. First of all, we set
\begin{equation*}
\begin{split}
\mfh_w&:=\left\{F\in \mfh\left|F(\bR^{n-1})\subseteq \bR^{n-1}\right.\right\},\\
\mfh_w^0&:=\left\{F\in \mfh_w\left|\mathrm{im}(F)\subseteq \bR^{n-1}\right.\right\}.
\end{split}
\end{equation*}
and note the following alternative description of $\cU$ similar to Lemmma \ref{le:charofcU}:
\begin{lemma}\label{le:charofcU2}
Let $\mfh$ be a linear subalgebra with $\cK_{\mfh}^{(1)}=S^2 \cU\otimes w$ for some subspace $\cU$ of $(\bR^{n-1})^*$ and some $w\in \bR^{n-1}\setminus \{0\}$. Then
\begin{equation*}
\cU=\left\{\left.\alpha\in (\bR^{n-1})^*\right|F|_{\bR^{n-1}}=\alpha\otimes w \textrm{ for some } F\in \mfh_w\right\}.
\end{equation*}
\end{lemma}
Next, we note that the the proof of Lemma \ref{le:nondegmetriccK1} yields that degenerate metric subalgebras fit into our set-up
\begin{lemma}\label{le:degmetriccK1}
Let $\mfh$ be a \emph{degenerate} metric subalgebra, i.e $\mfh$ is a subalgebra of $\mathfrak{so}(g)$ for some pseudo-Riemannian metric on $\bR^n$ such that $\bR^{n-1}$ is a degenerate subspace. Then there exists some subspace $\cU$ of $(\bR^{n-1})^*$ sich that
\begin{equation*}
\cK_{\mfh}^{(1)}=S^2 \cU\otimes (\bR^{n-1})^{\perp}=S^2 \cU\otimes w
\end{equation*}
for all $w\in (\bR^{n-1})^{\perp}\setminus \{0\}$.
\end{lemma}
We set now
\begin{equation*}
	\cU_0=\left\{\left.\alpha\in (\bR^{n-1})^*\right|F|_{\bR^{n-1}}=\alpha\otimes w \textrm{ for some } F\in \mfh_w^0\right\}
\end{equation*}
and prove, in analogy to Lemma \ref{le:nu}, the following result in our case:
\begin{lemma}\label{le:nu2}
Let $\mfh$ be an elliptic linear subalgebra with $\cK^{(1)}_{\mfh}=S^2 \cU\otimes w$ for some subspace $\cU$ of $(\bR^{n-1})^*$ and some $w\in \bR^{n-1}\setminus \{0\}$. Moreover, let $\gamma\in (\bR^n)^*$ be a one-form uniquely defined by $\gamma(\bR^{n-1})=\{0\}$ and $\gamma(v)=1$. Then there exists an injective linear map $\nu:\cU_0\rightarrow \bR^{n-1}$ uniquely defined by the property that
\begin{equation*}
\alpha\otimes w-\gamma\otimes \nu(\alpha)\in \mfh_w^0
\end{equation*}
for $\alpha\in \cU$. Moreover, $\alpha(w)=0$ for all $\alpha\in \cU_0$ and the bilinear form
\begin{equation*}
h(\alpha,\beta):=\alpha(\nu(\beta))
\end{equation*}
is symmetric.
\end{lemma}
\begin{proof}
The well-definedness and injectivity of the map $\nu$ follows as in the proof of Lemma \ref{le:nu2} directly from the ellipticity of $\mfh$. Next, let linearly independent $\alpha,\beta\in \cU_0$ be given and let
\begin{equation*}
F_1:=\alpha\otimes w-\gamma\otimes \nu(\alpha),\quad F_2:=\beta\otimes w-\gamma\otimes \nu(\beta)\in \mfh_w^0
\end{equation*}
be the associated elements of $\mfh_w^0$. Then
\begin{equation*}
\mfh\ni [F_1,F_2]=(\alpha(w)\beta-\beta(w)\alpha)\otimes w-(\alpha(\nu(\beta))-\beta(\nu(\alpha)))\gamma\otimes w
\end{equation*}
and the ellipticity of $\mfh$ forces $[F_1,F_2]=0$, i.e. $\alpha(w)=\beta(w)=0$ and $\alpha(\nu(\beta))=\beta(\nu(\alpha))$. This proves all claimed statements.
\end{proof}
In the following, we will always assume that $\mfh$ is elliptic and denote by $\gamma\in (\bR^n)^*$ the one-form and by $\nu:\cU_0\rightarrow \bR^{n-1}$ the injective linear map as in Lemma \ref{le:nu2}. We note that the natural linear map from $\mfh_w$ to $\cU$ is then, due to the ellipticity of $\mfh$, an isomorphim which also maps $\mfh_w^0$ isomorphically onto $\cU_0$.

We first consider the case $\mfh_w^0=\mfh_w$ or, equivalently, $\cU=\cU_0$. In this case, we obtain the existence of a pseudo-metric on $\bR^n$ such that $\nu$ is the sharp-operator and $\mfh_w$ is degenerate metric.
\begin{proposition}\label{pro:cK1S2Uwdegmetricfirst} 
Let $\mfh$ be an elliptic subalgebra with $\cK^{(1)}_{\mfh}=S^2 \cU\otimes w$ for some subspace $\cU$ of $(\bR^{n-1})^*$ and some $w\in \bR^{n-1}\setminus \{0\}$ such that $\mfh_w^0=\mfh_w$. Then there exists a pseudo-metric $g$ on $\bR^n$ such that $g|_{\cU\otimes \cU}=h$, such that $\nu=(\cdot)^{\sharp_g}|_{\cU}$ and such that $w$ and $v$ are null vectors with $g(v,w)=1$. In particular, then $\mfh_w$ is an Abelian degenerate metric subalgebra with respect to $g$.
\end{proposition}
\begin{proof}
Argueing as in the proof of Proposition \ref{pro:nusharp}, we get some $l\in \{1,\ldots,m\}$, $m:=\dim(\cU)$, $\epsilon_1,\ldots,\epsilon_l\in \{-1,1\}$ and linearly independent $\alpha_1,\ldots,\alpha_{2m-l}$ such that $\alpha_1,\ldots,\alpha_m$ is a basis of $\cU$, such that
\begin{equation*}
\alpha_i(\nu(\alpha_j))=\delta_{ij} \epsilon_j,\ \alpha_r(\nu(\alpha_k))=\alpha_k(\nu(\alpha_r))=0,\ \alpha_{s}(\nu(\alpha_i))=0,\ \alpha_s(\nu(\alpha_k))=\delta_{s-m,k-l}
\end{equation*}
for all $i,j=1,\ldots, l$, $k=l+1,\ldots,m$, $r=1,\ldots,m$, $s=m+1,\ldots, 2m-l$. We may arrange $\alpha_{l+1},\ldots,\alpha_{2m-l}$ in such a way that $\alpha_s(w)=0$ for all $s=m+1,\ldots,2m-l$ and may then extend $\alpha_1,\ldots,\alpha_{2m-l}$ to a basis $\alpha_1,\ldots,\alpha_{n-1}$of $(\bR^{n-1})^*$ such that $\alpha_t(w)=0$ for all $t=1,\ldots,n-2$ and such that $\alpha_{n-1}(w)=1$. We then finally set $\alpha_n:=\gamma$ and so have a basis $\alpha_1,\ldots,\alpha_n$ of $(\bR^{n-1})^*$. Finally, we define the pseudo-metric $g$ with respect to that basis by the matrix
 \begin{equation*}
	\begin{pmatrix}
		\epsilon_1 & & & & & & & \\
		& \ddots & & & & & & & \\
		& & \epsilon_l & & & & & \\
		& & & 0 & I_{m-l} &  & & \\
		& & & I_{m-l} & 0 & & & \\
		& & & & & I_{n-2-2m+l} & & \\
		& & & & & & 0 & 1\\
		& & & & & & 1 & 0
	\end{pmatrix}.
\end{equation*}
We then immediately see that $g|_{\cU\times \cU}=h$, $\alpha(\nu(\beta))=g(\alpha,\beta)$ for $\alpha\in (\bR^n)^*$ and $\beta\in \cU$ and $g(v,v)=g(w,w)=0$, $g(v,w)=1$. But then $\nu(\alpha)=\alpha^{\sharp_g}$ follows as in the proof of Proposition \ref{pro:nusharp} and similarly, we see that
$\gamma^{\sharp_g}=w$. Thus, the elements of $\mfh_w$ are of the form
\begin{equation*}
\alpha\otimes w-\gamma\otimes \nu(\alpha)=\alpha\otimes \gamma^{\sharp_g}-\gamma\otimes \alpha^{\sharp_g}
\end{equation*}
and so the elements are in $\mathfrak{so}(g)$, i.e. $\mfh_w$ is metric with respect to $g$. As $\bR^{n-1}$ is a degenerate subspace with respect to $g$, $\mfh_w$ is degenerate by definition and a direct computation, using $\alpha(w)=0$ for all $\alpha\in \cU_0=\cU$, yields that $\mfh_w$ is Abelian.
\end{proof}
We are now in the position to determine $\cF_{\mfh}$ in the case that $\mfh_w^0=\mfh_w$.
\begin{theorem}\label{th:cK1=S2cUwfirst}
Let $\mfh$ be an elliptic subalgebra with $\cK^{(1)}_{\mfh}=S^2 \cU\otimes w$ for some subspace $\cU$ of $(\bR^{n-1})^*$ and some $w\in \bR^{n-1}\setminus \{0\}$ such that $\mfh_w^0=\mfh_w$. Then
\begin{equation*}
\cF_{\mfh}=\tilde{\mfk}_{\mfh}+\spa{\alpha\otimes \nu(\alpha)|\alpha\in \cU}
\end{equation*}
\end{theorem}
\begin{proof}
Let $\nabla\in \cT^{-1}(\End(\bR^{n-1}))$ be given. Similar to the proof of Theorem \ref{th:cK1=S2cUv} we may use Sylvester's law of inertia to write
\begin{equation*}
\nabla|_{\bR^{n-1}\times \bR^n}=\sum_{i=1}^m \lambda_i \alpha_i\otimes (\alpha_i\otimes w-\gamma\otimes \nu(\alpha_i))
\end{equation*}
for a basis $\alpha_1,\ldots,\alpha_m$ of $\cU$ and certain $\lambda_i\in \{-1,1,0\}$. Thus,
\begin{equation*}
-\nabla v|_{\bR^{n-1}}=\sum_{i=1}^m \lambda_i \alpha_i\otimes \nu(\alpha_i)\in \spa{\alpha\otimes \nu(\alpha)|\alpha\in \cU}
\end{equation*}
In particular, $\nabla v$ preserves the subspace $\bR^{n-1}$ of $\bR^n$ and $\nabla\in \cT^{-1}(\End(\bR^{n-1}))$ implies that so $\nabla_v$ has to preserve $\bR^{n-1}$. Thus, $\nabla_v|_{\bR^{n-1}}\in \tilde{\mfk}_{\mfh}$ and so
\begin{equation*}
\cT(\nabla)=\nabla_v|_{\bR^{n-1}}-\nabla v|_{\bR^{n-1}}\in \tilde{\mfk}_{\mfh}+\spa{\alpha\otimes \nu(\alpha)|\alpha\in \cU}
\end{equation*}
as claimed. This shows $\cF_{\mfh}=\tilde{\mfk}_{\mfh}+\spa{\alpha\otimes \nu(\alpha)|\alpha\in \cU}$ and the other inclusion follows by directly constructing for a given element $F\in \tilde{\mfk}_{\mfh}+\spa{\alpha\otimes \nu(\alpha)|\alpha\in \cU}$ some $\nabla\in \cD_{\mfh}$ with $\cT(\nabla)=F$ by basically inverting the steps from above.
\end{proof}
We note that when $\mfh$ is a degenerate metric subalgebra for some pseudo-metric $g$, then $\cK_{\mfh}^{(1)}=S^2\cU\otimes w$ for $0\neq	w\in (\bR^{n-1})^{\perp_g}$. Choosing then $v\in \bR^n\setminus \bR^{n-1}$ such that $v$ is null and $g(v,w)=1$, we have $\gamma=w^{b_g}$ and so $\nu=(\cdot)^{\sharp_g}$. Thus, Theorem \ref{th:cK1=S2cUwfirst} yields the following result for $\cF_{\mfh}$ under the additional assumption that $\mfh_w=\mfh_w^0$:
\begin{corollary}\label{co:degmetricfirst}
Let $\mfh$ be a degenerate metric subalgebra for a pseudo-metric $g$ such that $\mfh_w=\mfh_w^0$, i.e. such that any $F\in \mfh$ with $F(\bR^{n-1})\subseteq (\bR^{n-1})^{\perp_g}$ satisfies $\im(F)\subseteq \bR^{n-1}$. Then
\begin{equation*}
	\cF_{\mfh}=\tilde{\mfk}_{\mfh}\oplus \spa{\alpha\otimes \alpha^{\sharp}|\alpha\in \cU}=\tilde{\mfk}_{\mfh}\oplus \spa{u^{b}\otimes u|u\in U}
\end{equation*}
for $U:=\{F(v)|F\in \mfh_w\}$, $\cU:=U^{b}$, with $v\in \bR^n$ chosen as above.
\end{corollary}
\begin{example}
	\begin{itemize}
		\item We show that Theorem \ref{co:degmetricfirst} gives another proof of Theorem \ref{th:unitarysubalgebras} (b). For this, let $\mfh$ be a degenerate unitary subalgebra. We first show that $\mfh$ satisfies $\mfh_w=\mfh_w^0$. For this, choose $w\in (\bR^{n-1})^{\perp_g}\setminus \{0\}$ and let $F\in \mfh_w$ be given. Then $F$ satisfies $F(\bR^{n-1})\subseteq (\bR^{n-1})^{\perp_g}$, $F|_{\bR^{n-1}}=\alpha\otimes w$ for some $\alpha\in (\bR^{n-1})^*$ and $Jw\in \bR^{n-1}=\spa{w}^{\perp_g}$ due to $g(Jw,w)=0$. Thus, $w\in (\bR^{2m-1})_J:=\bR^{2m-1}\cap J\bR^{2m-1}$. However, if $u\in (\bR^{2m-1})_J$, one has
		\begin{equation*}
		\alpha(Ju)\, w=F(Ju)=JF(u)=\alpha(u)\, Jw,
		\end{equation*}
        and so $\alpha(u)=\alpha(Ju)=0$. Thus, $\alpha(w)=0$, i.e. $F(w)=0$. However, if $v\in \bR^n$ is arbitrary, this implies $g(F(v),w)=-g(v,F(w))=0$, i.e. $F(v)\in \spa{w}^{\perp_g}=\bR^{n-1}$. Thus, $\im(F)\subseteq \bR^{n-1}$ and so $\mfh_w=\mfh_w^0$.
         
        Hence, Corollary \ref{co:degmetricfirst} applies to our situation. Now choose $v\notin \bR^{n-1}$ such that $v$ null, $g(v,w)=1$ and $Jv\in \bR^{n-1}$. Then 
        \begin{equation*}
        F(v)=-F(J^2 v)=-JF(Jv)=-\alpha(Jv)\, Jw\in \spa{Jw}.
        \end{equation*}
        Thus, either $U=\{0\}$ or $U=\spa{Jw}$ and Corollary \ref{co:degmetricfirst} implies the result of Theorem \ref{th:unitarysubalgebras} (b).
		\item 
		Let  $\mfh$ be a linear subalgebra conjugated to $\mathfrak{g}_2^*$ for which $\bR^6$ is degenerate. As $\G_2^*$ acts transitively on the degenerate codimension one subspaces in $\bR^7$ and we may assume that $\mfh$ is the stabiliser subalgebra of the three-form
		\begin{equation*}
		\varphi=-e^{156}-e^{236}+e^{245}-\tfrac{1}{2}e^{127}-e^{347}.	
		\end{equation*}
		cf. \cite{Fr2}. The induced pseudo-metric $g=g_{\varphi}$ is then explicitly given by
		\begin{equation*}
		g=-e^2\otimes e^2+e^1\odot e^7+2 e^3\odot e^6-2 e^4\odot e^5	
		\end{equation*}
	so that $(\bR^6)^{\perp_g}=\spa{e_1}$, i.e. we may choose $w:=e_1$. One may then compute that
		\begin{equation*}
			\tilde{\mfk}_{\mfh}=\left\{\left.\begin{pmatrix} -\tr(A) & -2 b & v^t & w^t \\ 0 & 0 & 0 & v^t \\ 0 & -J_0 v & A-\tr(A) I_2  & b I_2 \\ 0 & 0 & 0 & A
			\end{pmatrix}\right|A\in \mathfrak{gl}(2,\bR),\, v,w\in \bR^{2},\, b\in \bR\right\}
		\end{equation*}
		for $J_0:=\left(\begin{smallmatrix} 0 & -1 \\ 1 & 0 \end{smallmatrix} \right)$ and that
		\begin{equation*}
			\mfh_w=\spa{2 e^5\otimes e_1+e^7\otimes e_4, 2 e^6\otimes e_1- e^7\otimes e_3}=\mfh_w^0
		\end{equation*}
		holds. Thus,
		\begin{equation*}
			U=\spa{e_3,e_4}.
		\end{equation*}
		Moreover, $e_3^b=2 e^6$ and $e_4^b=-2 e^5$ and so Corollary \ref{co:degmetricfirst} yields
		\begin{equation*}
			\begin{split}
				\cF_{\mfh} &=\tilde{\mfk}_{\mfh}\oplus \spa{ u^b|_{\bR^6}\otimes u\left| u\in \tilde{U}\right. } \\
				&=\left\{\left.\begin{pmatrix} -\tr(A) & -\tr(B) & v^t & w^t \\ 0 & 0 & 0 & v^t \\ 0 & -J_0 v & A-\tr(A) I_2  & B \\ 0 & 0 & 0 & A
				\end{pmatrix}\right|A,B\in \mathfrak{gl}(2,\bR),\, v,w\in \bR^{2}\right\},
			\end{split}	                                                                                                                                                                
		\end{equation*}
		which is the result of \cite[Theorem 3.10]{Fr2}.
	\end{itemize}
\end{example}
We now turn to the case $\mfh_w^0\neq \mfh_w$. We first additionally assume that $\mfh_w^0\neq \{0\}$ or, equivalently, that $\cU_0\neq \{0\}$, and show that then we may choose $v\in \bR^n\setminus \bR^{n-1}$ in a particular nice way.
\begin{lemma}\label{le:appropriatev}
Let $\mfh$ be an elliptic linear subalgebra with $\cK^{(1)}_{\mfh}=S^2 \cU\otimes w$ for some subspace $\cU$ of $(\bR^{n-1})^*$ and some $w\in \bR^{n-1}\setminus \{0\}$ such that $\mfh_w^0\neq \mfh_w$ and such that $\mfh_w^0\neq \{0\}$. Then there exists some $F_0\in \mfh_w$ such that $\mfh_w=\mfh_w^0\oplus \spa{F_0}$ and some $v\in \bR^n\setminus \bR^{n-1}$ such that
\begin{equation*}
F_0=\alpha_0\otimes w-\gamma\otimes v
\end{equation*}
for some $\alpha_0\in \cU$ with $\cU=\cU_0\oplus \spa{\alpha_0}$ such that $\alpha_0(w)=1$ and such that $\alpha_0(\nu(\cU_0))=\{0\}$. Then, in particular, $F_0$ has eigenvector $w$ with eigenvalue $1$ and eigenvector $v$ with eigenvalue $-1$.
\end{lemma}
\begin{proof}
Choose first some arbitary element $F_0\in \mfh_w\setminus \mfh_w^0$ and note that then $\mfh_w=\mfh_w^0\oplus \spa{F_0}$. Moreover, since $F_0\notin \mfh_w^0$, there exists some element $\tilde{v}\in \im(F_0)$ with $\tilde{v}\notin \bR^{n-1}$. Now note that there is some $\alpha_0\in \cU$ with $\cU=\cU_0\oplus \spa{\alpha_0}$ such that
\begin{equation*}
F_0|_{\bR^{n-1}}=\alpha_0\otimes w
\end{equation*}
and that
\begin{equation*}
F_0(v)=\lambda \tilde{v}+\mu w
\end{equation*}
for certain $\lambda\in \bR^*$ and $\mu\in \bR$. By appropriately scaling $F_0$, we may assume that $\lambda=-1$, i.e. $F_0(v)=-\tilde{v}+\mu w$. Next, let $0\neq F\in \mfh_w^0$ be given. Then $F|_{\bR^{n-1}}=\alpha\otimes w$ for some $\alpha\in \cU$ and so
\begin{equation*}
[F,F_0](u)=F(\alpha_0(u)\,w)-F_0(\alpha(u)\,w)=\alpha_0(u)\alpha(w)\, w-\alpha_0(w)\alpha(u)\, w=-\alpha_0(w)\alpha(u)\, w
\end{equation*}
for all $u\in \bR^{n-1}$ due to $\alpha(w)=0$. We note that this shows that we must have $\alpha_0(w)\neq 0$ as otherwise
$[F,F_0]$ may have at most rank one, and so has to be zero, which yields, denoting by $\tilde{\nu}$ the map $\nu$ as in Lemma \ref{le:nu2} with respect to $\tilde{v}$,
\begin{equation*}
	0=[F,F_0](v)=F(-\tilde{v}+\mu w)+F_0(\tilde{\nu}(\alpha))=\tilde{\nu}(\alpha)+\mu \alpha(w) w+\alpha_0(\tilde{\nu}(\alpha))w=\tilde{\nu}(\alpha)+\alpha_0(\tilde{\nu}(\alpha))w
\end{equation*}
which is a contradiction since $\tilde{\nu}(\alpha)$ and $w$ are linearly independent due to $F$ having rank two.
Thus, $\alpha_0(w)\neq 0$ and so $[F,F_0]\in \mfh^0_w\setminus \{0\}$, which gives
\begin{equation*}
\alpha_0(w)\tilde{\nu}(\alpha)=\tilde{\nu}(\alpha_0(w)\alpha)=\tilde{\nu}(\alpha)+\alpha_0(\nu(\alpha))w,
\end{equation*}
leading to $\alpha_0(w)=1$ and $\alpha_0(\tilde{\nu}(\alpha))=0$. Now definining $v:=\tilde{v}-\tfrac{\mu}{2}w$, we observe that
\begin{equation*}
F_0(v)=-\tilde{v}+\mu w-\tfrac{\mu}{2} \alpha_0(w)\, w=-v.
\end{equation*}
Now observe that $\nu(\alpha)$ defined with respect to $v$ coinces with $\tilde{\nu}(\alpha)$ due to
\begin{equation*}
\nu(\alpha)=-F(v)=-F(\tilde{v}-\tfrac{\mu}{2}w)=-F(\tilde{v})=\tilde{\nu}(\alpha),
\end{equation*}
which proves that $\alpha_0(\nu(\cU))=\{0\}$.
\end{proof}
For the moment we stay with the case that $\mfh_w^0\neq \mfh_w$ and $\mfh_w^0\neq \{0\}$ and are then able to show the following helpful decomposition of $\mfh$:
\begin{lemma}\label{le:decompofh}
Let $\mfh$ be an elliptic linear subalgebra with $\cK^{(1)}_{\mfh}=S^2 \cU\otimes w$ for some subspace $\cU$ of $(\bR^{n-1})^*$ and some $w\in \bR^{n-1}\setminus \{0\}$ such that $\mfh_w^0\neq \mfh_w$ and such that $\mfh_w^0\neq \{0\}$. Choose $F_0\in \mfh_w$, $\alpha_0\in \cU$ and $v$ as in Lemma \ref{le:appropriatev}. Moreover, set
 \begin{equation*}
 \mfh_v^0:=\left\{\left.F\in \mfh\right|F(\ker(\alpha_0))\subseteq \spa{v},\ F(v)=0,\ F(w)\in \ker(\alpha_0)\right\},
 \end{equation*}
\begin{equation*}
\mfa_0:=\left\{\left.F\in \mfh\right|F(\ker(\alpha_0))\subseteq \ker(\alpha_0),\ F|_{\spa{v,w}}=\lambda\cdot \id_{\spa{v,w}} \textrm{ for some }\lambda\in \bR\right\}
\end{equation*}
and
\begin{equation*}
	\tilde{\mfa}_0:=\left\{\left.H\in \End(\bR^{n-1})\right|H=F|_{\bR^{n-1}} \textrm{ for some } F\in \mfa_0\right\}.
\end{equation*}
Then $\mfh$ decomposes as a vector space as
	\begin{equation*}
		\mfh=\mfa_0\oplus \mfh_w^0\oplus \spa{F_0}\oplus \mfh_v^0=\mfk_{\mfh}\oplus \mfh_v^0
	\end{equation*}
	and $\mfk_{\mfh}$ decomposes as a vector space as
	\begin{equation*}
		\mfk_{\mfh}=\mfa_0\oplus \mfh_w^0\oplus \spa{F_0}
	\end{equation*}
so that
	\begin{equation*}
	\tilde{\mfk}_{\mfh}=\tilde{\mfa}_0\oplus \cU\otimes w
\end{equation*}
\end{lemma}
\begin{proof}
Let $F\in \mfh$ be given.
We use the decomposition $\bR^n=\ker(\alpha_0)\oplus \spa{v}\oplus \spa{w}$ to write $F$ as a matrix
\begin{equation*}
F=\left(\begin{smallmatrix}
	 A & u_1 & u_2 \\
	 \alpha_2 & a & b \\
	 \alpha_1 & c & d
\end{smallmatrix}\right)	
\end{equation*}
with $A\in \End(\ker(\alpha_0))$, $\alpha_1,\alpha_2\in \ker(\alpha_0)^*$, $u_1,u_2\in \ker(\alpha_0)$, $a,b,c,d\in \bR$.
Writing $F_0=\alpha_0\otimes w-\gamma\otimes v$ in the same way, we have
\begin{equation*}
	F_0=\left(\begin{smallmatrix}
		0 & 0 & 0 \\
		0 & 1 & 0 \\
		0 & 0 & -1
	\end{smallmatrix}\right).	
\end{equation*}
We compute
\begin{equation*}
\begin{split}
G_1&:=[F,F_0]=\left(\begin{smallmatrix}
	0 & u_1 & -u_2 \\
	-\alpha_2 & 0 & -2 b \\
	\alpha_1 & 2c & 0
\end{smallmatrix}\right),\quad G_2:=[[F,F_0],F_0]=\left(\begin{smallmatrix}
0 & u_1 & u_2 \\
\alpha_2 & 0 & 4 b \\
\alpha_1 & 4 c & 0
\end{smallmatrix}\right),\\
G_3 &:=[[[F,F_0],F_0],F_0]=\left(\begin{smallmatrix}
	0 & u_1 & -u_2 \\
	-\alpha_2 & 0 & - 8 b \\
	\alpha_1 & 8 c & 0
	\end{smallmatrix}\right), \quad G_4:=[[[[F,F_0],F_0],F_0],F_0]=\left(\begin{smallmatrix}
	0 & u_1 & u_2 \\
	\alpha_2 & 0 & 16 b \\
	\alpha_1 & 16 c & 0
	\end{smallmatrix}\right),
\end{split}
\end{equation*}
and $G_i\in \mfh$ for $i=1,2,3,4$. Then there are linear combinations of $G_1,G_2,G_3,G_4$ representing the endomorphisms
\begin{equation*}
\begin{split}
\left(\begin{smallmatrix}
	0 & u_1 & 0 \\
	0 & 0 & 0 \\
	\alpha_1 & 0 & 0
\end{smallmatrix}\right),
\left(\begin{smallmatrix}
	0 & 0 & u_2 \\
  \alpha_2 & 0 & 0 \\
	0 & 0 & 0
\end{smallmatrix}\right),
\left(\begin{smallmatrix}
	0 & 0 & 0 \\
	0 & 0 & b \\
0 & 0 & 0
\end{smallmatrix}\right),\left(\begin{smallmatrix}
0 & 0 & 0 \\
0 & 0 & 0 \\
0 & c & 0
\end{smallmatrix}\right)
\end{split}
\end{equation*}
of $\bR^n$ and so all these endomorphism are in $\mfh$. However, since $\mfh$ was assumed to be elliptic, we need to have $b=c=0$. Now note that
\begin{equation*}
H_1:=\left(\begin{smallmatrix}
	0 & 0 & u_2 \\
	\alpha_2 & 0 & 0 \\
	0 & 0 & 0
\end{smallmatrix}\right) \in \mfh_w^0, \quad 	H_2:=\left(\begin{smallmatrix}
0 & u_1 & 0 \\
0 & 0 & 0 \\
\alpha_1 & 0 & 0
\end{smallmatrix}\right)\in \mfh_v^0
\end{equation*}
and with $H_3:=\tfrac{a-d}{2} F_0\in \spa{F_0}$, we get that
\begin{equation*}
H_4:=F-H_1-H_2-H_3=\left(\begin{smallmatrix}
	A & 0 & 0 \\
	0 & \tfrac{a+d}{2} & 0 \\
	0 & 0 & \tfrac{a-d}{2}
\end{smallmatrix}\right)\in \mfh_a^0
\end{equation*}
and so
\begin{equation*}
F=H_1+H_2+H_3+H_4\in  \mfh_w^0\oplus  \mfh_v^0\oplus \spa{F_0}\oplus \mfa_0,
\end{equation*}
proving the claimed decomposition of $\mfh$. We see that the condition $F\in \mfk$ is equivalent to $\alpha_1=0$, which implies $\mfh=\mfk_{\mfh}\oplus \mfh_v^0$ and so the claimed decomposition of $\mfk$, and then also the one of $\tilde{\mfk}$.
\end{proof}
Let us describe the subspace $\mfh_v^0$ in some more detail. We note that the ellipticity of $\mfh$ and the definition of the space $\mfh_v^0$ implies the existence of a subspace $\tilde{\cU}_0$ of $(\bR^{n-1})^*$ with $\alpha_0\notin \tilde{\cU}_0$ and an injective linear map $\tilde{\nu}:\tilde{\cU}_0\rightarrow \bR^{n-1}$ such that $F\in \mfh_v^0$ is of the form
\begin{equation*}
F=\tilde{\alpha}\otimes v-\alpha_0\otimes \tilde{\nu}(\tilde{\alpha})
\end{equation*}
for some $\tilde{\alpha}\in \tilde{\cU}_0$.

Similar to the bilinear form $h$, we may define a bilinear form $\tilde{h}$ by
\begin{equation*}
\tilde{h}:\tilde{\cU}_0\times \tilde{\cU}_0\rightarrow \bR,\quad \tilde{h}(\tilde{\alpha},\tilde{\beta}):=\tilde{\alpha}(\tilde{\nu}(\tilde{\beta})).
\end{equation*}
Argueing as in the proof of Lemma \ref{le:nu2}, we see that $\tilde{h}$ is symmetric. 

We now prove some relations between the spaces $\cU_0$ and $\tilde{\cU}_0$ and the maps $\nu$ and $\tilde{\nu}$:
\begin{lemma}\label{le:nutildenu}
Let $\mfh$ be an elliptic linear subalgebra with $\cK^{(1)}_{\mfh}=S^2 \cU\otimes w$ for some subspace $\cU$ of $(\bR^{n-1})^*$ and some $w\in \bR^{n-1}\setminus \{0\}$ such that $\mfh_w^0\neq \mfh_w$ and such that $\mfh_w^0\neq \{0\}$ and choose $F_0\in \mfh_w$, $\alpha_0\in \cU$ and $v$ as in Lemma \ref{le:appropriatev}. Using the notations from above, we have:
\begin{enumerate}[(a)]
	\item
	$\nu|_{\cU_0\cap \tilde{\cU}_0}=\tilde{\nu}|_{\cU_0\cap \tilde{\cU}_0}$.
	\item
	$\tilde{\cU}_0\subseteq \cU_0$ if $h\neq 0$ and $\cU_0\subseteq \tilde{\cU}_0$ if $\tilde{h}\neq 0$.
	\item 
	$\alpha(\tilde{\nu}(\tilde{\alpha}))=\tilde{\alpha}(\nu(\alpha))$ for all $\alpha\in \cU_0$, $\tilde{\alpha}\in \tilde{\cU}_0$.
\end{enumerate}
\end{lemma}
\begin{proof}
\begin{enumerate}[(a)]
	\item 
	Let $\alpha\in 	\cU_0\cap \tilde{\cU}_0$. Then, using again the above introduce matrix notation, we have
	\begin{equation*}
F_1:=\left(\begin{smallmatrix}
	0 & 0 & -\nu(\alpha) \\
	\alpha & 0 & 0 \\
	0 & 0 & 0
\end{smallmatrix}\right),\ F_2:=\left(\begin{smallmatrix}
0 & -\tilde{\nu}(\alpha) &  0 \\
0 & 0 & 0 \\
\alpha & 0 & 0
\end{smallmatrix}\right)\in \mfh,
\end{equation*}
and so
	\begin{equation*}
	[F_1,F_2]=\left(\begin{smallmatrix}
		\alpha\otimes (\tilde{\nu}(\alpha)-\nu(\alpha)) & 0 & 0 \\
		0 & 0 & 0 \\
		0 & 0 & 0
	\end{smallmatrix}\right)\in \mfh.
\end{equation*}
The ellipticity of $\mfh$ then forces $[F_1,F_2]=0$, i.e. $\nu(\alpha)=\tilde{\nu}(\alpha)$.
\item
Let arbitrary $\tilde{\alpha}_0\in \tilde{\cU}_0$ and $\alpha,\beta\in \cU_0$ be given. Then
\begin{equation*}
\mfh \ni F:=\left[\left(\begin{smallmatrix}
	0 & 0 & -\nu(\alpha) \\
	\alpha & 0 & 0 \\
	0 & 0 & 0
\end{smallmatrix}\right),\left(\begin{smallmatrix}
0 & -\tilde{\nu}(\tilde{\alpha}) &  0 \\
0 & 0 & 0 \\
\tilde{\alpha} & 0 & 0
\end{smallmatrix}\right)\right]=\left(\begin{smallmatrix}
\alpha\otimes \tilde{\nu}(\tilde{\alpha})-\tilde{\alpha}\otimes \nu(\alpha) & 0 & 0 \\
0 & -\alpha(\tilde{\nu}(\tilde{\alpha})) & 0 \\
0 & 0 & \tilde{\alpha}(\nu(\alpha))
\end{smallmatrix}\right)
\end{equation*}
and so
\begin{equation*}
	\mfh \ni \tilde{F}:=\left[F,\left(\begin{smallmatrix}
		0 & 0 & -\nu(\beta) \\
		\beta & 0 & 0 \\
		0 & 0 & 0
	\end{smallmatrix}\right)\right]
	=\left(\begin{smallmatrix}
		0 & 0 & \tilde{\alpha}(\nu(\beta)) \nu(\alpha)+\tilde{\alpha}(\nu(\alpha))\nu(\beta)-h(\alpha,\beta)\tilde{\nu}(\tilde{\alpha})\\
		-\alpha(\tilde{\nu}(\tilde{\alpha}))\beta-\beta(\tilde{\nu}(\tilde{\alpha}))\alpha+h(\alpha,\beta)\tilde{\alpha}& 0 & 0 \\
		0 & 0 & 0
	\end{smallmatrix}\right)
\end{equation*}
as well. However, $\tilde{F}$ is actually in $\mfh_w^0$. So if $h\neq 0$, then we may choose $\alpha,\beta\in \cU_0$ in such a way that $h(\alpha,\beta)\neq 0$ and then the condition that $\tilde{F}$ is in $\mfh_w^0$ yields $\tilde{\alpha}\in \cU_0$, and so $\tilde{\cU}_0\subseteq \cU_0$ in this case.

The inclusion $\cU_0\subseteq \tilde{\cU}_0$ under the assumption that $\tilde{h}\neq 0$ follows similarly.
\item
Considering the element $\tilde{F}$ as in the proof of part (b), choosing $\alpha=\beta$ and subtracting from $F$ the element $h(\alpha,\alpha)(\tilde{\alpha}\otimes w-\gamma\otimes \tilde{\nu}(\tilde{\alpha}))$, which is always in $\mfh$ even when $h=0$, we obtain that the element
$\alpha(\tilde{\nu}(\tilde{\alpha})\alpha\otimes w-\tilde{\alpha}(\nu(\alpha))\gamma\otimes \nu(\alpha)$
is in $\mfh_w$ and so $\alpha(\tilde{\nu}(\tilde{\alpha})=\tilde{\alpha}(\nu(\alpha))$
\end{enumerate}
\end{proof}	
We note that Lemma \ref{le:nutildenu} shows that $\cU_0=\tilde{\cU}_0$ if $h\neq 0$ and $\tilde{h}\neq 0$. Moreover, the same lemma shows that
the bilinear form
\begin{equation*}
H(\alpha,\beta):=\begin{cases} \alpha(\nu(\beta)) & \textrm{ if } \beta\in \cU_0,\\
	                            \alpha(\tilde{\nu}(\beta)) & \textrm{ if } \beta\in \tilde{\cU}_0,\\
	             \end{cases}        
\end{equation*} 
on $\cV_0:=\cU_0+\tilde{\cU}_0$ is well-defined and symmetric. We show now that we may extend $H$ to a pseudo-metric on $(\bR^n)^*$ with nice properties:
\begin{proposition}\label{pro:cK1S2Uwdegmetricsecond}
	Let $\mfh$ be an elliptic subalgebra with $\cK^{(1)}_{\mfh}=S^2 \cU\otimes w$ for some subspace $\cU$ of $(\bR^{n-1})^*$ and some $w\in \bR^{n-1}\setminus \{0\}$ such that $\mfh_w^0\neq \mfh_w$ and such that $\mfh_w^0\neq \{0\}$. Then there exists a pseudo-metric $g$ on $\bR^n$ such that $h|_{\cV_0\otimes \cV_0}=H$, such that $\alpha_0\perp_g \cV_0$, such that $\nu=(\cdot)^{\sharp_g}|_{\cU_0}$, $\tilde{\nu}=(\cdot)^{\sharp_g}|_{\cV_0}$ and such that $w$ and $v$ are null vectors with $g(v,w)=1$. Moreover,
\begin{equation*}
\mfh_g:=\spa{\alpha\otimes \tilde{\nu}(\tilde{\alpha})-\tilde{\alpha}\otimes \nu(\alpha)|\alpha\in \cU_0,\tilde{\alpha}\in \tilde{\cU}_0}\oplus \mfh_w\oplus \mfh_v^0
\end{equation*}
is a subalgebra of $\mfh$ which is degenerate metric with respect to $g$.
\end{proposition}
\begin{proof}
The proof of the existence of the pseudo-metric $g$ with the desired properties extending $H$ follows the same lines as the proof of the existence of the corresponding pseudo-metric in Proposition \ref{pro:cK1S2Uwdegmetricfirst} with $h$ being replaced by $H$. The construction then also gives that $\alpha_0=v^{\sharp_g}$ is orthogonal to $\cV_0$ and that $\mfg_g$ is a subspace of $\mathfrak{so}(g)$. Moreover, we immediately see that $\mfh_w=\mfh_w^0\oplus \spa{F_0}$ is a subalgebra of $\mfh$ with $\mfh_w^0$ being Abelian and that $\mfh_v^0$ is an Abelian subalgebra which is preserved by $[F_0,\cdot]$. Moreover,we computed in the proof of Lemma \ref{le:nutildenu} (b) that for $F_1:=\alpha\otimes w-\gamma\otimes \nu(\alpha)\in \mfh_w^0$, $F_2:=\tilde{\alpha}\otimes v-\alpha_0\otimes \tilde{\nu}(\tilde{\alpha})\in \mfh_v^0$, we have
\begin{equation*}
[F_1,F_2]=\alpha\otimes \tilde{\nu}(\tilde{\alpha})-\tilde{\alpha}\otimes \nu(\alpha)-\alpha(\tilde{\nu}(\tilde{\alpha}))\, F_0,
\end{equation*}
so that all elements in $\mfs:=\spa{\alpha\otimes \tilde{\nu}(\tilde{\alpha})-\tilde{\alpha}\otimes \nu(\alpha)|\alpha\in \cU_0,\tilde{\alpha}\in \tilde{\cU}_0}$ actually lie in $\mfh$. It is easy to see that $\mfs$ preserves both $\mfh_w$ and $\mfh_v^0$. Thus, we are left with computing the commutator of two elements in $\mfs$ and showing that it lies in $\mfh_g$. We will actually show that it lies again in $\mfs$ and for that compute
\begin{equation*}
\begin{split}
&\left[\alpha\otimes \tilde{\nu}(\tilde{\alpha})-\tilde{\alpha}\otimes \nu(\alpha),\beta\otimes \tilde{\nu}(\tilde{\beta})-\tilde{\beta}\otimes \nu(\beta)\right]\\
&=H(\alpha,\tilde{\beta})\, \beta\otimes \tilde{\nu}(\tilde{\alpha})-h(\alpha,\beta)\,\tilde{\beta}\otimes \tilde{\nu}(\tilde{\alpha})-\tilde{h}(\tilde{\alpha},\tilde{\beta})\, \beta\otimes \nu(\alpha)+H(\tilde{\alpha},\beta)\tilde{\beta}\otimes \nu(\alpha)\\
&-H(\tilde{\alpha},\beta)\, \alpha\otimes \tilde{\nu}(\tilde{\beta})+h(\alpha,\beta)\,\tilde{\alpha}\otimes \tilde{\nu}(\tilde{\beta})+\tilde{h}(\tilde{\alpha},\tilde{\beta})\, \alpha\otimes \nu(\beta)-H(\alpha,\tilde{\beta})\tilde{\alpha}\otimes \nu(\beta).
\end{split}
\end{equation*}
Now observe that $H(\alpha,\tilde{\beta}) (\beta\otimes \tilde{\nu}(\tilde{\alpha})-\tilde{\alpha}\otimes \nu(\beta))$ and
$H(\tilde{\alpha},\beta)(\tilde{\beta}\otimes \nu(\alpha)- \alpha\otimes \tilde{\nu}(\tilde{\beta}))$ both lie in $\mfs$.
Moreover, $h(\alpha,\beta)\,(\tilde{\alpha}\otimes \tilde{\nu}(\tilde{\beta})-\tilde{\beta}\otimes \tilde{\nu}(\tilde{\alpha}))$ either equals zero or $h\neq 0$ and then $\tilde{\alpha},\tilde{\beta}$ also lie in $\cU_0$ by Lemma \ref{le:nutildenu} (b). Thus also that summand is always in $\mfs$. A similar argumentation shows that also $\tilde{h}(\tilde{\alpha},\tilde{\beta})\, (\alpha\otimes \nu(\beta)-\beta\otimes \nu(\alpha))$
lies in $\mfs$, giving $[\mfs,\mfs]\subseteq \mfs$ and so finishing the proof.
\end{proof}
We provide an example that shows that the assertion of Proposition \ref{pro:cK1S2Uwdegmetricsecond} and also the one of Lemma \ref{le:appropriatev} does not hold in the case that $\mfh_w^0=\{0\}$:
\begin{example}\label{ex:counterextoprop}
Let $n=3$ and $\mfh=\spa{e^1\otimes e_2+e^3\otimes e_3}$. Then $\cK_{\mfh}^{(1)}=S^2 \spa{e^1}\otimes e_2$, i.e. $\cK_{\mfh}^{(1)}=S^2 \cU\otimes w$ with $\cU=\spa{e^1}$ and $w=e_2$. We have $\mfh_w=\mfh\neq \{0\}=\mfh_w^0$ and so also $\cU_0=0$. We note that the non-zero elements in $\mfh$ are not diagonalisable so Lemma \ref{le:appropriatev} does not hold here. Moreover, the non-zero elements in $\mfh$ have non-zero trace and so cannot be skew-symmetric with respect to some pseudo-metric on $\bR^3$. Hence, also Proposition \ref{pro:cK1S2Uwdegmetricsecond} does not hold in this case.
\end{example}
We are now finally in the position to compute $\cF_{\mfh}$:
\begin{theorem}\label{th:cK1=S2cUwsecond}
Let $\mfh$ be an elliptic subalgebra with $\cK^{(1)}_{\mfh}=S^2 \cU\otimes w$ for some subspace $\cU$ of $(\bR^{n-1})^*$ and some $w\in \bR^{n-1}\setminus \{0\}$ such that $\mfh_w^0\neq \mfh_w$.Then:
\begin{enumerate}[(a)]
\item
If $\mfh_w^0\neq \{0\}$, then
\begin{equation*}
\begin{split}
	\cF_{\mfh}&=\tilde{\mfk}_{\mfh}+\spa{\alpha\otimes \nu(\alpha)|\alpha\in \cU_0}+\alpha_0\otimes (U_0\cap \tilde{U}_0)\\
	&=(\tilde{\mfa}_0+\spa{\alpha\otimes \nu(\alpha)|\alpha\in \cU_0})\oplus \cU\otimes w\oplus \alpha_0\otimes (U_0\cap \tilde{U}_0)
\end{split}
\end{equation*}
for $U_0:=\nu(\cU_0)\subseteq \bR^{n-1}$, $\tilde{U}_0:=\tilde{\nu}(\tilde{\cU}_0)\subseteq \bR^{n-1}$. Moreover,
\begin{equation*}
\bigl((\cU_0\cap \tilde{\cU}_0)\oplus \spa{\alpha_0}\bigr)\otimes \bigl((U_0\cap \tilde{U}_0)\oplus \spa{w}\bigr)\subseteq \cF_{\mfh}.
\end{equation*}
\item
If $\mfh_w^0=\{0\}$ and $F_0(w)=0$ or $F_0(w)\neq 0$, respectively, for some $F_0\in \mfh_w\setminus \{0\}$, then $\cF_{\mfh}=\tilde{\mfk}_{\mfh}$ or
\begin{equation*}
\cF_{\mfh}=\left\{\left. \pi\circ F|_{\bR^{n-1}}\right|F(\ker(F_0))\subseteq \bR^{n-1}\right\}=\left\{\left. \pi\circ F|_{\bR^{n-1}}\right|F(\ker(F_0))\subseteq \ker(F_0),\ F(w)\subseteq \spa{v,w}\right\}
\end{equation*}
respectively, where $\pi:\bR^n\rightarrow \bR^{n-1}$ is the projection to $\bR^{n-1}$ along $\spa{v}$ for some $v\in \mathrm{im}(F_0)\setminus \{0\}$.
\end{enumerate}
\end{theorem}
\begin{proof}
\begin{enumerate}[(a)]
\item
Let $F\in \cF_{\mfh}$ be given and let $\nabla\in \cT^{-1}(\End(\bR^{n-1})$ be such that $F=\cT(\nabla)$. Then, set $F_1:=\nabla_v$ and observe that $F_1\in \mfh=\mfk_{\mfh}\oplus \mfh_v^0$. Hence, there exist $H_1\in \mfk_{\mfh}$, $G_1\in \mfh_v^0$ such that $F=H+G$. Consequently,
\begin{equation*}
F|_{\bR^{n-1}}=H|_{\bR^{n-1}}+\tilde{\alpha}\otimes v
\end{equation*}
for some $\tilde{\alpha}\in \tilde{\cU}_0$. Moreover, $\nabla|_{\bR^{n-1}\times \bR^{n-1}}\in S^2\cU\otimes w$ and using the decomposition $\cU=\cU_0\oplus \spa{\alpha_0}$, we may first write
\begin{equation*}
\nabla|_{\bR^{n-1}\times \bR^{n-1}}=\tilde{\nabla}+(\alpha\otimes \alpha_0+\alpha_0\otimes \alpha+\lambda\, \alpha_0\otimes \alpha_0)\otimes w
\end{equation*}
for $\tilde{\nabla}\in S^2 \cU_0\otimes w$, $\alpha\in \cU_0$ and $\lambda\in \bR$ and then Sylvester's law of inertia applied to $\tilde{\nabla}$ to get a basis $\alpha_1,\ldots,\alpha_m$ of $\cU_0$ and $\lambda_1,\ldots,\lambda_m\in \{-1,1,0\}$ to arrive at
\begin{equation*}
	\nabla|_{\bR^{n-1}\times \bR^{n-1}}=\left(\sum_{i=1}^m \lambda_i\, \alpha_i\otimes \alpha_i+(\alpha\otimes \alpha_0+\alpha_0\otimes \alpha+\lambda\, \alpha_0\otimes \alpha_0\right)\otimes w.
\end{equation*}
Thus,
\begin{equation*}
\begin{split}
	\nabla|_{\bR^{n-1}\times \bR^n}&=\sum_{i=1}^m \lambda_i\, \alpha_i\otimes (\alpha_i\otimes w-\gamma\otimes \nu(\alpha_i))+\alpha\otimes (\alpha_0\otimes w-\gamma\otimes v)\\
	&+\alpha_0\otimes (\alpha\otimes w-\gamma\otimes \nu(\alpha))+\lambda\, \alpha_0\otimes (\alpha_0\otimes w-\gamma\otimes v),
\end{split}
\end{equation*}
so that
\begin{equation*}
F_2:=-\nabla v|_{\bR^{n-1}}=\sum_{i=1}^m\lambda_i \alpha_i\otimes \nu(\alpha_i)+\alpha\otimes v+\alpha_0\otimes \nu(\alpha)+\lambda\, \alpha_0\otimes v
\end{equation*}
Thus,
\begin{equation*}
F=\cT(\nabla)=F_1|_{\bR^{n-1}}+F_2=H_1|_{\bR^{n-1}}+\sum_{i=1}^m\lambda_i \alpha_i\otimes \nu(\alpha_i)+\alpha_0\otimes \nu(\alpha)+(\tilde{\alpha}+\alpha+\lambda\, \alpha_0)\otimes v,
\end{equation*}
which forces, due to $F\in \End(\bR^{n-1})$, that $\tilde{\alpha}=-\alpha$, i.e. $\alpha\in \cU_0\cap \tilde{\cU}_0$, and $\lambda\in \bR$. Hence,
\begin{equation*}
	F=H_1|_{\bR^{n-1}}+\sum_{i=1}^m\lambda_i \alpha_i\otimes \nu(\alpha_i)+\alpha_0\otimes \nu(\alpha)\in \tilde{\mfk}_{\mfh}+\spa{\alpha\otimes \nu(\alpha)|\alpha\in \cU_0}+\alpha_0\otimes(\cU_0\cap \tilde{\cU}_0).
\end{equation*}
proving
\begin{equation*}
	\cF_{\mfh}\subseteq\tilde{\mfk}_{\mfh}+\spa{\alpha\otimes \nu(\alpha)|\alpha\in \cU_0}+\alpha_0\otimes (\cU_0\cap \tilde{\cU}_0).
\end{equation*}
Reversing the arguments, we may construct, for a given endomorphism $F$ in the right-hand side of the equation above an element $\nabla\in \cD_{\mfh}$ with $\cT(\nabla)=F$, which gives the converse inclusion and proves the statement.
\item
Let $F\in \cF_{\mfh}$ be given and let $\nabla\in \cT^{-1}(\End(\bR^{n-1})$ be such that $F=\cT(\nabla)$. Moreover, choose $v\in \im(F_0)\setminus \bR^{n-1}$. By appropriately scaling $F_0$, we may assume that $F_0(v)=v+\mu w$ for some $\mu\in \bR$. Thus,
\begin{equation*}
F_0=\alpha_0\otimes w+\gamma\otimes v+\mu\, \gamma\otimes w
\end{equation*}
Moreover, $\cU=\spa{\alpha_0}$ and so
\begin{equation*}
\nabla|_{\bR^{n-1}\otimes \bR^{n-1}}=\lambda\, \alpha_0\otimes \alpha_0\otimes w
\end{equation*}
for some $\lambda\in \bR$, yielding
\begin{equation*}
\nabla|_{\bR^{n-1}\otimes \bR^n}=\lambda\, \alpha_0\otimes (\alpha_0\otimes w+\gamma\otimes v-\mu\, \gamma\otimes w),
\end{equation*}
which implies
\begin{equation*}
G:=\nabla v |_{\bR^{n-1}}=\lambda\, \alpha_0\otimes v-\lambda \mu\, \alpha_0\otimes w.
\end{equation*}
We set
\begin{equation*}
H:=\nabla_v\in \mfh
\end{equation*}
and remark that
\begin{equation*}
F=\cT(\nabla)=H|_{\bR^{n-1}}-G.
\end{equation*}
As $G$ maps the subspace $\ker(F_0)$ of $\bR^{n-1}$ to zero, we must have $H(\ker(F_0))\subseteq \bR^{n-1}$.  We choose some $u_0\in \bR^{n-1}$ such that $\alpha_0(u_0)=1$ and will do this choice in the case that $w\notin \ker(F_0)$ so that $u_0=\tau w$ for some $\tau\in \bR^*$.

Then $\bR^{n-1}=\ker(F_0)\oplus \spa{u_0}$ and there is $\beta_0\in \ker(F_0)^*$ such that
\begin{equation*}
H|_{\ker(F_0)}=H_0+\beta_0\otimes u_0
\end{equation*}
for some $H_0\in \End(\ker(F_0))$. We identify $\beta_0$ with an element in the annihilator of $u_0$, and so with an element in $(\bR^{n-1})^*$. As such, we see that $\beta_0$ and $\alpha_0$ are necessarily linearly independent as long as $\beta_0\neq 0$. We compute now
\begin{equation*}
[F_0,H]|_{\ker(F_0)}=F_0\circ H|_{\ker(F_0)}=\beta_0\otimes F_0(u_0)=\beta_0\otimes w.
\end{equation*}
Writing $H(u_0)=\tilde{u}_0+a u_0+b v$ for $\tilde{u}_0\in \ker(F_0)$, $a,b\in \bR$, we obtain
\begin{equation*}
[F_0,H](u_0)=F_0(H(u_0))-H(w) =a w+ b v+\mu b w- H(w),
\end{equation*}
so that
\begin{equation*}
[F_0,H]|_{\bR^{n-1}}=\beta_0\otimes w+\alpha_0\otimes ((a+\mu b)w+b v+H(w))
\end{equation*}
Thus, the element
\begin{equation*}
\alpha_0\otimes [F_0,H]|_{\bR^{n-1}}+\beta_0\otimes F_0|_{\bR^{n-1}}=(\alpha_0\otimes \beta_0+\beta_0\otimes \alpha_0)\otimes w+ \alpha_0\otimes \alpha_0\otimes ((a-+\mu b)w+ b v-H(w))
\end{equation*}
is in $\cK_{\mfh}^{(1)}=\spa{\alpha_0\otimes \alpha_0\otimes w}$, which forces $\beta_0=0$. 

If now $F_0(w)=0$, then $H(w)\in \bR^{n-1}$ and we must have $b=0$ and so $H(u_0)\in \bR^{n-1}$.  But then $H(\bR^{n-1})\subseteq \bR^{n-1}$, which forces $G(\bR^{n-1})\subseteq \bR^{n-1}$ and so $G=0$. Thus, $F=H|_{\bR^{n-1}}\in \tilde{\mfk}_{\mfh}$ and so $\cF_{\mfh}=\tilde{\mfk}_{\mfh}$ in this case.

 If $F_0(w)\neq 0$, then $w=\lambda u_0$ and so $H(w)=\frac{1}{\tau} \tilde{u}_0+ a w+ \tfrac{b}{\tau} v$, implying
\begin{equation*}
(a+\mu b)w+b v-H(w)=\frac{1}{\tau} \tilde{u}_0+\mu b w+\tfrac{b}{\tau} (\tau-1) v,
\end{equation*}
which forces $\tau=1$ and $\tilde{u}_0=0$. Hence, $\alpha_0(w)=1$ and
\begin{equation*}
H\in \left\{\left. \tilde{F}\in \mfh \right|\tilde{F}(\ker(F_0))\subseteq \ker(F_0),\ \tilde{F}(w)\subseteq \spa{v,w}\right\}.
\end{equation*}
But then
\begin{equation*}
\bR^{n-1}\ni H(w)-G(w)=a w+ bv-\lambda v-\lambda \mu w
\end{equation*}
forces $b=\lambda$ and so we get
\begin{equation*}
F=H|_{\bR^{n-1}}-G=\pi|_{\bR^{n-1}}\circ H|_{\bR^{n-1}}+\lambda \mu \alpha_0\otimes w=\pi|_{\bR^{n-1}}\circ (H+\lambda \mu F_0)|_{\bR^{n-1}},
\end{equation*}
giving
\begin{equation*}
\cF_{\mfh}\subseteq \left\{\left. \pi\circ F|_{\bR^{n-1}}\right|F(\ker(F_0))\subseteq \bR^{n-1}\right\}=\left\{\left. \pi\circ F|_{\bR^{n-1}}\right|F(\ker(F_0))\subseteq \ker(F_0),\ F(w)\subseteq \spa{v,w}\right\}
\end{equation*}
since
\begin{equation*}
H+\lambda \mu F_0\in \left\{\left. \tilde{F}\in \mfh \right|\tilde{F}(\ker(F_0))\subseteq \ker(F_0),\ \tilde{F}(w)\subseteq \spa{v,w}\right\}.
\end{equation*}
The other inclusion follows by inverting all steps.
\end{enumerate}
\end{proof}
We illustrate the content of Theorem \ref{th:cK1=S2cUwsecond} (a) and some of the previous results in this subsection by an explicit example
and also give an example for Theorem \ref{th:cK1=S2cUwsecond} (b) with $F_0(w)\neq 0$ which shows that in this case we do not necessarily have $\cF_{\mfh}=\tilde{\mfk}_{\mfh}$:
\begin{example}\label{ex:notdegmetric}
\begin{itemize}
\item
Let $n=8$ and consider
\begin{equation*}
\begin{split}
\mfh:=\mathrm{span}\left(e^1\otimes e_5-e^2\otimes e_4,e^1\otimes e_6-e^3\otimes e_4, e^2\otimes e_6-e^3\otimes e_5, e^1\otimes e_7-e^8\otimes e_4,\right.\\ \left. e^2\otimes e_7-e^8\otimes e_5, e^1\otimes e_8-e^7\otimes e_4, e^3\otimes e_8-e^7\otimes e_6, e^7\otimes e_7-e^8\otimes e_8, I_8\right)
\end{split}
\end{equation*}
One checks that $\mfh$ is closed under the commutator, i.e. forms, indeed, a linear subalgebra of $\End(\bR^8)$. We note that $\mfh$ cannot be metric with respect to some pseudo-metric on $\bR^8$ since $I_8\in \mfh$ has non-zero trace.

Moreover, one checks that
\begin{equation*}
\cK_{\mfh}^{(1)}=S^2 \spa{e^1,e^2,e^4}\otimes e_7,
\end{equation*}
i.e. $\cK_{\mfh}^{(1)}=S^2\cU\otimes w$ for $\cU=\spa{e^1,e^2,e^7}$ and $w=e_7\in \bR^7$. Then $\cU_0=\spa{e^1,e^2}$ and
\begin{equation*}
\mfh_w^0=\spa{e^1\otimes e_7-e^8\otimes e_4,e^2\otimes e_7-e^8\otimes e_5},\ \mfh_w=\mfh_0^w\oplus \spa{F_0}
\end{equation*}
for $F_0:=e^7\otimes e_7-e^8\otimes e_8$. Thus, $\alpha_0=e^7$ and $v:=e_8$ makes $F_0$ look like in Lemma \ref{le:appropriatev}. Furthermore, the map $\nu$ is explicitly given by
\begin{equation*}
\nu: \cU_0=\spa{e^1,e^2}\rightarrow \bR^7,\quad \nu(e^1)=e_5,\quad \nu(e^2)=e_5
\end{equation*}
and $U_0:=\spa{e_4,e_5}$.
Furthemore,
\begin{equation*}
\mfh_v^0=\spa{e^1\otimes e_8-e^7\otimes e_4, e^3\otimes e_8-e^7\otimes e_6},
\end{equation*}
so that $\tilde{\cU}_0=\spa{e^1,e^3}$ and the map $\tilde{\nu}$ is explicitly given by
\begin{equation*}
\tilde{\nu}:\spa{e^1,e^3}\rightarrow \bR^7,\quad \tilde{\nu}(e^1)=e_5,\quad \tilde{\nu}(e^3)=e_6,
\end{equation*}
and so $\tilde{U}_0:=\spa{e_4,e_6}$. We observe that, indeed $\nu$ equals $\tilde{\nu}$ on $\cU_0\cap \tilde{\cU}_0=\spa{e^1}$ and that $U_0\cap \tilde{U}_0=\spa{e_4}$. However, neither $\cU_0\subseteq \tilde{\cU}_0$ nor $\tilde{\cU}_0\subseteq \cU_0$, reflecting the fact that $h=0$ and $\tilde{h}=0$, and even, $H=0$.

Moreover,
\begin{equation*}
\begin{split}
\mfa_0&=\spa{(e^1\otimes e_5-e^2\otimes e_4,e^1\otimes e_6-e^3\otimes e_4, e^2\otimes e_6-e^3\otimes e_5,I_8},\\
\tilde{\mfa}_0& =\spa{e^1\otimes e_5-e^2\otimes e_4,e^1\otimes e_6-e^3\otimes e_4, e^2\otimes e_6-e^3\otimes e_5,I_7}\subseteq \End(\bR^7).
\end{split}
\end{equation*}
Hence, Theorem \ref{th:cK1=S2cUwsecond} yields
\begin{equation*}
\begin{split}
\cF_{\mfh}&=\mathrm{span}\left(e^1\otimes e_5-e^2\otimes e_4,e^1\otimes e_6-e^3\otimes e_4, e^2\otimes e_6-e^3\otimes e_5,I_7,e^1\otimes e_4,e^2\otimes e_5,\right.\\
 & \qquad\quad\ \  \left. e^1\otimes e_7,e^2\otimes e_7,e^7\otimes e_7,e^7\otimes e_4\right)\\
&=\left\{\left.\left(\begin{smallmatrix} a_1 & 0 & 0 & 0 & 0 & 0 & 0 \\
                                   0 & a_1 & 0 & 0 & 0 & 0 & 0 \\
                                   0 & 0 & a_1 & 0 & 0 & 0 & 0 \\
                                   a_2 & -a_3 & -a_5 & a_1 & 0 & 0 & 0 \\
                                   a_3 & a_4 & -a_6 & 0& a_1 & 0 & 0 \\
                                   a_5 & a_6 & 0  & 0& 0 & a_1 & 0 \\
                                    a_7 & a_8 & 0  & a_9 & 0 & 0 & a_1+a_{10}
\end{smallmatrix}\right)\right|a_1,\ldots,a_{10}\in \bR\right\},
\end{split}
\end{equation*}
and we see that, indeed, $((\cU_0\cap \tilde{\cU}_0)\oplus \spa{\alpha_0})\otimes ((U_0\cap \tilde{U}_0)\oplus \spa{w})=\spa{e^1,e^7}\otimes \spa{e_4,e_7}\subseteq \cF_{\mfh}$.
\item
Let $n=4$ and $\mfh=\spa{F_1,F_2}$ with
\begin{equation*}
F_1:=e^1\otimes e_2-e^2\otimes e_1+e^3\otimes e_4,\quad F_2:=e^3\otimes e_3+e^4\otimes e_4.
\end{equation*}
Then $\mfh$ is an Abelian subalgebra with $\cK_{\mfh}^{(1)}=\spa{e^3\otimes e^3\otimes e_3}$, i.e. $\cK_{\mfh}^{(1)}=S^2\cU\otimes \spa{w}$ with $\cU=\spa{e^3}$, $w=e_3$. Then
$\mfh_w^0=\{0\}$, $\mfh_w=\spa{F_2}$ and $F_2(w)=F_2(e_3)=e_3$. Hence, we are in the situation of the second case in Theorem \ref{th:cK1=S2cUwsecond} (b). Thus,
\begin{equation*}
\begin{split}
\tilde{\mfk}_{\mfh}&=\{0\}\neq \{e^1\otimes e_2-e^2\otimes e_1\}=\left\{\left. \pi\circ F|_{\bR^{n-1}}\right|F(\ker(F_0))\subseteq \ker(F_0),\ F(w)\subseteq \spa{v,w}\right\}
\\
&=\cF_{\mfh}.
\end{split}
\end{equation*}
\end{itemize}
\end{example}
Finally, we note what Theorem \ref{th:cK1=S2cUwsecond} yields us in the case of a degenerate metric linear subalgebra with $\mfh_w\neq \mfh_w^0$:
\begin{corollary}\label{co:degmetricsecond}
Let $\mfh$ be a degenerate metric subalgebra such that $\mfh_w\neq \mfh_w^0$. Then
\begin{equation*}
\begin{split}
	\cF_{\mfh}&=\tilde{\mfk}_{\mfh}\oplus \spa{\alpha\otimes \alpha^{\sharp}|\alpha\in \cU_0}\oplus v^{b}\otimes (U_0\cap \tilde{U}_0)\\
	&=\tilde{\mfk}_{\mfh}\oplus \spa{u^{b}\otimes u|u\in U_0}\oplus v^{b}\otimes (U_0\cap \tilde{U}_0)\\
	&=\tilde{\mfa}_0\oplus \spa{u^{b}\otimes u|u\in U_0}\oplus \cU_0\otimes w\oplus v^{b}\otimes (U_0\cap \tilde{U}_0)\oplus \spa{v^b\otimes w}
\end{split}
\end{equation*}
for $U_0:=\{F(v)|F\in \mfh_w^0\}$, $\tilde{U}_0:=\{F(w)|F\in \mfh_v^0\}$,
$\cU_0:=(U_0)^{b}$ and 
\begin{equation*}
\tilde{\mfa}_0:=\left\{F|_{\bR^{n-1}}\left|F\in \mfh,\ F(v)=0,\ F(w)=0, \, \ F(\spa{v,w}^{\perp})\subseteq  \spa{v,w}^{\perp}\right.\right\}
\end{equation*}
with $v\in \bR^n$ chosen so that $v$ is null, not in $\bR^{n-1}$ and $g(v,w)=1$.
\end{corollary}

\end{document}